\documentclass[twoside,11pt]{article}

\usepackage{amsmath, mathtools}
\usepackage{amssymb}
\usepackage{microtype}
\usepackage{amsthm}
\usepackage{natbib}
\usepackage[linesnumbered,ruled,vlined]{algorithm2e}
\usepackage{array}
\usepackage{fullpage}
\usepackage[colorlinks=true,
            linkcolor=blue,
            urlcolor=blue,
            citecolor=blue]{hyperref}

\theoremstyle{plain}
\newtheorem{theorem}{Theorem}[section]
\newtheorem{lemma}[theorem]{Lemma}
\newtheorem{corollary}[theorem]{Corollary}
\newtheorem{proposition}[theorem]{Proposition}
\newtheorem{remark}[theorem]{Remark}
\newtheorem{definition}[theorem]{Definition}
\numberwithin{equation}{section}

\SetKwInput{KwInput}{Input}              
\SetKwInput{KwOutput}{Output}              
\SetKwRepeat{KwRepeat}{repeat}{until}

\newcommand{\EE}{\mathbb{E}}
\newcommand{\II}{\mathbb{I}}
\newcommand{\LL}{\mathbb{L}}
\newcommand{\NN}{\mathbb{N}}
\newcommand{\PP}{\mathbb{P}}
\newcommand{\RR}{\mathbb{R}}

\newcommand{\cC}{\mathcal{C}}
\newcommand{\cE}{\mathcal{E}}
\newcommand{\cF}{\mathcal{F}}
\newcommand{\cG}{\mathcal{G}}
\newcommand{\cI}{\mathcal{I}}
\newcommand{\cJ}{\mathcal{J}}
\newcommand{\cM}{\mathcal{M}}
\newcommand{\cN}{\mathcal{N}}
\newcommand{\cO}{\mathcal{O}}
\newcommand{\cP}{\mathcal{P}}
\newcommand{\cS}{\mathcal{S}}

\newcommand{\cMloc}{\mathcal{M}^{\operatorname{loc}}}
\newcommand{\cMKloc}{\mathcal{M}_{K}^{\operatorname{loc}}}

\newcommand{\epsLSE}[1][K]{\varepsilon_{#1,\mathrm{LS}}}
\newcommand{\epsSupwidth}[1][K]{\varepsilon_{#1,w}}
\newcommand{\epswidth}[1][\mu]{\varepsilon_{#1,w}}

\newcommand{\argmax}{\mathop{\mathrm{argmax}}}
\newcommand{\argmin}{\mathop{\mathrm{argmin}}}

\newcommand{\diam}{\mathrm{diam}}

\newcommand{\loc}{\operatorname{loc}}
\newcommand{\sign}{\mathop{\mathrm{sign}}}
\def\T{{ ^\mathrm{\scriptscriptstyle T} }}

\allowdisplaybreaks

\begin{document}

\begin{center}
{\LARGE Some facts about the optimality of the LSE in the Gaussian sequence model with convex constraint}

{\large
\begin{center}
Akshay Prasadan and Matey Neykov
\end{center}}

{Department of Statistics \& Data Science, Carnegie Mellon University\\
Department of Statistics and Data Science, Northwestern University\\
[2ex]\texttt{aprasada@andrew.cmu.edu}, ~~~ \texttt{mneykov@northwestern.edu}}
\end{center}

\begin{abstract} 
We consider a convex constrained Gaussian sequence model and characterize necessary and sufficient conditions for the least squares estimator (LSE) to be minimax optimal. For a closed convex set $K\subset \RR^n$ we observe $Y=\mu+\xi$ for $\xi\sim \cN(0,\sigma^2\II_n)$ and $\mu\in K$ and aim to estimate $\mu$. We characterize the worst case risk of the LSE in multiple ways by analyzing the behavior of the local Gaussian width on $K$. We demonstrate that optimality is equivalent to a Lipschitz property of the local Gaussian width mapping. We also provide theoretical algorithms that search for the worst case risk. We then provide  examples showing optimality or suboptimality of the LSE on various sets, including $\ell_p$ balls for $p\in[1,2]$, pyramids, solids of revolution, and multivariate isotonic regression, among others.
\end{abstract}

\tableofcontents

\section{Introduction}
\label{section:LSE:introduction}

In this paper we focus on the Gaussian sequence model problem. Specifically, we observe a single observation $Y = \mu + \xi$, where $\xi \sim \cN(0, \sigma^2\II_n)$ is a multivariate Gaussian noise and $\mu \in K$ where $K \subset \RR^n$ is a known closed convex constraint. Our goal is to estimate the vector $\mu$ taking into account the convex constraint, with the hope that our estimator will be optimal and computationally tractable. A popular estimator in this setting is the Least Squares Estimator (LSE) which essentially projects the observation $Y$ onto the convex constraint. In detail, the LSE is given by
\begin{align}\label{LSE:def}
    \hat \mu= \hat \mu (Y) := \argmin_{\nu\in K}\|Y - \nu\|^2,
\end{align}
where we abbreviated the Euclidean norm by $\|\cdot\|$. By definition, we have that $\hat \mu$ is the Euclidean projection of $Y$ onto the set $K$ which is sometimes denoted by $\Pi_K Y$. The LSE is perhaps one of the most intuitive estimators for this problem (it is also the maximum likelihood estimator); in addition, the LSE can be solved for a variety of sets $K$ as Euclidean projection is a well studied convex problem and there exist plenty of methods which can calculate $\hat \mu$. Unfortunately, it is known that there exist sets $K$ for which the LSE is far from optimal in the worst case. Our goal in this paper is to give some insights into when is $\hat \mu$ an ``optimal'' estimator. Here we will measure optimality with respect to the expected squared $\ell_2$ loss. Taking a worst case perspective we would like to compare $\sup_{\mu \in K}\EE_\mu \|\hat \mu - \mu\|^2$ to the minimax optimal rate, i.e., up to constants the expression $\inf_{\hat \nu} \sup_{\mu \in K} \EE_\mu \|\hat \nu(Y) - \mu\|^2$, with the infimum taken with respect to all measurable functions (i.e., estimators) of the data. Our primary technique will be analyzing the local geometry of $K$ using a localized version of the Gaussian width as well as the metric entropy. We will give numerous necessary and sufficient conditions for LSE optimality, and then apply these results to a wide variety of examples and non-examples of LSE optimality.

\subsection{Related Literature}

Recently, the minimax rate in this problem was characterized in terms of the local geometry of the set $K$ \citep{neykov2022minimax}. On the other hand, \citet{chatterjee2014new} characterized the expression $\EE \|\hat \mu - \mu\|^2$ for any point $\mu \in K$. It may appear that the problem we posed above is nearly solved. The only thing one needs to do is to somehow add a $\sup_{\mu \in K}$ in front of \citet{chatterjee2014new}'s variational formula (see \eqref{equation:varepsilon:mu} below) and compare that to the minimax rate. However, this turns out to be more challenging than one may anticipate, as we hope to convince the reader in Section \ref{main:results:sec}. 

While \citet{chatterjee2014new} does not establish LSE optimality conditions in full generality, the work does demonstrate that the LSE satisfies an admissibility condition (see \citet[Theorem 1.4]{chatterjee2014new}, \citet{chen_guntuboyina_admissability}, \citet{kur2023admissability}). That is, up to some universal constant, for any arbitrary estimator $\tilde\mu$, there exists a $\mu\in K$ for which the LSE has a lower risk than the arbitrary estimator up to that universal constant. This means the LSE is preferable for some region of the parameter space. Our paper on the other hand focuses on the worst case risk of the LSE.

The optimality of the LSE is well known for a wide class of examples. For isotonic regression, \citet{zhang2002risk}, \citet{bellec2018sharp} demonstrate upper bounds on the LSE risk that match the minimax lower bound from \citet[Corollary 5]{bellec2015sharp}. In multivariate isotonic regression, \citet{deng_and_zhang_2020} propose a block estimator that outperforms the LSE in certain cases. \citet{wei2020gauss} establishes LSE optimality for ellipsoid estimation problems under regularity assumptions that hold for Sobolev ellipsoids with smoothness parameter $\alpha>1/2$. Suboptimal examples can be found in \citet{zhang2013nearly, chatterjee2014new}.


An example of our set-up is non-parametric regression with fixed design, i.e., \[K=\{(f(x_1), \ldots, f(x_n))\colon f\in\cF\},\] where the $x_i$ are fixed points and $\cF$ is a convex class of functions. \citet{kur_gil_rakhlin} give suboptimality results for the LSE in non-Donsker regimes, where the logarithm of the metric entropy has rate $\delta^{-\alpha}$ for $\alpha>2$. \citet{qiyang_han_2021_donsker} further investigates the non-Donsker setting, and under a stronger assumption on the entropy of the function class, aims to close the gap between the upper and lower bounds on the LSE risk that derives from an entropy integral condition \citep{birge1993rates}. \citet{kur2020convex} establishes suboptimality when the domain of functions in $\cF$ is a polytope of dimension at least 5, with applications to convex function classes with Lipschitz or boundedness constraints. \citet{kur2023admissability}, who also give a generalization of the admissibility result from \citet{chatterjee2014new}, demonstrates that suboptimality emerges from the bias portion of the risk following a bias-variance decomposition.

\subsection{Notation and Definitions}\label{main:section}

We define $[n]:=\{1,\dots,n\}$ for each $n\in\NN$. We write $a\lesssim b$ if for some absolute constant $C>0$ we have $a\le Cb$, and similarly define $\gtrsim$. We write $a\asymp b$ if $a\lesssim b$ and $b\lesssim a$ both hold, for possibly different constants. Unless otherwise specified, we operate on sets $K\subset \RR^n$ and use $\|\cdot\|_2$ to denote the Euclidean $\ell_2$-norm on $\RR^n$, dropping the subscript when clear. For any $\theta\in\RR^n$ and $\varepsilon>0$, we define $B(\theta,\varepsilon) =\{\theta'\in\RR^n:\|\theta-\theta'\|_2\le\varepsilon\}$. We define $\RR^+=(0,\infty)$. We write the $n\times n$ identity matrix as $\II_n$. We denote by $d$ the diameter of $K$, i.e., $d:= \mathrm{diam}(K) := \sup\{\|\theta-\theta'\|_2 :\theta,\theta'\in K\}$. We will define several quantities related to the Gaussian width and least squares estimator, usually denoted by adding subscripts to $\varepsilon$. For convenience, we summarize the notation in Table \ref{tab:1}.

\begin{definition}[Packing Sets and Global Entropy] An $\varepsilon$-packing of a (totally) bounded set $T \subset \RR^n$ with respect to $\|\cdot\|_2$ is a set $\{\theta_1,\ldots, \theta_M \} \subset T$ such that $\|\theta_i - \theta_j\|_2 > \varepsilon$ for all $i\ne j$. The $\varepsilon$-packing number $\cM(\varepsilon, T)$ is the largest possible cardinality of an $\varepsilon$-packing of $T$. The global entropy of $T$ at distance $\varepsilon$ is simply the number $\log \cM(\varepsilon, T)$.
\end{definition}

\begin{definition}[Local Entropy]\label{local:entropy:def} Let $c^* \in \RR^+$ be a sufficiently large absolute constant.\footnotemark{} Define 
\begin{align*}
    \cMKloc(\varepsilon) = \sup_{\theta \in K} \cM(\varepsilon/c^{\ast}, B(\theta, \varepsilon) \cap K),
\end{align*} i.e., the largest $(\varepsilon/c^{\ast})$-packing of a set of the form $B(\theta, \varepsilon) \cap K$.
 We refer to $\log \cMKloc(\varepsilon)$ as the local entropy of $K$. We drop the subscript $K$ from $\cMKloc$ if the context is clear.
\end{definition}

\begin{definition}[(Local) Gaussian Width]\label{gaussian:def} The Gaussian width of a set $T\subset\RR^n$ is defined by $w(T) = \EE\left[\sup_{t\in T}\langle x,t\rangle\right]$ where $x \sim \cN(0,\II_n)$. Let $\mu \in \RR^n$ and consider the set $B(\mu,\varepsilon)\cap K$. Then the local Gaussian width at $\mu$ is the function defined as $w_{K,\mu}(\varepsilon) = w(B(\mu,\varepsilon)\cap K)$. When the set $K$ is clear from the context, we will denote $w_{K,\mu}(\varepsilon)$ with $w_\mu(\varepsilon)$. 
\end{definition} \footnotetext{We also reserve $c^*$ to refer to this constant throughout this paper.}

\citet{neykov2022minimax} characterized up to absolute constant factors the minimax rate. It turns out to be $\inf_{\hat \nu} \sup_{\mu \in K} \EE_\mu \|\hat \nu(Y) - \mu\|_2^2 \asymp \varepsilon^{*2} \wedge d^2$, where
\begin{align}\label{varepsilon:star:def}
\varepsilon^*:=\sup \{\varepsilon: \varepsilon^2/\sigma^2 \leq \log \cMKloc(\varepsilon) \}.
\end{align}
In fact, we can show $\varepsilon^{\ast}\lesssim d$. To see this, take $\varepsilon \ge (c^{\ast}+\kappa)d$ for any $\kappa>0$. Then $\log \cMKloc(\varepsilon) = 0$, since for any $\mu\in K$, if we pack $K\cap B(\mu,(c^{\ast}+\kappa)d)=K$ using points of distance at least $\frac{c^{\ast}+\kappa}{c^{\ast}}d>d$ apart, there can only be one point in the packing. This implies by \eqref{varepsilon:star:def} that $\varepsilon^{\ast}\le\varepsilon$ for all $\varepsilon \ge (c^{\ast}+\kappa)d$, implying $\varepsilon^{\ast}\lesssim (c^{\ast}+\kappa)d$. Hence $\varepsilon^{\ast}\wedge d \asymp \varepsilon^{\ast}$.

We will now give a quick universal lower bound on the minimax rate which will be useful later on. We give a proof using the definition of $\varepsilon^{\ast}$ but it can also be deduced from a two point lower bound argument. In addition, we show a different lower bound on the rate  involving the Gaussian width of the set $K$ in the appendix (see Lemma \ref{generic:bound:minimax:rate}). 

 \begin{lemma}[Minimax Rate Bound]\label{minimax:bound:versus:sigma:and:d} The minimax rate $\varepsilon^{\ast}$ satisfies $\varepsilon^{\ast} \gtrsim \sigma\wedge d$.
 \end{lemma}

Another useful property is invariance of the minimax rate to the constants used inside or outside the local metric entropy term, as the following lemma shows. The proof relies on the fact that $\varepsilon\mapsto \cMloc_K(\varepsilon)$ is non-increasing for convex sets $K$ (or more generally, star-shaped ones), as established in \citet[Lemma II.8]{neykov2022minimax}. 

\begin{lemma}[Equivalent Forms of Information Theoretic Lower Bound] \label{lemma:equivalent:information:lower:bound} Define $\varepsilon^{\dag} = \sup\{\varepsilon>0:\varepsilon^{2}/\sigma^2 \leq C_1\log \cMKloc(C_2\varepsilon)\}$ for any fixed $C_1, C_2>0$. Then we have $\varepsilon^{\ast} \asymp \varepsilon^\dag$, where $\varepsilon^{
\ast}=\sup \{\varepsilon: \varepsilon^2/\sigma^2 \leq \log \cMKloc(\varepsilon)\}$.
\end{lemma}

\subsection{Organization}

In Section \ref{main:results:sec}, we begin by introducing a variational quantity from \citet{chatterjee2014new} related to the least squares error. We use the relationship between the two to then derive bounds on the worst case LSE error. We then characterize a sufficient condition for the LSE to be minimax optimal. Then we derive numerous other variational formulas that let us bound the worst case LSE rate.

Section \ref{section:LSE:examples} is split into two parts. First, we illustrate numerous examples where the LSE is minimax optimal, or nearly so. This includes isotonic regression in both univariate and multivariate settings, hyper-rectangles, subspaces, and $\ell_p$ balls for $p\in\{1,2\}$. We then show examples where the LSE is suboptimal, including pyramids, multivariate isotonic regression when the noise is too large, solids of revolution, ellipsoids, and $\ell_p$ balls for $p\in(1,2)$.

In Appendix \ref{section:algorithm_proofs}, we develop two theoretical algorithms that aim to find this worst case LSE rate provided the set $K$ is bounded, using some of the variational quantities we defined. The remaining appendices contain the proofs of our main results and examples.

\begin{table*}[t]
    \renewcommand{\arraystretch}{1.5}
      \caption{Summary of common notation.}
  \centering
  \begin{tabular}{m{3cm}|m{6cm}|m{4cm}}
  \hline
    Symbol & Definition & Meaning / Usage \\\hline 
    $\epsLSE(\sigma)^2$ & $\sup_{\mu \in K} \EE \|\hat \mu - \mu\|^2$ & Worst-case LSE Rate \\ \hline 
    $w_{\mu}(\varepsilon)=w_{K,\mu}(\varepsilon)$ & $w(B(\mu,\varepsilon)\cap K)$ & Local Gaussian Width \\ \hline
    $\epswidth(\sigma)$ & $\argmax_{\varepsilon} [\sigma w_{\mu}(\varepsilon) - \tfrac{\varepsilon^2}{2}]$ & Controls $\epsLSE$ \\\hline
    $\epsSupwidth =\epsSupwidth(\sigma)$ & $\sup_{\mu\in K}\epswidth(\sigma)$ & Controls $\epsLSE$ \\ \hline
    $\cMKloc(\varepsilon)$ &  $\sup_{\theta \in K} \cM(\varepsilon/c^{\ast}, B(\theta, \varepsilon) \cap K)$ & Local Metric Entropy \\\hline 
    $\varepsilon^{\ast}$ &  $\sup \{\varepsilon: \varepsilon^2/\sigma^2 \leq \log \cMKloc(\varepsilon) \}$ & $\sqrt{\text{Minimax Rate}}$ \\\hline 
  \end{tabular}
  \vspace*{5pt}
  \label{tab:1}
\end{table*}

\section{Main Results}\label{main:results:sec}


According to \citet[Theorem 1.1]{chatterjee2014new} the maximizer of 
\begin{align} \label{equation:varepsilon:mu}
    \epswidth(\sigma) := \argmax_{\varepsilon} [\sigma w_{\mu}(\varepsilon) - \tfrac{\varepsilon^2}{2}]
\end{align}
is very related (i.e., in some sense it controls) the risk $\EE \|\hat \mu - \mu\|^2$ where $\hat \mu = \hat \mu(Y)$ is the LSE defined in \eqref{LSE:def}.
This fact can also be seen in the following lemma, which follows a similar logic to  \citet[Corollary 1.2]{chatterjee2014new}. For completeness, we give the full proof in the appendix.

\begin{lemma}\label{lemma:chatterjee:analogue} If $\epswidth(\sigma) \geq C \sigma$, then $\EE \|\hat \mu - \mu\|^2 \asymp \epswidth^2$. On the other hand, if $\epswidth(\sigma) < C \sigma$, then $\EE \|\hat \mu - \mu\|^2 \lesssim \sigma^2$. Here $\gtrsim, \asymp$ and $\lesssim$ hide universal constants and $C > 0$ is another universal constant (some values of the constants are calculated in the proof of the lemma).
\end{lemma}

 We now define $\epsSupwidth(\sigma) := \sup_{\mu \in K} \epswidth(\sigma)$. When the context is clear, we drop the $\sigma$ from $\epswidth$ or $\epsSupwidth$. The following lemmas give some additional properties of  $\epsSupwidth(\sigma)$ that we use in Section \ref{section:algorithm_proofs} where we give algorithms to find the worst case LSE rate.

\begin{lemma} \label{lemma:overline_epsilon_K_diameter}  $\epsSupwidth(\sigma) \le d$ where $d$ is the diameter of $K$.
\end{lemma}
    \begin{proof}
        Fix any $\mu\in K$, and let $\delta>d$. Note that $w_{\mu}(\delta)=w_{\mu}(d)=w(K)$ since $B(\mu,d)\cap K=B(\mu,\delta)\cap K =K$. Thus, \[\sigma w_{\mu}(\delta)-\delta^2/2 < \sigma w_{\mu}(d)-d^2/2.\] Hence $\epswidth(\sigma)=\argmax_{\varepsilon}[\sigma w_{\mu}(\varepsilon)-\varepsilon^2/2]\le d$. Since this holds for all $\mu\in K$, we have $\epsSupwidth(\sigma) \le d$.
    \end{proof}



\begin{lemma} \label{lemma:epsilon_mu:nondecreasing} The map $\sigma\mapsto \epsSupwidth(\sigma)$ is non-decreasing on $[0,\infty)$. If $c\ge 1$, then $\epsSupwidth(\sigma) \leq \epsSupwidth(c \sigma)\leq c \epsSupwidth(\sigma)$. If $0\le c<1$, then $c\epsSupwidth(\sigma) \le \epsSupwidth(c\sigma)\le \epsSupwidth(\sigma)$. 
\end{lemma}

\subsection{Sufficient conditions on the worst case performance of the LSE}
\label{subsection:sufficient:conditions:worst:case}

We now consider the worst case risk of the LSE estimator and prove some simple bounds. Let 
\begin{align*}
    \epsLSE^2(\sigma) := \epsLSE^2 := \sup_{\mu \in K} \EE \|\hat \mu - \mu\|^2
\end{align*}
denote the worst case risk for the LSE estimator $\hat \mu$. Since $\hat \mu \in K$, clearly $\epsLSE \leq \diam(K) =: d$. We will now establish the following upper bound on $\epsLSE$.

\begin{table*}[t]
\caption{Controlling the worst case LSE rate with different choices of  $\overline{\varepsilon}(\sigma)$.}
  \centering
  \begin{tabular}{m{0.5cm}|m{10cm}|m{2cm}}
  \hline
     & Definition of $\overline{\varepsilon}=\overline{\varepsilon}(\sigma)$ & Usage \\
    \hline 
    1 & $\begin{aligned} &\sup_{\varepsilon}\{\varepsilon: \tfrac{\varepsilon^2}{2\sigma} \leq \sup_{\mu \in K} w_{\mu}(\varepsilon) \} \end{aligned}$ & Prop. \ref{proposition:first:sufficient:condition} \\[1.2ex] \hline 
    2 & $\begin{aligned} &\sup_{\varepsilon} \bigg\{\varepsilon: \sup_{\mu \in K} \big[w_{\mu}(\varepsilon) - \inf_{\nu \in B(\mu, \varepsilon) \cap K} w_{\nu}(\tfrac{\varepsilon}{c}) \big] \geq \left(4 + \tfrac{4}{c}\right) \tfrac{\varepsilon^2}{2\sigma}\bigg\}\end{aligned}$ &  Thm. \ref{big:width:minus:small:width:thm} \\[4ex] \hline
    3 & $\begin{aligned} &\sup_{\varepsilon} \Big\{\varepsilon: \sup_{\substack{\nu_1,\nu_2 \in K \\ \|\nu_1 - \nu_2\|\leq 2\varepsilon}} w_{\nu_1}(\tfrac{\varepsilon}{c^{\ast}})-  w_{\nu_2}(\tfrac{\varepsilon}{c^{\ast}})  - \tfrac{C \varepsilon^2}{2\sigma} \\ &\qquad \qquad + \tfrac{L}{c^{\ast}} \cdot\varepsilon \sqrt{\log \cMloc(\varepsilon)}\ge 0\Big\}\end{aligned}$ & Thm. \ref{difference:of:local:widths:compared:to:eps:squared:thm} \\[4ex] \hline
    4 & $\begin{aligned} &\sup_{\varepsilon}\Big\{\varepsilon:\sup_{\nu_1,\nu_2 \in K} w_{\nu_1}(\tfrac{\varepsilon}{c^{\ast}}) - w_{\nu_2}(\tfrac{\varepsilon}{c^{\ast}})  - \frac{C \varepsilon \|\nu_1 - \nu_2\|}{\sigma} \\ &\qquad\qquad + \tfrac{L}{c^{\ast}} \|\nu_1 - \nu_2\|\sqrt{\log \cMloc(\varepsilon)} \geq 0\Big\}\end{aligned}$ & Thm. \ref{Lipschitz:map:theorem} \\[4ex] \hline 
  \end{tabular}
  \vspace*{5pt}
  \label{tab:2}
\end{table*}

\begin{proposition}\label{proposition:first:sufficient:condition}
    Let $\overline \varepsilon:= \sup_{\varepsilon}\{\varepsilon^2/(2\sigma) \leq \sup_{\mu \in K} w_{\mu}(\varepsilon) \}$. Then $\epsLSE \lesssim \overline \varepsilon \wedge d$.
\end{proposition}

\begin{proof}

Since $\epsLSE\lesssim d$ always holds, it suffices to show $\epsLSE\lesssim \overline\varepsilon$. Recall we set $\epsSupwidth = \sup_{\mu \in K} \epswidth$, where $\epswidth = \argmax_{\varepsilon} \sigma w_{\mu}(\varepsilon)  - \varepsilon^2/2$. Observe that for every $\mu \in K$, we have $\sigma w_{\mu}(2\overline\varepsilon) - 2 \overline{\varepsilon}^2 < 0$, and thus by  \citet[Proposition 1.3]{chatterjee2014new} we have $2\overline\varepsilon > \epsSupwidth$. We now relate $\epsSupwidth$ to ${\epsLSE}$. 

\textsc{Case 1:} Suppose $\epsSupwidth \gtrsim \sigma$. Now pick $\tilde{\mu}$ that maximizes $\EE\|\hat\mu-\tilde\mu\|^2$. This leads to two subcases.

\textsc{Case 1(a):} Suppose $\epswidth[\tilde\mu]\gtrsim \sigma$. Then by definition of $\tilde\mu$ and $\epsSupwidth$ along with Lemma \ref{lemma:chatterjee:analogue}, we have \[\epsLSE^2=\EE\|\hat\mu-\tilde\mu\|^2\asymp\epswidth[\tilde\mu]^2\lesssim\epsSupwidth^2.\] But we showed $\epsSupwidth\lesssim \overline{\varepsilon}$, so we conclude $\epsLSE\lesssim \overline\varepsilon$ as claimed.   

\textsc{Case 1(b):} Suppose $\epswidth[\tilde\mu]\lesssim \sigma$. Then by Lemma \ref{lemma:chatterjee:analogue} and our Case 1 assumption, \[\epsLSE^2=\EE\|\hat\mu-\tilde\mu\|^2\lesssim \sigma^2 \lesssim\epsSupwidth^2 \lesssim 4\overline{\varepsilon}^2.\] This concludes Case 1.


\textsc{Case 2:} Suppose $\epsSupwidth \lesssim \sigma$. This means for any $\mu\in K$, $\epswidth\lesssim \sigma$ which in turn implies for all $\mu$ that $\EE\|\hat\mu-\mu\|\lesssim \sigma$ by Lemma \ref{lemma:chatterjee:analogue}. Hence $\epsLSE\lesssim \sigma$.
 
 \textsc{Case 2(a):} Suppose furthermore that $d \geq \frac{2\sigma (n+1)}{n\sqrt{2\pi}}$. We claim this implies $\overline{\varepsilon} \gtrsim\sigma$ which would prove $\overline{\varepsilon}\gtrsim \epsLSE$ since $\epsLSE\lesssim \sigma$ in Case 2. To see that $\overline{\varepsilon} \gtrsim\sigma$, notice that $K$ contains a diameter of length $2\sigma/\kappa$ (for a sufficiently large absolute constant $\kappa$). Thus by  \citet[Proposition 7.5.2(vi)]{vershynin2018high}, $\sup_{\mu \in K} w_\mu(\sigma/\kappa) \geq 2\sigma/(\kappa\sqrt{2\pi})$. Now for sufficiently large $\kappa$ we have $\sigma/(2\kappa^2) \leq 2\sigma/(\kappa\sqrt{2\pi})$, so that $\sup_{\mu \in K} w_\mu(\varepsilon) \ge \varepsilon^2/2\sigma$ holds for $\varepsilon = \sigma/\kappa$. But then  $\overline{\varepsilon} \geq \sigma/\kappa$ by definition as a supremum of such $\varepsilon$. Hence $\epsLSE \lesssim \sigma \lesssim \overline\varepsilon$ as desired. 
 
 \textsc{Case 2(b):} Suppose $d \leq \frac{2\sigma (n+1)}{n\sqrt{2\pi}}$. By Jung's Theorem and \citet[Proposition 7.5.2(vi)]{vershynin2018high}, there exists $\mu \in K$ such that \[w_\mu\left(\sqrt{\tfrac{n}{2(n + 1)}} d\right) = w(K) \geq \tfrac{d}{\sqrt{2\pi}}\ge \tfrac{d}{2\sqrt{\pi}}.\] Our assumption on $d$ implies $\frac{nd^2}{2(n + 1)(2\sigma)}  \leq \tfrac{d}{2\sqrt{2\pi}}$ which in turn implies that $w_\mu(\varepsilon) \ge \varepsilon^2/(2\sigma)$ for $\varepsilon= \sqrt{\frac{nd^2}{2(n + 1)}} \gtrsim d$. Therefore, $\overline{\varepsilon} \gtrsim  d$ due to its definition as a supremum, and since $d \lesssim \overline{\varepsilon}\wedge d \leq d$, we have $\epsLSE \lesssim d \asymp \overline{\varepsilon}\wedge d$.
\end{proof}

\begin{remark}
    It should be noted that Proposition \ref{proposition:first:sufficient:condition} is simply an upper bound on the rate of the LSE. We will later see an example with hyper-rectangles (Section \ref{subsubsection:hyperrectangle}) where this upper bound is very suboptimal.
\end{remark}

\begin{corollary}\label{cool:corollary}
    Suppose that $\sup_{\mu \in K} w_{\mu}(\varepsilon)/\varepsilon \lesssim \sqrt{\log \cMKloc(c \varepsilon)}$ for all $\varepsilon\le d$. Then the LSE is minimax optimal for all $\sigma$. For (centrally) symmetric sets it suffices to examine the $0$ point only (both for the maximal local Gaussian width and for the local entropy).
\end{corollary}


\begin{proof}
    For each $\mu$ the map $\varepsilon \mapsto  \tfrac{w_{\mu}(\varepsilon)}{\varepsilon}$ is non-increasing, hence $\varepsilon\mapsto\sup_{\mu \in K} \tfrac{w_{\mu}(\varepsilon)}{\varepsilon}$ is also non-increasing. To see this, note that for $\varepsilon < \delta$,  since $w_{\mu}(0)=0$, we have \begin{align*}
        w_{\mu}(\varepsilon) \ge \tfrac{\varepsilon}{\delta}\cdot w_{\mu}(\delta) + (1- \tfrac{\varepsilon}{\delta})\cdot w_{\mu}(0)
    \end{align*} by concavity of $\varepsilon\mapsto w_{\mu}(\varepsilon)$ \citep[Proof of Theorem 1.1]{chatterjee2014new},  which implies $\tfrac{w_{\mu}(\varepsilon)}{\varepsilon} \ge \tfrac{w_{\mu}(\delta)}{\delta}.$

    Recall Proposition \ref{proposition:first:sufficient:condition} and the definition of $\overline{\varepsilon}$ and note that $\epsLSE \lesssim \overline{\varepsilon}\wedge d\le d$. This means for some $c'>0$, $c'\epsLSE \le \overline{\varepsilon}$. Pick some $c''\in(0,1)$ so that $c'\epsLSE\le c''\overline{\varepsilon}$. By definition, $c''\overline\varepsilon/(2\sigma)\le  \sup_{\mu\in K}w_{\mu}(c''\overline\varepsilon)/(c''\overline\varepsilon)$. Using these facts and the aforementioned non-increasing property of $\varepsilon\mapsto\sup_{\mu \in K} \tfrac{w_{\mu}(\varepsilon)}{\varepsilon}$, we have \begin{align*}
        \frac{c'\epsLSE}{2\sigma} &\le \frac{c''\overline\varepsilon}{2\sigma}\le \frac{\sup_{\mu\in K}w_{\mu}(c''\overline\varepsilon)}{c''\overline\varepsilon} \\ &\le \frac{\sup_{\mu\in K}w_{\mu}(c'\epsLSE)}{c'\epsLSE}.
    \end{align*}
    
    If $c'>1$, then $\sup_{\mu\in K} \frac{w_{\mu}(c'\epsLSE)}{c'\epsLSE}\le \sup_{\mu\in K} \frac{w_{\mu}(\epsLSE)}{\epsLSE}$ again by the non-increasing property so that $\frac{\epsLSE}{2\sigma} \lesssim \sup_{\mu\in K}\frac{w_{\mu}(\epsLSE)}{\epsLSE}$. Then using our assumption, $\sup_{\mu \in K} \frac{w_{\mu}(\varepsilon)}{\varepsilon} \lesssim \sqrt{\log \cMloc(c \varepsilon)}$ holds for $\varepsilon=\epsLSE$ since $\epsLSE\le d$. Hence $\frac{\epsLSE^2}{\sigma^2} \lesssim \log \cMloc(c \epsLSE)$, which by Lemma \ref{lemma:equivalent:information:lower:bound} implies $\varepsilon^{\ast}\gtrsim \epsLSE$, i.e., the LSE is minimax optimal.

    On the other hand, if $c'\le 1$, then we have $\sup_{\mu\in K} w_{\mu}(c'\epsLSE)\le \sup_{\mu\in K} w_{\mu}(\epsLSE)$, therefore \begin{align*}
        \frac{c'\epsLSE}{2\sigma} &\le \frac{\sup_{\mu\in K}w_{\mu}(c'\epsLSE)}{c'\epsLSE} \le  \frac{\sup_{\mu\in K} w_{\mu}(\epsLSE)}{c'\epsLSE}\\ &\lesssim \sqrt{\log \cMloc(c\epsLSE)}.
    \end{align*} Once more we have $\varepsilon^{\ast} \gtrsim\epsLSE$.
    
    To see the last implication, let $K$ be a centrally symmetric set and suppose we have $w_{0}(\varepsilon)/\varepsilon \lesssim \sqrt{\log \cM(c\varepsilon/c^{\ast}, B(0,c\varepsilon)\cap K)}$ for all $\varepsilon\le d$. Pick any $\mu \in K$. Note that    
    \begin{align*}
        \MoveEqLeft \alpha [B(\nu, \varepsilon) \cap K] + (1-\alpha) [B(\mu, \varepsilon)\cap K ] \\ &\subseteq B(\alpha \nu + (1-\alpha) \mu, \varepsilon)\cap K.
    \end{align*}
    
    Take $\nu = -\mu$ and $\alpha = 1/2$, so that $(1/2)[B(-\mu,\varepsilon)\cap K] + (1/2)[B(\mu,\varepsilon)\cap K]\subseteq B(0,\varepsilon)\cap K$. The width of $B(-\mu,\varepsilon)\cap K$ and $B(\mu,\varepsilon)\cap K$ are the same by central symmetry. Hence \begin{align*}
        w_{\mu}(\varepsilon)=w(B(\mu,\varepsilon)\cap K)\le w(B(0,\varepsilon)\cap K)= w_{0}(\varepsilon).
    \end{align*}
    
    We then consider the local entropy. Let $\nu \in K$ be arbitrary, and let $\theta_1, \ldots, \theta_M$ be a maximal $c\varepsilon/c^*$ packing of $B(\nu, c\varepsilon)\cap K$. By central symmetry, the points $(\theta_i - \nu)/2$ form a $c\varepsilon/(2c^*)$ packing of $B(0, c\varepsilon/2)\cap K$. Therefore, \begin{align*}
       \log \cMloc(c\varepsilon) &= \sup_{\nu\in K}\log \cM(c\varepsilon/c^{\ast}, B(0,c\varepsilon)\cap K) \\ &\leq \log \cM(c\varepsilon/(2c^*), B(0, c\varepsilon/2)\cap K) \\ &\le \log \cMloc(c\varepsilon/2).
    \end{align*} Thus, $\sup_{\mu\in K} w_{\mu}(\varepsilon)/\varepsilon \le\sqrt{\log \cMloc(c\varepsilon/2)}$. Hence our hypothesis need only check the zero point of a centrally symmetric set. 
\end{proof}

\begin{remark}
    One may conjecture that the condition in Corollary \ref{cool:corollary} is also necessary. We will see a counterexample in Section \ref{subsubsection:hyperrectangle} with hyper-rectangles. We also consider ellipsoids and derive a necessary condition similar in spirit to the corollary. 
\end{remark}

\begin{remark}\label{remark:upper:bound:wK}Another quick corollary to Proposition \ref{proposition:first:sufficient:condition} is that $\epsLSE \lesssim \sqrt{\sigma w(K)}$ by trivially bounding  $w_\mu(\varepsilon) \leq w(K)$. This bound is achievable for some sets $K$; for an example, see Section \ref{subsubsection:optimality:ellipsoids}. 
\end{remark}

Proposition \ref{proposition:first:sufficient:condition} can be extended to a slightly more general upper bound. Suppose we are interested in the worst case risk over a convex subset $K' \subseteq K$. In other words let
\begin{align}\label{varesilongKprime:K}
    \epsLSE[K';K]^2 := \sup_{\mu \in K'} \EE \|\hat \mu - \mu\|^2
\end{align}
denote the worst case risk for the LSE estimator $\hat \mu$. We will now establish the following:

\begin{proposition}\label{proposition:first:sufficient:condition:generalized}
    Let $\overline \varepsilon_{K';K} := \sup_{\varepsilon}\{\varepsilon^2/(2\sigma) \leq \sup_{\mu \in K'} w_{K,\mu}(\varepsilon)\}$. Then $\epsLSE[K';K] \lesssim \overline \varepsilon_{K';K}$.
\end{proposition}

Since the proof of Proposition \ref{proposition:first:sufficient:condition:generalized} is almost identical to that of Proposition \ref{proposition:first:sufficient:condition}, we defer it to the appendix. So far we saw some sufficient conditions for the worst-case performance of the LSE. We now derive a similar version of the celebrated result of \citet{birge1993rates} using Proposition \ref{proposition:first:sufficient:condition}. For a related result, see \citet[Corollary 13.7]{wainwright2019high}.

\begin{corollary}
    Let $\overline\varepsilon$ be as defined in Proposition \ref{proposition:first:sufficient:condition}, and let $c'$ be some absolute constant. Suppose $\sigma = \frac{1}{\sqrt{n}}$. If $\varepsilon$ is such that $\int_{(c'/16)\varepsilon^2}^{2\varepsilon} \sqrt{\log \cM(t,K)} \mathrm{d}t \lesssim \sqrt{n}\varepsilon^2$, then $\varepsilon \gtrsim \overline{\varepsilon}$.
\end{corollary}
\begin{proof}
    Using Dudley's entropy bound (see \citet[Theorem 5.22]{wainwright2019high}), for any fixed $c$ we have
    \begin{align*}
         w_{\mu}(\varepsilon) \leq 2 \sqrt{n}c^2\varepsilon^2 + C\int_{c^2c'\varepsilon^2}^{2\varepsilon} \sqrt{\log \cM(t, B(\mu, \varepsilon) \cap K)} \mathrm{d}t,
    \end{align*}
    where $c',C$ are absolute constants. We can now bound
    \begin{align*}
        \log \cM(t, B(\mu, \varepsilon) \cap K) \leq \log \cM(t, K),
    \end{align*}
    which makes the bound independent of $\mu$. Let $c = 1/4$. Then
    \begin{align*}
         \sigma \sup_{\mu}w_{\mu}(\varepsilon) \leq \varepsilon^2/8 + C/\sqrt{n}\int_{(c'/16)\varepsilon^2}^{2\varepsilon} \sqrt{\log \cM(t, K)} \mathrm{d}t,
    \end{align*}
    and so as long as $\int_{(c'/16)\varepsilon^2}^{2\varepsilon} \sqrt{\log \cM(t, K)} \mathrm{d}t \lesssim \sqrt{n}\varepsilon^2$, we have $\sigma \sup_{\mu} w_{\mu}(\varepsilon) \le C'\varepsilon^2 $ for some absolute constant $C'>1$. This implies $\varepsilon\gtrsim \overline{\varepsilon}$. To see this, first observe that $\varepsilon \mapsto \frac{\sup_{\mu} w_{\mu}(\varepsilon)}{\varepsilon}$ is non-increasing. Therefore, \begin{align*}
       \frac{\sup_{\mu} w_{\mu}(2C'\varepsilon)}{2C'\varepsilon} &\le \frac{\sup_{\mu} w_{\mu}(\varepsilon)}{\varepsilon} \le \frac{C'\varepsilon }{\sigma}.
    \end{align*} Rearranging, $\sup_{\mu} w_{\mu}(2C'\varepsilon) \le \frac{(2C'\varepsilon )^2}{2\sigma}$, but by definition of $\overline\varepsilon$, this implies $\overline\varepsilon \le 2C'\varepsilon$.
\end{proof}

The next set of results bound the LSE with a geometric average of the minimax rate and a trivial estimator up to some logarithmic factors. Define $\underline{\varepsilon} = \sup_{\varepsilon}\{\varepsilon^2/(2\sigma) \leq \frac{1}{2}\sup_{\delta \leq \varepsilon} \tfrac{\delta}{c^{\ast}} \sqrt{\log \cMloc(\delta)}\}$. We first compare $\epsLSE$ to $\underline\varepsilon$ up to log factors (Theorem \ref{important:thm}) and then derive the geometric average result (Corollary \ref{corollary:geometric_average:minimax}). Remark \ref{remark:geometric_average:minimax} examines the sharpness of the bound and connects the results to a Donsker-regime assumption.

\begin{theorem} \label{important:thm} Let $C_n = 4 C\bigg({1 + \log_{c^{\ast}} \sqrt{2 \pi n}}\bigg)^{3/2}$, where $C > 1$ and $c^{\ast}$ are sufficiently large absolute constants (and $c^{\ast}$ is the constant from the definition of local entropy). Then $\epsLSE \lesssim (4 C_n \underline{\varepsilon}) \wedge d$. 
\end{theorem}

\begin{corollary} \label{corollary:geometric_average:minimax}The following inequality always holds for the LSE:
\begin{align*}
    \epsLSE \lesssim \sqrt{\sigma} C_n (\sqrt{\varepsilon^*}\sqrt[4]{n}).
\end{align*}
\end{corollary}
    \begin{proof}
         Let $\delta_{\kappa}\le \underline{\varepsilon}$ be such that $\delta_{\kappa} \sqrt{\log \cMloc(\delta_{\kappa})} \ge \sup_{\delta\le\underline\varepsilon}\delta \sqrt{\log \cMloc(\delta)} -\kappa$.

Case 1: $\delta_{\kappa} > \varepsilon^*$. Then using that $\delta_{\kappa}\le \underline\varepsilon$, the non-decreasing property of $\varepsilon\mapsto\log \cMloc(\varepsilon)$, and the definition of $\varepsilon^{\ast}$, we obtain \[\delta_{\kappa} \sqrt{\log \cMloc(\delta_{\kappa})} \leq \underline\varepsilon \sqrt{\log \cMloc(\varepsilon^*)} \asymp \frac{\underline\varepsilon \varepsilon^*}{\sigma}.\] Using the definition of $\underline\varepsilon$, we can show $\underline\varepsilon\lesssim\varepsilon^{\ast}$ since \begin{align*}
    \frac{\underline\varepsilon^2}{2\sigma} \leq \frac{1}{2c^{\ast}}\left(\delta_{\kappa} \sqrt{\log \cMloc(\delta_{\kappa})} +\kappa\right) \lesssim \frac{\underline\varepsilon\varepsilon^{\ast}}{\sigma}.
\end{align*} By Theorem \ref{important:thm}, we obtain \[\epsLSE\lesssim C_n\underline\varepsilon \lesssim C_n \varepsilon^* = C_n\sqrt{\varepsilon^{\ast}}\sqrt{\varepsilon^{\ast}}\lesssim C_n\sqrt{\varepsilon^*}\sqrt[4]{n}\sqrt{\sigma},\] where we use the fact that $\varepsilon^{\ast}\le \sqrt{n}\sigma$ (which follows since the minimax estimator is at least better than using $Y$ as an estimator of $\mu$, and this estimator has error rate $n\sigma^2$).

Case 2: $\delta_{\kappa} \leq \varepsilon^*$. Then we have  $ \delta_{\kappa} \sqrt{\log \cMloc(\delta_{\kappa})} \lesssim \varepsilon^*\sqrt{n}$. Here we use the fact that $\log \cMloc(\delta_{\kappa}) \lesssim n$. To see this, note that $\cMloc(\delta_{\kappa}) \le \cM(\delta_{\kappa}/c^{\ast}, B(\theta,\delta_{\kappa})) \lesssim (1 +{c^\ast})^n$ using the well-known metric entropy of a scaled $\ell_2$-ball. Again using $\underline\varepsilon^2 \lesssim \sigma(\delta_k\sqrt{\log \cMloc(\delta_{\kappa})} +\kappa)$, we have  \begin{align*}
    \epsLSE &\lesssim C_n\underline\varepsilon \lesssim C_n \sqrt{\sigma \cdot \left(\delta_k\sqrt{\log \cMloc(\delta_{\kappa})} +\kappa\right) } \\ &\lesssim  C_n\sqrt{\varepsilon^*}\sqrt[4]{n}\sqrt{\sigma}.
\end{align*}
    \end{proof}

\begin{remark}\label{remark:geometric_average:minimax}
We will see later that without further assumptions on $K$ this bound is sharp up to the logarithmic factors, i.e., there exist sets $K$ for which the bound is met with equality dropping the $\log$ factors and constant terms (Section \ref{subsubsection:optimality:ellipsoids}). For now, it suffices to say that the bound is sharp when $\sigma \ll 1/\sqrt{n}$, and $K$ is the unit $\ell_2$ ball.\footnotemark{} 

Note that the rate $\sqrt{\varepsilon^*} \sqrt{\sigma}\sqrt[4]{n}$ is the geometric mean of the optimal rate, and a trivial rate (which is achieved by using the observation $Y$ as the estimator of the mean). So the rate of the LSE is always not worse than this geometric mean. Furthermore, when $\log \cMloc(\delta) \asymp \delta^{-\alpha}$ for some $\alpha < 2$, the supremum satisfies $\sup_{\delta \leq \underline\varepsilon} \delta \sqrt{\log \cMloc(\delta)} \asymp (\varepsilon/c^{\ast}) \log \cMloc(\varepsilon)$. Hence,upon equating this to $\varepsilon^2/(2\sigma)$, we realize that the LSE is minimax optimal up to logarithmic factors for all $\sigma$. This latter assumption is known as the ``Donsker regime,'' although it is typically assumed that the global entropy scales like $\delta^{-\alpha}$, a stronger requirement than the same assumption on the local entropy (see, e.g., Lemma \ref{yb:local:entropy:bound}).
\end{remark}
\footnotetext{The minimax rate on the squared level for the unit $\ell_2$ ball is given by $\min(1,n \sigma^2)$ \citep[see][e.g.]{zhang2013nearly}.}

\subsection{Characterizations of the worst case rate of the LSE}
\label{subsection:characterizations:conditions:worst:case}

We will now see a series of results which attempt to characterize (up to constants) the worst case risk of the LSE. In the proof of these and several later results, we repeatedly use \citet[Proposition 1.3]{chatterjee2014new} to bound $\epswidth(\sigma)$. That proposition notes that $\varepsilon\mapsto w_{\mu}(\varepsilon) - \varepsilon^2/(2\sigma)$ is a strictly concave mapping, and as a consequence, if $\alpha \ge \beta >0$ and \[w_{\mu}(\alpha) - \alpha^2/(2\sigma) \ge w_{\mu}(\beta) - \beta^2/(2\sigma),\] then $\epswidth(\sigma)\ge \beta$. If on the other hand $w_{\mu}(\alpha) - \alpha^2/(2\sigma) \le w_{\mu}(\beta) - \beta^2/(2\sigma)$, then $\epswidth(\sigma)\le \alpha$. Choosing $\mu\in K$ and $\alpha,\beta>0$ appropriately will give us the desired bounds on $\epswidth(\sigma)$, and therefore $\epsSupwidth(\sigma)$ and $\epsLSE$ with the help of Lemma \ref{lemma:chatterjee:analogue}.

We proceed to our first proper characterization, which we use in Lemma \ref{lemma:strongly:convex:body} to prove suboptimality of the LSE for $\ell_p$ balls with $p\in(1,2)$. 

\begin{theorem}\label{big:width:minus:small:width:thm}
    Let $\overline \varepsilon(\sigma)$ be defined as
    \begin{align*}
        \sup_{\varepsilon} \bigg\{\varepsilon: \sup_{\mu \in K} \bigg[w_{\mu}(\varepsilon) - \inf_{\nu \in B(\mu, \varepsilon) \cap K} w_{\nu}(\tfrac{\varepsilon}{c}) \bigg] \geq \left(4 + \tfrac{4}{c}\right) \tfrac{\varepsilon^2}{2\sigma}\bigg\},
    \end{align*} where $c$ is some absolute constant. Set $\sigma' = 4c\sigma/(c-1)$. Then if $\overline{\varepsilon}(\sigma) \gtrsim \sigma$ for a sufficiently large constant, we have $\overline{\varepsilon}(\sigma)/c \lesssim {\epsLSE}(\sigma) \lesssim \overline{\varepsilon}(\sigma')$; if $\overline{\varepsilon}(\sigma') \lesssim \sigma$, we have ${\epsLSE}(\sigma) \asymp \sigma \wedge d$.
\end{theorem}

A similar result to the one above is given by the next theorem. We later apply it to demonstrate LSE optimality for subspaces (Section \ref{subsubsection:subspace:linear}) and suboptimality for pyramids (Lemma \ref{lemma:pyramid:suboptimal}). 

\begin{theorem}\label{difference:of:local:widths:compared:to:eps:squared:thm}Define $\overline{\varepsilon}(\sigma)$ by 
    \begin{align*}
         \sup_{\varepsilon} \bigg\{\varepsilon: &\sup_{\substack{\nu_1,\nu_2 \in K \\ \|\nu_1 - \nu_2\|\leq 2\varepsilon}}  w_{\nu_1}(\tfrac{\varepsilon}{c^{\ast}})-  w_{\nu_2}(\tfrac{\varepsilon}{c^{\ast}})  - \tfrac{C \varepsilon^2}{2\sigma} \\ &+ \tfrac{L}{c^{\ast}} \cdot\varepsilon \sqrt{\log \cMloc(\varepsilon)}\ge 0\bigg\}
    \end{align*}
    where $C = 2[(2 + 1/c^{\ast})^2-1/{c^{\ast}}^2] = 8 + 8/c^{\ast}$ for some $c^{\ast}>1$, and $L$ is an absolute constant. Set $\sigma' = \tfrac{C\sigma}{1-1/{c^{\ast}}^2}$.
    If $\overline{\varepsilon}(\sigma) \gtrsim \sigma$ for a sufficiently large constant, then $\overline{\varepsilon}(\sigma)/c^{\ast} \lesssim {\epsLSE}(\sigma) \leq \overline{\varepsilon}(\sigma')$. If $\overline{\varepsilon}(\sigma') \lesssim \sigma$, then ${\epsLSE}(\sigma) \asymp \sigma \wedge d$.
\end{theorem}

To prove this theorem, we additionally define 
\begin{align}\label{underline:varepsilon:def}
\underline\varepsilon^{\ast} = \sup\{\varepsilon>0:C^2 \varepsilon^{2}/(4\sigma^2) \leq (L/c^{\ast})^2 \log \cMloc(\varepsilon)\}.
\end{align}
Luckily up to constants this is the information theoretic lower bound defined in \eqref{varepsilon:star:def}, as we proved in Lemma \ref{lemma:equivalent:information:lower:bound}. 

\begin{remark} \label{remark:difference:of:local:widths}
    We have $\overline{\varepsilon}(\sigma) \ge \underline\varepsilon^{\ast} \asymp \varepsilon^{\ast}$. To see this, pick any $\varepsilon\le\underline\varepsilon^{\ast}$. By definition of $\underline\varepsilon^{\ast}$, we have $L/c^{\ast}\cdot\sqrt{ \log \cMloc(\varepsilon)} - C\varepsilon/2\sigma >0$, and by choosing $\nu_1,\nu_2$ appropriately, the expression we take the supremum over in $\overline\varepsilon(\sigma)$ is non-negative and $\varepsilon$ satisfies the given condition. By definition of $\overline{\varepsilon}(\sigma)$ as a supremum,  $\overline\varepsilon(\sigma) \ge \varepsilon$. Thus for any such $\varepsilon \le \underline\varepsilon^{\ast}$, we can show that $\overline\varepsilon(\sigma) \ge \varepsilon$, which implies $\overline{\varepsilon}(\sigma)\ge \underline\varepsilon^{\ast}$. Then we invoke Lemma \ref{lemma:equivalent:information:lower:bound}.
\end{remark}

In our next set of results, we analyze the Lipschitz constant of the map $\nu \mapsto w_\nu(\varepsilon)$ over $K$, noting that this map is always Lipschitz (Remark \ref{remark:Lipschitz}). It turns out this Lipschitz constant controls the worst case LSE rate $\epsLSE$ (Theorem \ref{Lipschitz:map:theorem}) and thus yields an equivalent condition for LSE optimality (Corollary \ref{corollary:Lipschitz}). We conclude the section by observing some easier ways to prove this map is Lipschitz (Remark \ref{remark:Lipschitz:on:boundary}). We apply Theorem \ref{Lipschitz:map:theorem} in our solids of revolution suboptimality result (Lemma \ref{lemma:solid:revolution}), and both this theorem and Corollary \ref{corollary:Lipschitz} trivially yield an optimality result for subspaces (Section \ref{subsubsection:subspace:linear}).

\begin{remark} \label{remark:Lipschitz} The mapping $\nu \mapsto w_\nu(\varepsilon)$ is Lipschitz over $K$.
\end{remark}
    \begin{proof}

     Suppose for a fixed Gaussian vector $\xi$, $x_{\xi}$ achieves the max $\langle \xi, x\rangle$ over $x\in B(\mu, \varepsilon) \cap K$ (if such a vector does not exist, i.e., the maximum is unattainable, the same reasoning works using a limiting sequence argument).  Then the vector 
     \begin{align*}
         y_{\xi}' = \frac{\varepsilon}{\|\mu- \nu\| + \varepsilon}\cdot x_{\xi} + \frac{\|\mu-\nu\|}{\|\mu- \nu\| + \varepsilon}\cdot \nu
     \end{align*} 
     belongs to $B(\nu, \varepsilon)\cap K$ since $y_{\xi}'$ is a convex combination of points in $K$ and $\|y_{\xi}' - \nu\| = \tfrac{\varepsilon}{\|\mu- \nu\| + \varepsilon} \cdot \|x_{\xi} - \nu\| \leq \varepsilon$. Using that $w_{\nu}(\varepsilon)\ge \EE_{\xi}\langle\xi,y_{\xi}'\rangle$ and expanding the definition of $y_{\xi}'$, we have
     \begin{align*}
         w_{\mu}(\varepsilon) - w_{\nu}(\varepsilon) &\le \EE_{\xi}\langle \xi, x_{\xi}\rangle -\EE_{\xi}\langle \xi, y_{\xi}'\rangle 
        \\ &\le \frac{\|\mu-\nu\|}{\|\mu- \nu\| + \varepsilon}(\EE_{\xi}\langle \xi, x_{\xi}\rangle - \underbrace{\EE_{\xi}\langle \xi, \nu\rangle}_{=0}) \\
         &= \frac{\|\mu - \nu\|\cdot w_{\mu}(\varepsilon)}{\|\mu- \nu\| + \varepsilon} \le  \frac{\|\mu - \nu\|w_{\mu}(\varepsilon)}{\varepsilon}.
     \end{align*} Now, take $\alpha \mu + (1-\alpha)\nu$ for $\alpha\in(0,1)$. We have by concavity of $\nu\mapsto w_{\nu}(\varepsilon)$ (Lemma \ref{lemma:concavity:wnu:epsilon}) and the penultimate inequality above that
    \begin{align*}
    \alpha(w_{\mu}(\varepsilon) - w_{\nu}(\varepsilon)) &\leq  w_{\alpha \mu + (1-\alpha)\nu}( \varepsilon) - w_{\nu}(\varepsilon) \\ &\leq \alpha \|\mu - \nu\|\cdot \frac{w_{\alpha \mu + (1-\alpha)\nu}( \varepsilon)}{\varepsilon}.
    \end{align*}
    Thus dividing by $\alpha$ and taking $\alpha \rightarrow 0$ (and noting that by its Lipschitz condition, the map $\alpha \rightarrow w_{\alpha \mu + (1-\alpha)\nu}( \varepsilon)$ is continuous) shows that 
    \begin{align*}
         w_{\mu}(\varepsilon) - w_{\nu}(\varepsilon) &\le  \frac{\|\mu - \nu\|w_{\nu}(\varepsilon)}{\varepsilon}.
     \end{align*}
    Using the symmetry in $\mu$ and $\nu$, we conclude
    \begin{align}
        |w_{\mu}( \varepsilon) -w_{\nu}( \varepsilon)| &\leq \|\mu-\nu\|\cdot\frac{w_{\nu}( \varepsilon)\wedge w_{\mu}( \varepsilon)}{\varepsilon} \notag \\ &\leq \sqrt{n}\|\mu - \nu\|. \label{remark:Lipschitz:result}
    \end{align} The final inequality used \citet[Proposition 7.5.2(vi)]{vershynin2018high}
    \end{proof}

Equipped with this fact, we now define a quantity $\overline{\varepsilon}(\sigma)$ which encodes the tightest such Lipschitz constant of our map, and then relate it to $\epsLSE(\sigma)$. The proof will rely on the concavity of $\nu\mapsto w_{\nu}(\varepsilon)$ over $K$ (Lemma \ref{lemma:concavity:wnu:epsilon}). 

\begin{theorem}\label{Lipschitz:map:theorem} Define $\overline{\varepsilon}(\sigma)$ by
    \begin{align*}
         \sup_{\varepsilon}\bigg\{\varepsilon:&\sup_{\nu_1,\nu_2 \in K} w_{\nu_1}(\varepsilon/c^{\ast}) - w_{\nu_2}(\varepsilon/c^{\ast}) - \frac{C \varepsilon \|\nu_1 - \nu_2\|}{\sigma} \\ &+\frac{L}{c^{\ast}} \|\nu_1 - \nu_2\|\sqrt{\log \cMloc(\varepsilon)} \geq 0\bigg\},
    \end{align*} where $C = 1+\frac{2}{c^{\ast}}$ and $L$ is a sufficiently large absolute constant. We further require that $c^{\ast}>2$. Set $\sigma' =4C\sigma/(1-4/{c^{\ast}}^2)$. Then ${\epsLSE}(\sigma) \lesssim\overline{\varepsilon}(\sigma')$. 
   If $\overline{\varepsilon}(\sigma) \gtrsim \sigma$ for a sufficiently big constant, we  have $\overline{\varepsilon}(\sigma) \lesssim {\epsLSE}(\sigma)$. If $\overline{\varepsilon}(\sigma') \lesssim \sigma$, then ${\epsLSE}(\sigma) \asymp \sigma \wedge d$.
\end{theorem}

The next corollary applies Theorem \ref{Lipschitz:map:theorem} to obtain an elegant characterization of LSE optimality. 

\begin{corollary} \label{corollary:Lipschitz} Fix $\sigma$. Then the LSE is minimax optimal if and only if the map $\mu \mapsto w_\mu(\varepsilon)$ is $(\varepsilon/\sigma)$-Lipschitz up to constants for all $\varepsilon\gtrsim\varepsilon^{\ast}(\sigma)$.
\end{corollary}

\begin{remark} \label{remark:Lipschitz:on:boundary} 
    
     \citet{lipschitz_duy_tran} demonstrate that to show a convex (or   concave) function is Lipschitz, it suffices to show it is locally Lipschitz near boundary points. Combined with Corollary \ref{corollary:Lipschitz}, we obtain a potentially easier way to verify the Lipschitz property of the mapping $\mu\mapsto w_{\mu}(\varepsilon)$ and thus the optimality of the LSE.

\end{remark}

\section{Examples} \label{section:LSE:examples}

In this section, we consider several examples in order to illustrate the utility of the theory we laid out in Section \ref{main:results:sec}. While some of the examples we consider are well known, the techniques we use to obtain these results are mostly distinct from existing techniques. In addition, we also exhibit many new results including counterexamples to Corollary \ref{cool:corollary} and Proposition \ref{proposition:first:sufficient:condition}, and new classes of examples where the LSE is suboptimal. We will repeatedly apply the following useful result  from \citet{yang1999information} to study the local entropy. 

\begin{lemma}[{\cite[Lemma 3]{yang1999information}}] \label{lemma:yang:barron}
    For any set $K$ we have,
    \begin{align}\label{yb:local:entropy:bound}
        \log \frac{\cM(\tfrac{\varepsilon}{c^*}, K)}{\cM(\varepsilon, K) } \leq \log \cMloc_K(\varepsilon) \leq \log \cM(\tfrac{\varepsilon}{c^*}, K). 
    \end{align}
\end{lemma}

\subsection{Examples with optimal LSE}
\label{subsection:example:optimal:LSE}

We begin with several examples where the LSE is (nearly) minimax optimal. 


\subsubsection{Isotonic regression with known total variation bound}

In this section, we consider the isotonic estimator with known total variation bound. The case with unknown total variation bound, i.e., $S^{\uparrow} = \{\mu: \mu_1\le \mu_2\le\dots \le \mu_n\}$, was essentially analyzed by \citet{chatterjee2014new} and he showed in his equation (51) that for any $\mu\in S^{\uparrow}$, we have \[w_{S^\uparrow,\mu}(\varepsilon) \leq C \sqrt{\max(\mu_n - \mu_1,1) \varepsilon n^{1/2}} + \varepsilon^2/4.\] Set $V = \mu_n-\mu_1$. Using Proposition \ref{proposition:first:sufficient:condition:generalized} with $\sigma = 1$,  the worst case rate of the LSE is upper bounded by $\varepsilon^2$ where $\varepsilon$ solves $\sup\{\varepsilon>0:\varepsilon^2/2 \leq C \sqrt{\max(V,1) \varepsilon n^{1/2}} + \varepsilon^2/4\}$, which is attained by $\varepsilon \asymp n^{1/6}(V\vee 1)^{1/3}$.  

We now consider the case with known total variation bound. Consider the set \[S_V^\uparrow := \{\mu: \mu_1 \leq \mu_2 \leq \ldots \leq \mu_n, \mu_n - \mu_1 \leq V\},\] for some $V \in \RR_+$. Importantly, we assume the value of $V$ or some appropriate upper bound is known to the statistician so that she can fit the LSE on the set $S_V^\uparrow$.
We will need to calculate the local entropy of the set $S_V^\uparrow$. Take any point $\mu^* \in S_V^\uparrow$. We need to pack at a distance $\varepsilon/c$ the set $S_{\mu^*}^\uparrow(V) := \{\mu: \mu \in S_V^\uparrow, \|\mu - \mu^*\|\leq \varepsilon\}$. Pick $\mu\in S_{\mu^*}^\uparrow(V)$. Write $\mu_1 - \mu_1^* =: \delta$  and suppose that $\delta < -V$. Using this assumption along with the fact that $\mu^{\ast},\mu\in S_V^{\uparrow}$, for any $i$, we have \[\mu_i^{\ast}\ge \mu_1^* = \mu_1 - \delta > \mu_1 + V \geq \mu_n \ge\mu_i\ge \mu_1.\] Hence for each $i$, \begin{align*}
    \mu_i^{\ast} - \mu_i &= \underbrace{\mu_i^{\ast} - \mu_1^{\ast}}_{\ge 0}+\underbrace{\mu_1^{\ast}-\mu_1}_{=-\delta}+\underbrace{\mu_1-\mu_i}_{\ge -V} \\
    &\ge -\delta-V  > 0.
\end{align*} Thus,
\begin{align*}
    \varepsilon^2 \geq \sum_{i \in [n]} (\mu_i - \mu_i^*)^2 \geq n(-\delta - V)^2.
\end{align*} Therefore $-\delta \leq V + \varepsilon/\sqrt{n}$, i.e., $\delta \ge -V-\varepsilon/\sqrt{n}$. On the contrary, if our assumption that $\delta<-V$ does not hold, we have $\delta\geq-V \ge -V-\varepsilon/\sqrt{n}.$ So in either case, $\delta \ge -V-\varepsilon/\sqrt{n}$.

Hence $\mu_1^* -\mu_1 =-\delta \leq V + \varepsilon/\sqrt{n}$.
A similar argument shows that $\mu_n \leq \mu_n^* + V + \varepsilon/\sqrt{n}$. Thus \[\mu_n-\mu_1 \le \mu_n^*-\mu_1^*+2V+\frac{2\varepsilon}{\sqrt{n}}\le 3V+\frac{2\varepsilon}{\sqrt{n}}.\] By \citet[Lemma 4.20]{chatterjee2014new} we therefore have that 
\begin{align}
    \log \cM(\varepsilon/c^*, S_{\mu^*}^\uparrow(V)) \lesssim \frac{(V\sqrt{n} + \varepsilon)}{(\varepsilon/c^*)} \asymp \frac{V\sqrt{n}}{\varepsilon}. \label{eq:univariate:isotonic:upper}
\end{align} 
Since the above holds for any $\mu^*$ it is an upper bound on the local entropy of $S_V^\uparrow$ by \eqref{yb:local:entropy:bound}. 

Next we will show a lower bound on $\log \cMloc_{S_V^\uparrow}(\varepsilon)$. To this end, we define the set $\cS^\uparrow(a,b) = \{\mu: a\leq \mu_1 \leq \ldots \leq \mu_n \leq b\}$. We are specifically interested in $S^{\uparrow}(0, V) = \{\mu: 0\leq \mu_1 \leq \ldots \leq \mu_n \leq V\}$. First, we note that by \citet[Lemma 4.20]{chatterjee2014new}, $\log \cM(\varepsilon, \cS^\uparrow(a,b)) \lesssim \frac{\sqrt{n}(b-a)}{\varepsilon}$. On the other hand, one can show using Varshamov-Gilbert's bound that $\log \cM(\varepsilon, \cS^\uparrow(a,b)) \gtrsim \frac{\sqrt{n}(b-a)}{\varepsilon}$ for values of $\varepsilon \gtrsim \frac{(b-a)}{\sqrt{n}}$. Formally we have:
\begin{lemma}\label{isotonic:lower:bound} We have that $\log \cM(\varepsilon, \cS^\uparrow(a,b)) \gtrsim \frac{\sqrt{n}(b-a)}{\varepsilon}$, for $\varepsilon \gtrsim \frac{(b-a)}{\sqrt{n}}$.
\end{lemma}

Because the metric entropy is non-increasing, for any $\varepsilon\lesssim \tfrac{b-a}{\sqrt{n}}$, we have \begin{align*}
    \log \cM(\varepsilon, \cS^\uparrow(a,b)) &\ge  \log \cM\left(\frac{b-a}{\sqrt{n}}, \cS^\uparrow(a,b)\right)\\ &\gtrsim\frac{\sqrt{n}(b-a)}{(b-a)/\sqrt{n}} \gtrsim n.
\end{align*} But also $\log \cM(\varepsilon, \cS^\uparrow(a,b))\lesssim n$ trivially holds always \citep[Example 5.8]{wainwright2019high}. So for all $\varepsilon\lesssim \tfrac{b-a}{\sqrt{n}}$, we have $\log \cM(\varepsilon, \cS^\uparrow(a,b))\asymp n$, while for all $\varepsilon\gtrsim \tfrac{b-a}{\sqrt{n}}$, $\log \cM(\varepsilon, \cS^\uparrow(a,b))\asymp \tfrac{\sqrt{n}(b-a)}{\varepsilon}$. Hence by using bound \eqref{yb:local:entropy:bound} one can see that for sufficiently large $c^*$, $\log \cMloc_{\cS^\uparrow(a,b)}(\varepsilon) \asymp \frac{\sqrt{n}(b-a)}{\varepsilon}$ for $\varepsilon \gtrsim \frac{(b-a)}{\sqrt{n}}$, and $\log \cMloc_{\cS^\uparrow(a,b)}(\varepsilon) \asymp n$ for $\varepsilon \lesssim \frac{(b-a)}{\sqrt{n}}$.  
 
 Taking $a=0$ and $b=V$, it follows that for $\varepsilon \gtrsim V/\sqrt{n}$, using $S^{\uparrow}(0,V)\subset S_V^{\uparrow}$, we have  $\sqrt{n}V/\varepsilon \lesssim \cM_{S^{\uparrow}(0,V)}^{\loc}(\varepsilon) \le \cM_{S_V^{\uparrow}}^{\loc}(\varepsilon)$, and we already argued from \eqref{eq:univariate:isotonic:upper} that $\cM_{S_V^{\uparrow}}^{\loc}(\varepsilon)\lesssim \sqrt{n}V/\varepsilon$. Hence $\cM_{S_V^{\uparrow}}^{\loc}(\varepsilon)\asymp \sqrt{n}V/\varepsilon$ for $\varepsilon \gtrsim V/\sqrt{n}$. 
 
 Let us apply Theorem \ref{important:thm} to upper bound $\epsLSE(\sigma)$. Consider the function \[\Psi\colon \varepsilon \mapsto \sup_{\delta \leq \varepsilon} \delta\sqrt{\log \cMloc_{\cS^\uparrow_V}(\delta)},\] noting that $\Psi(\varepsilon) \lesssim \min(\sqrt{\varepsilon \sqrt{n} V}, \sqrt{n}\varepsilon)$ for values of $\varepsilon \ge V /\sqrt{n}$ using \eqref{eq:univariate:isotonic:upper} and the trivial bound of $n$. Equating this to $\varepsilon^2/(2\sigma)$, we obtain $\varepsilon \asymp n^{1/6}V^{1/3}\sigma^{2/3} \wedge \sqrt{n}\sigma$. This matches the minimax rate that one can obtain from the equation $\log \cMloc_{S_V^{\uparrow}}(\varepsilon) \asymp \frac{\varepsilon^2}{\sigma^2}$ for $\varepsilon \geq V/ \sqrt{n}$. 
 
 Thus we conclude that the LSE is minimax optimal (up to logarithmic factors). We note that it is known that the LSE is exactly minimax optimal (see \citet{zhang2002risk, bellec2018sharp} for upper bounds and \citet[Corollary 5]{bellec2015sharp} for lower bounds). 

\subsubsection{Multivariate Isotonic Regression} \label{section:multivariate:isotonic:optimal}

We now demonstrate the known result that the LSE for multivariate isotonic regression (with known total variation) is optimal up to logarithmic factors. We will use a combination of existing work from \citet{isotonic_general_dimensions} on the Gaussian width, \citet{gao_Wellner_monotone} on packing numbers, and Remark \ref{remark:upper:bound:wK}. 

Let us introduce the set-up from \citet{isotonic_general_dimensions}. Let $\LL_{p,n}= \prod_{j=1}^p \{\frac{1}{n^{1/p}},\frac{2}{n^{1/p}},\dots, 1\}$ be the lattice with $(n^{1/p})^p=n$ points in $\RR^p$. For $p=1$, these are just the equispaced points from $1/n$ to $1$ (inclusive) of distance $n^{-1}$ apart. We will consider $p$ to be a fixed constant that does not scale with $n$.

Let \[\cF_p=\{f\colon [0,1]^p \to[0,1], f \text{ non-decreasing in each variable}\},\] where by non-decreasing in each variable we mean that \begin{align*}
    \MoveEqLeft f(z_1,\dots,z_{i-1},z_i,z_{i+1},\dots, z_p)\\  &\le f(z_1,\dots,z_{i-1}, z_i+y_i,z_{i+1},\dots, z_p)
\end{align*} for any $(z_1,\dots,z_p)\in [0,1]^p$, $y_i> 0$, and $1\le i \le p$. We are interested in the set of evaluations of monotone function on the lattice, i.e., \begin{align*}
    \MoveEqLeft Q_{a,b} =\{(f(l_1), f(l_2),\dots, f(l_n))\mid \\ &\qquad f\colon[0,1]^p\to[a,b] \text{ non-decreasing in each variable} \} 
\end{align*} where $l_1,\dots,l_n$ are the distinct elements of $\LL_{p,n}$. For $a=0$ and $b=1$, $Q_{0,1}$ is the set of tuples of  evaluations of functions in $\cF_p$.

Let us compute an upper bound on the log metric entropy of $Q_{a,b}$ by using the log-entropy of $\cF_p$, as is done in \citet[Lemma 4.20]{chatterjee2014new} in a univariate case. We are using the packing numbers which are up to absolute constants of the same order as the covering numbers used in \citet{chatterjee2014new} and \citet{gao_Wellner_monotone}. Here we use the $L^2$-norm with Lebesgue measure in $p$-dimensions in the metric entropy calculations for $\cF_p$.

\begin{lemma}[{\cite[Theorem 1.1]{gao_Wellner_monotone}}] \label{gao_Wellner_theorem} There is an absolute constant $C$ such that if $p>2$, we have \[\log \cM(\varepsilon, \cF_p) \le C\varepsilon^{-2(p-1)}.\] If $p=2$, then  \[\log \cM(\varepsilon, \cF_p) \le C\varepsilon^{-2}(\log 1/\varepsilon)^2.\]
\end{lemma} 
 
Using this result, we will obtain the following: \begin{lemma} \label{multivariable_isotone_upper} Let $Q_{a,b}$ be the set of evaluations of monotone functions on the lattice defined above. Then for $p>2$, we have \[\log \cM(\varepsilon, Q_{a,b}) \le C\left(\frac{\varepsilon}{2\sqrt{n}(b-a)}\right)^{-2(p-1)}\] for some absolute constant $C$. For $p=2$, we have \[\log \cM(\varepsilon, Q_{a,b}) \le C\left(\frac{\varepsilon}{2\sqrt{n}(b-a)}\right)^{-2}\left(\log \frac{2\sqrt{n}(b-a)}{\varepsilon}\right)^2.\]
\end{lemma} The proof is broadly similar to \citet{chatterjee2014new} except instead of constructing monotone step functions on the real-line, we construct a multivariate analogue on $[0,1]^p$. We will upper bound the cardinality of a minimal $\varepsilon$-covering of $Q_{a,b}$, which therefore bounds the cardinality of a maximal $\varepsilon$-packing of $Q_{a,b}$.

Next, we derive in Lemma \ref{multivariable_isotone_lower} a lower bound on $\log \cM(\varepsilon, Q_{a,b})$ following the technique in \citet[Proposition 2.1]{gao_Wellner_monotone}. The proof involves partitioning $[0,1]^p$ into cubes of side length $\varepsilon$ and constructing many piece-wise monotone functions whose values disagree with each other on sufficiently many cubes (equivalently, binary strings in some dimension that have sufficiently large Hamming distances between each other).

\begin{lemma} \label{multivariable_isotone_lower}  Let $Q_{a,b}$ be the set of evaluations of monotone functions on the lattice defined above. Then for $p\ge 2$, if $\frac{\varepsilon}{b-a} \gtrsim \sqrt{n/n^{1/p}}$, we have $\log \cM\left(\varepsilon, Q_{a,b}\right) \gtrsim \left(\frac{\varepsilon}{2\sqrt{n}(b-a)}\right)^{-2(p-1)}.$ 
\end{lemma}

To summarize the results of these two lemmas, when $\frac{\varepsilon}{b-a} \gtrsim \sqrt{n/n^{1/p}}$, we have for $p\ge 2$ that \begin{align} \label{eq:global:entropy:multivar:isotone}
    \MoveEqLeft \left(\tfrac{\varepsilon}{2\sqrt{n}(b-a)}\right)^{-2(p-1)} \\ &\lesssim \log \cM\left(\varepsilon, Q_{a,b}\right) \notag\\ &\lesssim \begin{cases}
        \left(\frac{\varepsilon}{2\sqrt{n}(b-a)}\right)^{-2}\Big(\log \frac{2\sqrt{n}(b-a)}{\varepsilon}\Big)^2, & p =2, \\ \left(\frac{\varepsilon}{2\sqrt{n}(b-a)}\right)^{-2(p-1)}, & p > 2. 
    \end{cases}\notag
\end{align} In the following lemma, we use the bound \eqref{yb:local:entropy:bound} to derive the local metric entropy from the global entropy results when $p>2$. 

\begin{lemma} \label{lemma:multivariate:isotonic:local:metric} Suppose $\frac{\varepsilon}{b-a} \gtrsim \sqrt{n/n^{1/p}}$ and $p>2$. Then for sufficiently large $c^{\ast}$, $\log \cM_{Q_{a,b}}^{\loc}\left(\varepsilon\right) \asymp \left(\frac{\varepsilon}{2\sqrt{n}(b-a)}\right)^{-2(p-1)}$. 
\end{lemma}

Using this result, we show that the least squares estimator achieves the minimax rate up to logarithmic factors. Consider the set $Q_{a,b}$ where $a=-b=\frac{1}{\sqrt{n}}$, and take $\sigma \ge \frac{1}{\sqrt{n}}$.
Then we have for $\frac{1}{n^{1/2p}} \lesssim \varepsilon \lesssim 1 $ (noting that the packing set has cardinality 1 when $\varepsilon$ is larger than the diameter) and $p>2$ that \begin{equation} \label{eq:multivariate_isotonic_metric_entropy}
    \log \cM_{Q_{-1/\sqrt{n}, 1/\sqrt{n}}}^{\loc}\left(\varepsilon\right) \asymp \varepsilon^{-2(p-1)}.  
\end{equation}

\begin{lemma} \label{lemma:multivariate:isotone:LSE:optimal} Suppose $p>2$. Take $1/\sqrt{n}\le \sigma \le 1$. Then minimax rate $\varepsilon^{\ast}=\sup\{\varepsilon:\varepsilon^2/\sigma^2\le\log \cM_{Q_{-1/\sqrt{n}, 1/\sqrt{n}}}^{\loc}\left(\varepsilon\right)\}$ satisfies $\varepsilon^{\ast}\asymp\sigma^{1/p}$. If $\sigma = \frac{1}{\sqrt{n}}$, then the LSE rate satisfies $\epsLSE \lesssim \sigma^{1/p}\log^2(n).$ 
\end{lemma}
 
Thus, we have shown for $p>2$ the LSE achieves the minimax rate up to a logarithmic factor when $\sigma = \frac{1}{\sqrt{n}}$, as already demonstrated in \citet{isotonic_general_dimensions}. We later show in Section \ref{section:suboptimal:multivariate} that if $\sigma > \frac{1}{\sqrt{n}}$, the LSE is suboptimal for some values of $\sigma$. In an improvement to the LSE, \citet{deng_and_zhang_2020} propose a computationally-tractable block estimator that actually achieves the minimax rate for a range of $\sigma$ without the logarithmic factor under some moment conditions on the noise term.

\subsubsection{Hyperrectangle Example} \label{subsubsection:hyperrectangle}

We will first argue that the LSE is optimal on any hyperrectangle. This appears to be a folklore fact, but since we could not find a reference containing a proof we attach an argument for completeness. Since it is hard to evaluate the local widths for arbitrary hyperrectangles we use other means. First observe that without loss of generality we may assume that the hyperrectangle is axis aligned. This is because the LSE is invariant to shifts and rotations. Let $H := \prod_{i = 1}^n [-a_i/2,a_i/2] $ for some  $a_1,\dots,a_n \in \RR^+$, where we assume without loss of generality that $a_1 \leq \ldots \leq a_n$.

First observe that the minimization $\sum_{i = 1}^n (Y_i - \nu_i)^2$ subject to $\nu \in H$ splits into minimizations $\hat \nu_i := \min_{\nu_i \in [-a_i/2,a_i/2]} (Y_i - \nu_i)^2$. In the one dimensional case with convex set $K = [-a_i/2, a_i/2]$, the estimator defined by \citet{neykov2022minimax} coincides with $\hat \nu_i$ (see Lemma \ref{LSE:is:minimax:in:R1}). It follows from \citet[Corollary III.1]{neykov2022minimax} (see also \citet{donoho1990minimax}) that $\EE (\hat \nu_i - \mu_i)^2 \asymp \min(a_i^2,\sigma^2)$, a fact which is also easy to directly verify. Thus the LSE has risk
\begin{align} 
    \EE \|\hat \nu - \mu\|^2 &= \sum_{i = 1}^n \EE (\hat \nu_i - \mu_i)^2 \asymp \sum_{i = 1}^n \min (a_i^2,\sigma^2) \notag \\ &\asymp [(k+2)\sigma^2 \wedge \sum_{i = 1}^n a_i^2], \label{eq:minimax:rate:hyperrectangle}
\end{align}
where $k \in \{0,n-1\}$ is such that $(k + 1)\sigma^2 \leq \sum_{i = 1}^{n-k}a_i^2$ and $(k + 2)\sigma^2 \geq \sum_{i = 1}^{n-k-1}a_i^2$. The last implication is proved in the following lemma. But \citet[Section III.A]{neykov2022minimax} establishes $[(k+2)\sigma^2 \wedge \sum_{i = 1}^n a_i^2]$ as the minimax rate in $n$ dimensions, proving optimality of the LSE. Note that we did not directly apply Corollary III.1 to the $n$-dimensional optimization since we would have to establish the estimator from \citet{neykov2022minimax} in $\RR^n$ coincides with the LSE, which may not be true.

\begin{lemma} \label{hyperrectangle:optimal:lemma} We have $\sum_{i = 1}^n \min (a_i^2,\sigma^2) \asymp [(k+2)\sigma^2 \wedge \sum_{i = 1}^n a_i^2]$ where $k$ is as defined above.
\end{lemma}

Consider now the hyperrectangle $H$ defined by taking $a_i = 1/\sqrt{n}$ for $ 1\le i \leq n-1$ and $a_n = C \gg 1$. Let $c^{\ast}>1/2$ be the sufficiently large absolute constant in \citet{neykov2022minimax} and then consider the Gaussian width of $B(0,2c^{\ast})\cap H$, i.e., $w_{H,0}(2c^{\ast})=\EE \sup_{x \in H, \|x\|\leq 2c^*} \langle x, \xi \rangle$. Taking $x_i = \sign(\xi_i)/(2\sqrt{n})$, we see that the width is $\gtrsim \sqrt{n}$. On the other hand, consider packing $B(0,2c^{\ast})\cap H$ at a distance $2c^*/c^* = 2$, i.e., computing $\cM_H^{\loc}(2c^{\ast})$.  Observe that the intersection of the ball and $H$ is fully contained in a hyperrectangle $H'$ of side lengths $n^{-1/2} \times n^{-1/2} \times \cdots n^{-1/2}\times 4c^*$, so $\cM_H^{\loc}(2c^{\ast})\le \cM_{H'}^{\loc}(2c^{\ast})$. Partition $H'$ into hypercubes of side length $1/\sqrt{n}$ along the long edge of $H'$, with possibly a region left-over of diameter $\le 1$.  We can fit at most $\lceil 4c^*/(1/\sqrt{n}) \rceil \le 4c^* \sqrt{n} + 1$ hypercubes inside $H'$. Note that each of these hypercubes of side length $1/\sqrt{n}$ can have at most $1$ point from the packing set since its diameter is bounded by $1$. This means $\log \cM_H^{\loc}(2c^{\ast})\le \log \cM_{H'}^{\loc}(2c^{\ast})\lesssim \log n$. Thus $w_{H,0}(2c^{\ast})/2c^{\ast}
\gtrsim \sqrt{n}\gg \sqrt{\log n} \gtrsim \sqrt{\log \cM_{H}^{\loc}(2c^{\ast})}$, noting that $2c^{\ast}$ is smaller than the diameter since we take $C\gg 1$. Thus, the sufficient condition from Corollary \ref{cool:corollary} is not necessary for LSE optimality.

The same example can serve to show that the condition of Proposition \ref{proposition:first:sufficient:condition} is only a sufficient condition. Suppose $a_n = C \geq \sqrt[4]{n}$ and $\sigma = 1$.  Since for $\varepsilon$  of constant order the width is at least proportional to $\sqrt{n}$, we can show $\bar \varepsilon \gtrsim \sqrt[4]{n}$. To see this, take $\varepsilon\asymp \sqrt[4]{n}$, so \[\frac{\varepsilon^2}{2\sigma}\asymp \sqrt{n}\lesssim w_{H,0}(1)\le w_{H,0}(\varepsilon)\le \sup_{\mu\in H}w_{H,\mu}(\varepsilon).\] Then, $\overline\varepsilon\gtrsim \varepsilon\asymp\sqrt[4]{n}$ by definition as a supremum. Since $d\ge a_n \ge \sqrt[4]{n}$, we have $\overline\varepsilon\wedge d\gtrsim\sqrt[4]{n}$. Hence Proposition \ref{proposition:first:sufficient:condition} gives a prediction which is far away from the true LSE rate which is constant in this case, using \eqref{eq:minimax:rate:hyperrectangle}.  

\subsubsection{Subspace (Linear Regression)} 
\label{subsubsection:subspace:linear}

Suppose we are given a linear regression model $Y = X\beta + \varepsilon$, where $\varepsilon \sim \cN(0,\sigma^2\II_p)$ and $X$ is an arbitrary fixed design matrix $X \in \RR^{n \times p}$. This is a special case of our setting where we can consider $K = \operatorname{col}(X)$ to be a fixed $p$-dimensional subspace in $\RR^n$. It is very simple to see that $w_{\mu}(\varepsilon) = w_{\nu}(\varepsilon)$ for any $\mu,\nu \in K$, so that the LSE will be minimax optimal by Corollary \ref{corollary:Lipschitz} (or Theorem \ref{difference:of:local:widths:compared:to:eps:squared:thm}). This is a significant albeit well-known result: for any fixed design $X$, the least squares procedure (i.e., linear regression) is optimal in terms of in-sample squared prediction error as long as $p \leq n$. By in-sample prediction error, we mean $\EE_\varepsilon \|X(\hat \beta - \beta)\|_2^2$, where we assume $\mu = X\beta \in \operatorname{col}(X)$, and the least squares estimate $\hat \mu \in \operatorname{col}(X)$ is $\hat \mu = X \hat \beta$ where $\hat \beta \in \argmin_{\gamma \in \RR^p} \|Y- X\gamma\|_2^2$. To see how this result can be derived using other means, one can consult Exercise 13.2, Example 13.8,  and Example 15.14 in \citet{wainwright2019high}.

\subsubsection{\texorpdfstring{$\ell_1$}{} ball and \texorpdfstring{$\ell_2$}{} balls: LSE is optimal} \label{subsubsection:l_1:l_2}

Let $K=\{x\in\RR^n:\|x\|_1\le 1\}$ be  the $\ell_1$ ball. Then $w_{K,0}(\varepsilon) \asymp \sqrt{\log(en(\varepsilon^2\wedge 1))} \wedge \varepsilon\sqrt{n}$ as stated in  \citet{bellec_2019_gaussian_widths}. The following lemma states the local entropy and is proven by combining Lemma \ref{lemma:yang:barron} along with the global entropy. 

\begin{lemma} \label{lemma:local:metric:entropy:ell_1} The local metric entropy of the $\ell_1$ ball satisfies  \begin{align*}
   \log \cMloc(\varepsilon)  \asymp \begin{cases}
    \frac{\log (\varepsilon^2 n)}{\varepsilon^2} & \varepsilon \gtrsim 1/\sqrt{n} \\
    n & \varepsilon \lesssim 1/\sqrt{n} \text{ or } \varepsilon\asymp 1/\sqrt{n},
\end{cases}
\end{align*} provided $c^{\ast}$ is taken sufficiently large.
\end{lemma}

We now apply Corollary \ref{cool:corollary} by verifying that $w_{K,0}(\varepsilon)/\varepsilon \lesssim \sqrt{\log \cMloc(\varepsilon)}$ using Lemma \ref{lemma:local:metric:entropy:ell_1}. When $\varepsilon\gtrsim 1/\sqrt{n}$, \begin{align*}
    \frac{w_{K,0}(\varepsilon)}{\varepsilon} &\asymp \sqrt{n}\wedge \frac{\sqrt{\log(en(\varepsilon^2\wedge 1))}}{\varepsilon} \lesssim \sqrt{n}\wedge \frac{\sqrt{\log(n\varepsilon^2)}}{\varepsilon} \\ &\lesssim \frac{\sqrt{\log(n\varepsilon^2)}}{\varepsilon} \asymp \sqrt{\log \cMloc(\varepsilon)}.
\end{align*} When $\varepsilon\lesssim 1/\sqrt{n}$,  \begin{align*}
   \frac{w_{K,0}(\varepsilon)}{\varepsilon} &\asymp \sqrt{n}\wedge \frac{\sqrt{\log(en(\varepsilon^2\wedge 1))}}{\varepsilon}\lesssim \sqrt{n} \\ &\asymp \sqrt{\log \cMloc(\varepsilon)}.
\end{align*} Thus the LSE is minimax optimal for all $\sigma$.

Next, we let $K=\{x\in\RR^n:\|x\|_2\le 1\}=B(0,1)$ be  the $\ell_2$ ball. We consider $\varepsilon< 1$. Observe from \citet[Proposition 7.5.2(vi)]{vershynin2018high} that \begin{align*}
    \frac{w_{K,0}(\varepsilon)}{\varepsilon} &= \frac{w(B(0,\varepsilon)\cap K)}{\varepsilon} = \frac{w(B(0,\varepsilon))}{\varepsilon} \\ &\lesssim \frac{\varepsilon\cdot \sqrt{n}}{\varepsilon}=\sqrt{n}.
\end{align*} Then using scaling properties of the metric entropy of a set along with \citet[Corollary 4.2.13]{vershynin2018high}, \begin{align*}
    \log \cM(\varepsilon/c^{\ast}, B(0,\varepsilon)\cap K) &= \log \cM(\varepsilon/c^{\ast}, B(0, \varepsilon)) \\ &= \log \cM\left(1/c^{\ast}, B(0,1)\right) \\ &\gtrsim n\log c^{\ast}.
\end{align*} Thus, $\frac{w_{K,0}(\varepsilon)}{\varepsilon}\lesssim \sqrt{ \log \cM(\varepsilon/c^{\ast}, B(0,\varepsilon)\cap K)}$. Hence 
 by Corollary \ref{cool:corollary} (using the symmetry of $K$) the LSE is minimax optimal for all $\sigma$.

\subsection{Examples with suboptimal LSE}
\label{subsection:example:suboptimal:LSE}

Here we attach several examples to illustrate the suboptimality of the LSE. Examples of this nature have previously appeared in \citet{chatterjee2014new, zhang2013nearly}. Below, when we say that the LSE is suboptimal, we mean that there exist values of $\sigma$ for which the LSE does not achieve the minimax optimal rate. It is easy to see (Proposition \ref{proposition:extreme:values:sigma:optimal}) that the LSE is always optimal for either very small or large values of $\sigma$.

\begin{proposition} \label{proposition:extreme:values:sigma:optimal} Let $r$ be the largest radius of a ball fully embedded in $K$. Then the LSE is minimax optimal if either $\sigma\lesssim r/\sqrt{n}$ or $\sigma\gtrsim d$.
\end{proposition}
    \begin{proof} First assume $\sigma\lesssim r/\sqrt{n}$. Suppose the ball with maximal radius $r$ is centered at $\mu\in K$.
        Take $\varepsilon = \sqrt{n}\sigma\lesssim r$. Then $\varepsilon^2/\sigma^2 = n$, while \begin{align*}
            \log \cMKloc(\varepsilon) &\gtrsim \log \cMloc(r) \ge \log \cM(r/c^{\ast}, B(\mu, r)\cap K) \\ &= \log \cM(r/c^{\ast}, B(\mu, r))  \gtrsim n,
        \end{align*} following the argument we used in Section \ref{subsubsection:l_1:l_2} for the $\ell_2$ ball. Hence $\varepsilon^2/\sigma^2\lesssim \log \cMKloc(\varepsilon)$, so $\varepsilon^{\ast}\gtrsim \varepsilon = \sqrt{n}\sigma$ by \eqref{varepsilon:star:def}. But $\sqrt{n}\sigma$ is the rate for the trivial estimator that just returns $Y$ (see also the proof of Corollary \ref{corollary:geometric_average:minimax}). The LSE is also always better than the trivial estimator (projections onto convex sets move closer to any point in the set), so $\epsLSE\lesssim \sqrt{n}\sigma$. Thus, the LSE is minimax optimal for $\sigma \lesssim r/\sqrt{n}$.

        Lastly, assume $\sigma\gtrsim d$. By Lemma \ref{minimax:bound:versus:sigma:and:d}, $\varepsilon^{\ast}\gtrsim \sigma\wedge d\asymp d$ but also  $\epsLSE\le d$ as desired.
    \end{proof}

\subsubsection{Pyramid Example}

Assume $\sigma = 1$. Let $v$ be an orthogonal vector to the convex set $K\subset v^\perp$, and assume $w(K) \geq \diam^2(K)\ge 1$, and $\|v\|_2 \geq \diam(K) + c$. Assume further that $K$ is symmetric (i.e., if $k\in K$, $-k\in K$). Consider the set $P=\cup_{\alpha \in [0,1]}[\alpha v + (1-\alpha)K]$. It is simple to see $P$ is a convex set, in fact a pyramid, using the symmetry assumption. 

\begin{lemma} \label{lemma:pyramid:suboptimal} For sufficiently large $c$ and assuming $w(K)\gtrsim \|v\|_2^2$  for a sufficiently large constant, the worst case LSE risk $\epsLSE[P]^2$ satisfies $\epsLSE[P] \gtrsim \|v\|_2$.
\end{lemma}

 On the other hand, there exists a simple estimator that achieves a better rate of convergence. Suppose $Y=p+\xi$ where $p=\alpha v+(1-\alpha)k\in P$ and $\xi\sim \cN(0,\sigma^2\II_n)$. Consider the linear projection $P_v Y$ of $Y$ onto $v$ (i.e., $P_v = v v^T/\|v\|^2$), so that $P_vY = \alpha v+P_v\xi$ using orthogonality of $v$ to $K$. Then
\begin{align*}
    \EE\|P_vY-p\|_2^2 &=\EE \| \alpha v+P_v\xi - (\alpha v + (1-\alpha)k)\|_2^2 \\ &=\EE \| P_v\xi - (1-\alpha)k\|_2^2 \\ &\le \EE \|P_v\xi\|_2^2 + (1-\alpha)^2\diam(K)^2 \\ &\leq 1 + \diam(K)^2 \ll \|v\|^2_2. 
\end{align*} Note that we used $\sup_{k\in K} \|k\|_2^2 \le \diam(K)^2$, which holds since $0 \in \operatorname{relint} K$. 

\subsubsection{Multivariate Isotonic Regression with \texorpdfstring{$\sigma > 1/\sqrt{n}$}{}} \label{section:suboptimal:multivariate}

We return to the multivariate isotonic regression setting, where now we take $\sigma > 1/\sqrt{n}$ and $p> 2$ and show the LSE is suboptimal. Our result relies on analyzing the local Gaussian width at the 0 point (Lemma \ref{lemma:multivariate:isotonic:suboptimal:lower:bound} and Lemma \ref{lemma:multivariate:isotonic:suboptimal:upper:bound} in the appendix) by mimicking the proof of \citet[Proposition 5]{isotonic_general_dimensions}.  The authors consider a more general set than  $Q_{a,b}$, though, with no restrictions on the range.

\begin{lemma} Set $K= Q_{-1/\sqrt{n},1/\sqrt{n}}\subset\RR^n$ and set $p>2$. Then for \begin{align} \label{eq:suboptimal:isotonic:range}
    \sigma \in \left(\max\left\{n^{-\frac{1}{2}}, n^{\frac{2-p}{2p-2}}(\log n)^{\frac{4p}{p-1}} \right\},n^{-\frac{1}{2}+\frac{1}{p}}\right),
\end{align} the LSE rate $\epsLSE^2$ exceeds the minimax rate $\sigma^{2/p}$ .
\end{lemma}
  \begin{proof} Recall from Section \ref{section:multivariate:isotonic:optimal} that the minimax rate is given by $\sigma^{2/p}$ when $\frac{1}{\sqrt{n}} \lesssim \sigma \lesssim 1$. Now assume $\sigma$ satisfies \eqref{eq:suboptimal:isotonic:range}, assuming momentarily that the given interval is a non-empty interval. This clearly satisfies $\frac{1}{\sqrt{n}} \lesssim \sigma \lesssim 1$ condition.
  
    Next, using Lemma \ref{lemma:multivariate:isotonic:suboptimal:lower:bound} and Lemma \ref{lemma:multivariate:isotonic:suboptimal:upper:bound} and taking $c = \frac{1}{\log^4 n}\in(0,1)$, we conclude $w_0(t) - w_0(ct) \gtrsim  n^{1/2 - 1/p} t$ for $t\in(0,1)$. Since $\sigma \le n^{-1/2 + 1/p}$, we have $n^{1/2-1/p}\sigma<1$. Set $t \asymp n^{1/2-1/p}\sigma$ so that \[w_0(t) - w_0(ct) \geq \frac{t^2}{\sigma} \ge \frac{t^2 (1 - c^2)}{\sigma}.\]  Rearranging, \[w_0(t) - \frac{t^2}{\sigma}\geq w_0(ct) - \frac{t^2 c^2}{\sigma},\]
       which by the concavity of $t \mapsto w_0(t) - t^2/\sigma$ implies that  $\epswidth[0](\sigma)\ge ct $ (where $\epswidth[0]$ is defined as in \eqref{equation:varepsilon:mu}). Thus, since  $\sigma>1/\sqrt{n}$, it follows that \[\epswidth[0]\gtrsim ct \asymp cn^{1/2-1/p}\sigma >  \sigma .\] The last inequality follows for sufficiently large $n$ since $n^{1/2-1/p}\gg \log^4 n = c^{-1}$. 

Then, by Lemma \ref{lemma:chatterjee:analogue}, $\EE\|\hat{\mu}-0\|^2\asymp \epswidth[0]^2$, which implies $\epsLSE\gtrsim \epswidth[0]$. Since we showed $\epswidth[0]^2 \gtrsim c^2t^2$, we have $\epsLSE^2\gtrsim c^2 n^{1-2/p}\sigma^2$. This quantity is $\gg\sigma^{2/p}$ whenever $\sigma \gg n^{(2-p)/(2p-2)}(\log n)^{4p/(p-1)}$, as we required. Thus, the LSE is suboptimal for such $\sigma$.

It remains to check that the interval in \eqref{eq:suboptimal:isotonic:range} is non-empty. One can verify that \begin{equation*} 
    n^{\frac{2-p}{2p-2}}(\log n)^{\frac{4p}{p-1}} \ll n^{-\frac{1}{2}+\frac{1}{p}},
\end{equation*} since this is equivalent to $(\log n)^{\beta_1} \ll n^{\beta_2}$ where $\beta_1=\frac{4p}{p-1} >0$ and $\beta_2 = \frac{p-2}{2p(p-1)}>0$ for $p>2$. Moreover, clearly $n^{-1/2}\ll n^{-\frac{1}{2}+\frac{1}{p}}$. Thus we have a non-trivial range for $\sigma$ with a suboptimal LSE.
    \end{proof}

\subsubsection{Solid of Revolution}

Let $f\colon[0,b]\to\RR$ be concave, satisfy $f(0)=f(b)=0$, and satisfy $f(x)=f(b-x)$ for $x\le b/2$ (i.e., symmetric about $b/2$). We assume $b/4\gg f(b/2)$ and that \begin{align}
    f(b/2)>f(b/4)+\tfrac{b}{4\sqrt{2\pi (n-1)}}, \label{eq:solid:revolution:secant}
\end{align} i.e., the secant line between the points at $x=b/4$ and $x=b/2$ has slope at least $\tfrac{1}{8\sqrt{2\pi(n-1)}}$. To define the solid of revolution in $\RR^n$, for any $x\in[0,b]$, define $B_x$ as the the $(n-1)$-dimensional ball with radius $f(x)$  and centered at the origin in $\RR^{n-1}$. Then let $K=\bigcup_x (\{x\}\times B_x)$.

\begin{lemma} \label{lemma:solid:revolution} The worst case LSE risk for the solid of revolution $K$ satisfies $\epsLSE\gtrsim b$.
\end{lemma}

On the other hand, given $y = \mu+\xi$ where $\mu \in K$, consider the estimator $\hat{y}=\langle y,e_1\rangle e_1$ that projects $y$ onto the $x$-axis. Then decomposing $\mu = \langle \mu, e_1\rangle e_1 + (\mu - \langle \mu, e_1\rangle e_1 )$, we have \begin{align*}
   \MoveEqLeft \|\mu - \langle y,e_1\rangle e_1\|_2^2 \\ &= \| (\langle \mu, e_1\rangle- \langle y, e_1\rangle) e_1 \|_2^2 + \|\mu - \langle \mu, e_1\rangle e_1\|_2^2\\
    &\le \| \langle\xi, e_1\rangle e_1 \|_2^2 + [f(b/2)]^2\\ 
    &= \langle \xi, e_1\rangle^2 + [f(b/2)]^2.
\end{align*} The second inequality came from noting that $\mu - \langle \mu, e_1\rangle e_1$ lies in $e_1^{\perp}\cap K$, and any such point is no more than distance $f(b/2)$ from the $x$-axis. So in expectation, we have $\EE \|\mu - \langle y,e_1\rangle e_1\|_2^2  \le 1+ [f(b/2)]^2\ll 1+ b^2/8.$ Thus, $\epsLSE^2\gtrsim b^2$ while  $\EE \|\mu - \hat{y}\|_2^2 \ll 1+b^2$.

\subsubsection{Prelude to Ellipsoids}
\label{some:general:thoughts:prelude:ellipse}

In this section, we derive a lower bound on the LSE rate for a broad class of sets $K$ that are pre-images of smooth functions, which in turn are transformations of Minkowski gauges. Later, we specialize our result to the case of ellipsoids.

Suppose we are given a convex set of the type $K = \{x: G(x) \leq 1\}\subseteq \RR^n$, where $G : \RR^n \mapsto \RR$ is a twice continuously differentiable non-negative convex function with $G(0) = 0$. Suppose that the Hessian of $G$ admits a lower bound on $K$, i.e., for all $x \in K$ we have $H(x) = \frac{\partial^2}{\partial y y\T} G(y) \big\vert_{y = x} \succeq M \succ 0$,  for some symmetric and strictly positive-definite matrix $M$ (so that $G$ is strongly convex). Take any boundary point $x \in \operatorname{bd} K$, i.e., such that $G(x) = 1$. We know that $\nabla G(x)\T x \geq \nabla G(x)\T y$ for all $y \in K$, or in other words, the gradient $\nabla G(x)$ is an outward normal to $K$. However, under our assumptions above, we actually know more about the set $K$. We know that 
\begin{align}
    G(y) &= G(x) + \nabla G(x)\T(y-x) + \frac{(y-x)\T H(\tilde x) (y-x)}{2}\notag \\ &\geq G(x)+ \nabla G(x)\T(y-x) + \frac{(x-y)\T M(x-y)}{2}, \label{taylor:expansion:convex:G}
\end{align}
where $\tilde x = \alpha x + (1-\alpha)y$ for some $\alpha \in [0,1]$. Since $x$ is a boundary point and $y\in K$, $G(y)-G(x)\le 0$, hence \eqref{taylor:expansion:convex:G} implies
\begin{align} \label{eq:strongly:convex:condition}
    \nabla G(x)\T(x-y) \geq (x-y)\T M (x-y)/2
\end{align}
for all $y \in K$. We will also argue that under these assumptions that the set $K$ is compact. It is clear by the continuity of $G$ that $K$ is closed. For boundedness, pick $y\in K$ and observe
\begin{align*}
    1 \geq G(y) \geq G(0) + \nabla G(0)\T(y) + y\T M y/2.
\end{align*}
But, since $0$ is a point of minimum, we have $G(0) = 0$ and $\nabla G(0) = 0$, so that $y^T M y/2 \le G(y)\le 1$. Thus, since $M$ is strictly positive-definite, the set $K$ is bounded. Denote by $d$ the finite diameter of $K$.

Suppose we now select the boundary point with the smallest gradient, i.e., let $x \in \operatorname{bd}K$ be a point such that $\|\nabla G(x)\|_2$ is minimized. This smallest gradient will not be a zero vector since \[0= G(0) \geq G(x) - \nabla G(x)\T x = 1 - \nabla G(x)\T x.\] Applying \eqref{eq:strongly:convex:condition}, there exists some outward normal vector at $x$ given by $x^* =  \frac{\nabla G(x)}{\|\nabla G(x)\|_2}$ such that for all $y\in K$ we have
\begin{align}\label{strong:convexity:separation}
    \langle x^*, x - y \rangle \geq (x-y)\T \tilde M (x-y)/2,
\end{align}
where $\tilde M = M / \|\nabla G(x)\|_2$.

A lower bound on $\epsSupwidth(\sigma)$ for suitable $\sigma$ is given in the following bound using these vectors $x$ and $x^{\ast}$, where recall $\epsSupwidth(\sigma)$ is the supremum over $\mu\in K$ of the term $\epswidth(\sigma)$ defined in \eqref{equation:varepsilon:mu}.

\begin{lemma} \label{lemma:ellipse:prelude} Let $\lambda=(\lambda_1,\dots, \lambda_n)$ be the vector of eigenvalues of $\tilde M/2$. Given the above assumptions,  $\epsSupwidth(\sigma) \gtrsim w^2(K)/\sum (1/\lambda_i)$ for $\sigma \gtrsim d^2/(w(K))$.
\end{lemma}

\subsubsection{A necessary condition for optimality on ellipsoids}
\label{subsubsection:optimality:ellipsoids}

We now specialize the scenario in the previous section to ellipsoids. That is, suppose we are given $K = \{x: \|D x\|^2_2 \leq 1\}$ with $D$ is a diagonal positive-definite matrix, i.e., we take $G(x) = \|D x\|^2_2$ in the notation of Section \ref{some:general:thoughts:prelude:ellipse}. Let the diagonal entries of $D$ be  given by $d_1 \geq d_2 \geq \ldots \geq d_n$. We again derive a lower bound on $\epsSupwidth(\sigma)$ and obtain a necessary condition for optimality of the LSE similar to the sufficient condition in Corollary \ref{cool:corollary}.

\begin{lemma}\label{ellipsoid:lemma}
    Suppose we have the ellipsoid $K = \{x \in \RR^n : \|D x\|_2^2 \leq 1\}$ where $D$ is a positive definite diagonal matrix with entries $d_1 \geq d_2 \geq \ldots \geq d_n$. Define $\delta_{n-k}=1/d_{n-k}$ for any $k\in  \{0\}\cup[n-1]$. Then for each such $k$, we have $\epsSupwidth(\sigma) \gtrsim \delta_{n-k}$ when $\sigma \gtrsim \delta_{n-k}^2/w_0(\delta_{n-k})$ for a sufficiently large absolute constant. Moreover, if the LSE is minimax optimal for $K$ for all $\sigma$, then for all $k \in\{0\}\cup[n-1]$ we must have 
    \begin{align*}
        w_0(\delta_{n-k})\lesssim \delta_{n-k} \sqrt{\log \cMloc(c \delta_{n-k})},
    \end{align*}
    for sufficiently large absolute constant and sufficiently small $c > 0$.
\end{lemma}

Let us now construct some examples of ellipsoids with a suboptimal LSE using this lemma. Take $\delta = 1/d_n \asymp \diam(K)$. By Lemma \ref{ellipsoid:lemma} we know that for values of $\sigma = 1/(d_n^2\sqrt{\sum 1/d_i^2})$ we have that $\epsSupwidth(\sigma) \gtrsim \delta$, where we used \citet[Exercise 5.9]{wainwright2019high} to write $w_0(\delta) = w(K) \asymp \sqrt{\sum 1/d_i^2}$. Note that
\begin{align*}
\delta \gtrsim \sigma = 1\Big\slash\left(d_n^2\sqrt{\sum 1/d_i^2}\right).
\end{align*} Then for some $\nu\in K$, $\epswidth[\nu](\sigma)\gtrsim \delta\gtrsim \sigma$, so by Lemma \ref{lemma:chatterjee:analogue}, \[\epsLSE(\sigma)\ge \EE\|\hat\mu-\nu\|^2\asymp \epswidth[\nu](\sigma)\gtrsim \delta\asymp\diam(K).\] But since $\epsLSE(\sigma)\lesssim\diam(K)$ always holds, this means $\epsLSE(\sigma)\asymp\delta\asymp\diam(K)$. On the other hand, the minimax rate for ellipsoids,  stated in \citet{neykov2022minimax} (where in his notation $a_i = 1/d_i^2$ and $d_0 = \infty$), is given by $(k + 1)\sigma^2 \wedge \diam(K)^2$ if $1/d_{n-k}^2 \leq (k + 1)\sigma^2$ but $1/d_{n-k+ 1}^2 > k \sigma^2$, and is $\diam(K)^2$ if $1/d_n^2 \leq \sigma^2$. 

Let us ensure that we choose $k$ and the $d_i$ so that the minimax rate is given by $(k+1)\sigma^2$ with $k=1$. It is clear that $\sigma < 1/d_n$. Suppose now that $d_1/\sqrt{n}\ll d_{n} \ll d_{n-1}$ and in addition $1/d_{n-1}^2 \le 2\sigma^2$ which is equivalent to $\sum \frac{1}{d_i^2} \le \frac{2d_{n-1}^2}{d_n^4}$. Then $k=1$ satisfies the stated conditions in the minimax rate in the first scenario. Note that $d_1/\sqrt{n}\ll d_n$ implies $2\sigma^2 \ll 1/d_n^2 \asymp\diam(K)^2 \asymp\delta^2$. It follows that the minimax rate would be at most $\frac{2}{d_n^4 \sum 1/d_i^2}= 2\sigma^2 \ll  \delta^2$, while we showed $\delta^2$ is the worst case rate for the LSE. 

One can construct multiple such examples. One example is given in \citet{zhang2013nearly} with $d_i = 1$ for $i < n$ and $d_n = 1/\sqrt[4]{n}$, where the lower bound on the worst-case rate of the LSE is established with a bare hands argument. Then the minimax rate is $2\sigma^2 \asymp 1$ for this example, while the worst-case rate for the LSE is $\epsLSE^2 \asymp \diam(K)^2 \asymp n^{1/2}$. Since here $w(K) \asymp \sqrt{n}$ and $\sigma \asymp 1$, we obtain an example where the bound $\epsLSE \lesssim \sqrt{\sigma w(K)}$ from Remark \ref{remark:upper:bound:wK} is tight. One can also check that $\varepsilon^{\ast}\asymp 1$, which means the bound from Corollary \ref{corollary:geometric_average:minimax} of $\epsLSE \lesssim \sqrt{\sigma} \sqrt{\varepsilon^*} C_n \sqrt[4]{n}$ is tight (without the logarithmic factors from $C_n$), as we mentioned in Remark \ref{remark:geometric_average:minimax}.

We provide another example which does not even require $d_{n-1}\gg d_n$. Consider the Sobolev type ellipsoid with $d_k = (n - k + 1)^{\alpha}$ for $0<\alpha < 1/2$. Then we can calculate $\sigma \asymp \sqrt{n^{-1 + 2\alpha}} \ll 1 \asymp \diam(K)$. The conditions $1/d_{n-k}^2 \leq (k + 1)\sigma^2$ but $1/d_{n-k+ 1}^2 > k \sigma^2$ are equivalent to $(k+1)^{1 + 2\alpha} \geq \sigma^{-2} \geq k^{1 + 2\alpha}$, so we take $k = \lceil n^{(1 - 2\alpha)/(1 + 2\alpha)}\rceil$. Hence the minimax rate is \begin{align*}
    \lceil n^{(1 - 2\alpha)/(1 + 2\alpha)}\rceil n^{-1 + 2\alpha} &\asymp n^{-2\alpha(1 - 2\alpha)/(1 + 2\alpha)} \\ &\ll 1\asymp \epsLSE(\sigma)^2.
\end{align*} This result complements that of \citet{wei2020gauss} who consider the $\alpha>1/2$ case by contrast and demonstrate optimality of the LSE. 

\subsubsection{\texorpdfstring{$\ell_p$}{} balls for \texorpdfstring{$p\in(1,2)$}{}}

In order to show the suboptimality of the LSE for the $\ell_p$ unit ball for a fixed $p \in (1,2)$, we first take a detour and prove a general bound on the LSE rate for a strongly convex body $K$. A body $K$ is called strongly convex if there exists a constant $k$ such that for any $\mu, \nu \in K$ and any $\lambda \in [0,1]$ we have that $B(\lambda \mu + (1- \lambda) \nu, k \lambda(1-\lambda) \|\mu - \nu\|^2) \subset K$. Let $\mu$ and $\nu$ be opposite ends of a diameter of $K$ and pick $\lambda=1/2$. Then, by strong convexity, $B((\mu+\nu)/2, kd^2/4)\subseteq K$ so that $2k d^2/4 < d$. Rearranging $k < 2d^{-1}$.

\begin{lemma} \label{lemma:strongly:convex:body} Let $K$ be a strongly convex body with parameter $k$. Suppose $\sigma \asymp (k\sqrt{n})^{-1}$. Then $\epsLSE(\sigma)\asymp d$. 
\end{lemma}

Returning to $\ell_p$ balls for $p \in (1,2)$, by \citet[Corollary 1]{garber2015faster}, we know that $\ell_p$ unit balls are $k = (p-1)n^{1/2 - 1/p}$ strongly convex. Hence for $\sigma \asymp 1/n^{1- 1/p}$, we have $\epsLSE\asymp d$ from Lemma \ref{lemma:strongly:convex:body}. Now, it is easy to verify that the diameter of the $\ell_1$ and $\ell_2$-balls are both 2, and since for any $p\in(1,2)$ we have that the unit $\ell_1$ ball is a subset of the unit $\ell_p$-ball which is a subset of the $\ell_2$-ball, it follows that $d\asymp 1$.

On the other hand, following Section III.E of \citet{neykov2022minimax} (formally that section  treats convex weak $\ell_p$ balls but since the packing numbers are the same it applies to $\ell_p$ balls as well) we know that the minimax rate satisfies $\varepsilon^{\ast}(\sigma)\asymp\sigma^{1- p/2} (\log n)^{(2-p)/4} \wedge d$ for values of $\sigma$ satisfying $\log (n \sigma^p(\log n)^{p/2}) \asymp \log n$ and $\sigma^{(4-2p)/4} (\log n)^{(2-p)/4} \gtrsim n^{1/2 - 1/p}$. 

Observe that for $\sigma \asymp 1/n^{1- 1/p}$ we have $\sigma^{1- p/2}( \log n)^{(2-p)/4} \ll 1 \asymp d$ which implies $\varepsilon^{\ast}(\sigma) \ll 1$ provided the other two conditions hold. To see this, first note that $\log n \ll n^{\alpha}$ for any $\alpha>0$. Thus, \begin{align*}
   \sigma^{1- p/2}( \log n)^{(2-p)/4} &\asymp n^{1/p-3/2+p/2} ( \log n)^{(2-p)/4} \\ &\ll n^{1/p-3/2+p/2+\alpha/2-\alpha p/4}.
\end{align*} Hence if we take $0<\alpha < \frac{1/p-3/2+p/2}{p/4-1/2}= \frac{2(p-1)}{p}$, then the right-hand side is of the form $n^{\beta}$ for $\beta<0$ and is thus $\ll 1\asymp d$. 

We must check that $\sigma\asymp1/n^{1-1/p}$ satisfies the other two conditions. First 
\begin{align*}
    \log (n \sigma^p(\log n)^{p/2}) = \log (n^{2-p} (\log n)^{p/2}) \asymp \log n.
\end{align*}
Next, using the fact that $(\log n)^{(2-p)/4} \gtrsim 1$ and substituting $\sigma\asymp1/n^{1-1/p}$, we have
\begin{align*}
    \sigma^{(4-2p)/4} (\log n)^{(2-p)/4} \gtrsim n^{1/p-3/2+p/2} \geq n^{1/2 - 1/p},
\end{align*}
where the final inequality follows since $p/2 + 2/p \geq 2$ for any $p>0$ (using the AM-GM inequality). This proves that $\varepsilon^{\ast}(\sigma)\ll 1$ while by the previous lemma, $\epsLSE(\sigma)  \asymp 1$. It follows that for $p\in(1,2)$, the LSE is suboptimal for $\ell_p$ balls for $\sigma\asymp1/n^{1-1/p}$.

\section{Discussion}

We have established numerous necessary or sufficient conditions for minimax optimality of the constrained least squares estimator in the convex Gaussian sequence model setting. Our techniques focused on the local behavior of the Gaussian width and metric entropy of the set. We then provided a series of examples where the LSE is minimax optimal or suboptimal for noise $\sigma$ chosen in an appropriate range. Our examples included isotonic regression in both one and many dimensions, hyper-rectangles, ellipsoids, and $\ell_p$ balls with $p\in[1,2]$. Our results also lead to theoretical algorithms (Appendix \ref{section:algorithm_proofs}) that bound the worst case LSE rate. 

Future work could consider extensions to handle estimation with sub-Gaussian noise, which may clash with our use of the Gaussian width following \citet{chatterjee2014new}. Moreover, our examples with $\ell_p$ balls could be extended for the $p>2$ case. While a formidable task, our examples of LSE-suboptimality underscore the need for a general algorithm to replace the LSE while remaining computationally tractable.

\bibliographystyle{abbrvnat}
\bibliography{lse_ref}

\clearpage

\appendix

\section{Algorithms searching for the worst case rate of the LSE on bounded sets} \label{section:algorithm_proofs}

 The results of Section \ref{subsection:characterizations:conditions:worst:case}, in particular Theorems \ref{big:width:minus:small:width:thm} and \ref{difference:of:local:widths:compared:to:eps:squared:thm}, inspire theoretical algorithms for searching for worst case rate for the LSE for bounded sets $K$. Our Algorithm \ref{algo_local_packing} below is based on local packings while Algorithm \ref{algo_global_packing} instead uses global packings. Since both algorithms make usage of evaluating the map $\nu \mapsto w_{\nu}(\varepsilon)$, before we introduce them we will illustrate how one can evaluate the map $\nu \mapsto w_\nu(\varepsilon)$ on a given closed convex body $K$, where it is assumed that we have a separation oracle for $K$. A separation oracle for $K$ is a function $\cO_K\colon\RR^n\to\RR^n$ such that $\cO_K(\nu) = 0$ if $\nu \in K$ and if $\nu\not\in K$,  $\cO_K(\nu) = a$ where $a \in \RR^n$  is such that $a\T \nu > a\T \mu$ for all points $\mu \in K$. The next lemma states that if one can compute the LSE for $K$, then assuming a separation oracle for $K$ is not a stringent assumption.
\begin{lemma}\label{sep:oracle:with:projection}
    If one can project on $K$, i.e., one can calculate $\Pi_K\nu$ for any $\nu \in \RR^n$, then one has a separation oracle on $K$.
\end{lemma}

\begin{proof}
Define $\cO_K(\nu) = \nu - \Pi_K \nu$ for all $\nu\in\RR^n$. Then $\cO_K(\nu)=0$ for all $\nu \in K$. If $\nu\not\in K$, observe that for any  $\mu \in K$ we have $(\nu - \Pi_K \nu)\T (\nu - \mu) > 0$. This follows by \citet[equation (1.19)]{bellec2018sharp} since \begin{align*}
    (\nu - \Pi_K \nu)(\mu - \nu) + \|\nu - \Pi_K \nu\|^2 &= (\nu - \Pi_K \nu)\T (\mu - \Pi_K(\nu)) \\ &\leq 0,
\end{align*} and also that $\|\nu - \Pi_K \nu\| > 0$ since $\nu\not\in K$.
\end{proof}

The idea for calculating the width $w_\nu(\varepsilon)$ is simple, and we summarize it in Algorithm \ref{algo_gaussian_width}. We will sample $N$ points $\xi_i \sim \cN(0,\II_n)$, $i \in [N]$. Assume for a moment we can calculate the values of $\sup_{\eta \in B(\nu,\varepsilon)\cap K} \langle \xi_i, \eta - \nu \rangle $. We will then average all these values as an estimate of $w_\nu(\varepsilon)$. This strategy works because of the celebrated concentration inequality of Lipschitz functions of Gaussian variables \citep[e.g.,][Theorem 2.26]{wainwright2019high}, noting that the map $x \mapsto \sup_{\eta \in B(\nu,\varepsilon)\cap K} \langle x, \eta -\nu\rangle $ is $\varepsilon$-Lipschitz \citep[see][Example 2.30]{wainwright2019high}. The concentration result states that $\langle \xi_i, \eta^*_i -\nu\rangle$ (where $\eta^*_i$ denotes the maximum of $\sup_{\eta \in B(\nu,\varepsilon)\cap K} \langle \xi_i, \eta  -\nu \rangle$) is a sub-Gaussian random variable with variance proxy $\varepsilon^2$, hence satisfies
\begin{align*}
    \PP\bigg(\big|N^{-1}\sum_{i \in [N]} \langle \xi_i, \eta^*_i  -\nu \rangle - w_\nu(\varepsilon)\big| \geq t\bigg) \leq 2 \exp\left(-\frac{Nt^2}{2\varepsilon^2}\right).
\end{align*}
Thus if one wants to approximate $ w_\nu(\varepsilon)$ to precision $t$ with probability at least $1 - \delta$, then one needs at most $N = 2\varepsilon^2\log(2/\delta)/t^2$ observations.

To solve $\sup_{\eta \in B(\nu,\varepsilon)\cap K} \langle \xi_i, \eta - \nu \rangle $ for a fixed $\xi$ given a separation oracle for $K$, we use the ellipsoid algorithm \citep[Chapter 3]{grotschel_ellipsoid}. At each iteration, we must either find a separation oracle for $B(\nu, \varepsilon)$ (which is trivial by Lemma \ref{sep:oracle:with:projection}), or a separation oracle for $K$. If the point $\eta$ happens to be in the set $B(\nu,\varepsilon) \cap K$ then one simply needs to take the gradient of $-\langle\xi, \eta -\nu \rangle$ with respect to $\eta$ which is also trivial. 

\begin{algorithm}
\SetKwComment{Comment}{/* }{ */}
\caption{Local Gaussian Width Algorithm}\label{algo_gaussian_width}
\KwInput{$K$ a compact convex set in $\RR^n$, a separation oracle $\cO_K$ for $K$, a point $\nu\in K$, $\varepsilon>0$, desired precision $t$ and probability $1-\delta$.}
Let $\cE$ be a subroutine that can solve $\arg\sup_{\eta\in B(\nu, \varepsilon)\cap K}\langle \xi,\eta-\nu\rangle$ given $\cO_K$\;
Set $N = \lceil 2\varepsilon^2 \log(2/\delta)/t^2\rceil$\;
Draw $\xi_1,\dots,\xi_N \sim \cN(0,\II_n)$ \;
Compute using $\eta_i^{\ast}=\arg\sup_{\eta\in B(\nu, \varepsilon)\cap K}\langle \xi_i,\eta - \nu\rangle$ for $i=1,\dots,N$ using $\cE$\;
Set $\widehat{w_{\nu}}(\varepsilon) = N^{-1}\sum_{i\in[N]}\langle \xi_i, \eta_i^{\ast} - \nu\rangle$\;
\Return{$\widehat{w_{\nu}}(\varepsilon)$}  
\end{algorithm}


\subsection{A local packing algorithm}

\begin{algorithm}[t]
\SetKwComment{Comment}{/* }{ */}
\caption{Local Packing Algorithm}\label{algo_local_packing}
\KwInput{Compact convex set $K$, diameter $d$,  constant $c^{\ast}$ from definition of local metric entropy}
\SetKwFunction{ChildrenDistance}{ChildrenDistance}
  \SetKwProg{Fn}{Function}{:}{}
  \Fn{\ChildrenDistance{$\nu\in K$, $k\in\NN$}}{
  $\delta\gets\frac{d}{2^{k-1}c^{\ast}}$\;
    Form a maximal $\delta$-packing set $\{\nu_1^k,\nu_2^k, \dots,\nu_M^k\}$ of the set $B(\nu,\delta c^{\ast})\cap K$ where $M=\cM(\delta, B(\nu,\delta c^{\ast})\cap K)$\; 
    Solve $\nu_i^{k\ast} = \argmin_i w_{\nu_i^k}(\delta)$\;
    $\Psi \gets w_{\nu}(\delta c^{\ast}) - w_{\nu_i^{k\ast}}(\delta)$\; 
    $T\gets C(\delta c^{\ast})^2/(2\sigma$)\;
    \KwRet{$\Psi - T$}\;
  }
  Initialize $\nu^{\ast}\in K$\;
  Set $\phi(\nu^{\ast})=1$ \tcc*{Track the level of any node $\nu$ in the tree}
  Let $Q=(\nu^{\ast})$ be an ordered queue\;
\While{$Q$ is not empty}{
    Remove the point $\nu$ at the end of the $Q$\;
    Compute $\gamma= $ \ChildrenDistance{$\nu$, $\phi(\nu)$} and the packing set therein\;
    Define $\beta=\dfrac{d}{2^{\phi(\nu)-1}c^*}$\;
    \uIf{$\gamma> 0$}{
    \Return{$\beta$}  
    }
    \Else{
    For each child node $\nu_i$ in the packing set, set $\phi(\nu_i)=\phi(\nu)+1$\;
    Add each $\nu_i$ to the start of $Q$\;
    }
}
\end{algorithm}


Suppose $K$ is a compact convex set with diameter $d$. Take $c^{\ast}>4$ in the definition of the local metric entropy, and set $C = 4 - \frac{1}{{c^{\ast}}^2}$. We now consider partitioning $K$ in the following way. Begin by fixing an arbitrary point $\nu^* \in K$ (its location is inconsequential). Maximally pack the set $B(\nu^*, d)\cap K = K$ at a distance $d/c^{\ast}$. Next, for each point  $\nu_i^*$ in that packing set, consider maximally packing the set $B(\nu_i^*, d/2) \cap K$ at a distance $d/(2c^{\ast})$. Continue the process infinitely, obtaining an infinite tree of packing sets. At the $k$th level, the number of descendants of a point from the $(k-1)$th level is bounded by $\cMloc(d/2^{k-1})$. 

Following this packing set construction, we proceed with our algorithm starting with level $k=1$. Algorithm \ref{algo_local_packing} considers the difference between $w(B(\nu^*, d/2^{k-1})\cap K)$ and the smallest possible point from the children of $\nu^*$, i.e., $\min_{i} w(B(\nu_i^*, d/(2^{k-1}c^{\ast})) \cap K)$. If that difference is bigger than $C d^2/(2\sigma)$ then the algorithm stops and outputs $d/c^{\ast}$. If not, it proceeds to look at all children of $\nu^*$ in the same way but replacing $d$ with $d/2$. For example, if we take $\nu_i^*$ the algorithm looks at the difference between $w(B(\nu^*_i, d/2^{2-1})\cap K)$ and the smallest possible difference from its children. If the algorithm finds a point whose corresponding difference exceeds $C (d/2)^2/(2\sigma)$, the algorithm terminates and outputs $d/(2c^{\ast})$. Else, it continues to look at the children of children, and so on, in a breadth-first search manner.

The following theorem (which uses Lemma \ref{lemma:local:packing:algorithm} in the Appendix) states our resulting bounds on $\epsLSE(\sigma)$ depending on how many iterations it takes for the algorithm to terminate.

\begin{theorem} \label{theorem:local:packing:algorithm} Define $c' =\frac{(2c^{\ast}-4)(4c^{\ast}-1/c^{\ast})}{(c^{\ast}-4)c^{\ast}}$ for some $c^{\ast}>4$. \renewcommand\labelenumi{(\theenumi)} 
\renewcommand{\theenumi}{\roman{enumi}}
\begin{enumerate}
    \item  Suppose Algorithm \ref{algo_local_packing} terminates after $k$ iterations. Then $\epsLSE(\sigma) \gtrsim d/(2^{k-1}c^{\ast})$.
    \item Suppose Algorithm \ref{algo_local_packing} does not terminate within $k$ iterations.
        \begin{enumerate}
            \item  If $\sigma \ge c'\cdot \frac{d}{2^k}$, then $\epsLSE(\sigma)\asymp \sigma\wedge d$.
            \item  If $\sigma \le c'\cdot \frac{d}{2^k}$, then  $\epsLSE(\sigma) \lesssim c'\cdot \frac{d}{2^k}$.
        \end{enumerate}
\end{enumerate}
\end{theorem}

\subsection{A global packing algorithm}

In this subsection, we present a global packing algorithm based on Theorem \ref{difference:of:local:widths:compared:to:eps:squared:thm}. We use the same absolute constants $L,c^{\ast},$ and $C$ from the theorem. For a fixed $\sigma>0$, Algorithm \ref{algo_global_packing} (if it does not immediately terminate) produces an $\varepsilon$ that will either match $\epsLSE$ or upper bound it. The algorithm accomplishes this by repeatedly forming $\delta$-packing sets of $K$ with $\delta=\left(\varepsilon^3/(4c^{\ast}\sup_{\eta \in K} w_\eta(\varepsilon/c^{\ast})\sigma) \right)\wedge \varepsilon$, and doubling $\varepsilon$ until a certain condition is met.

\begin{algorithm}[t]
\SetKwComment{Comment}{/* }{ */}
\caption{Global Packing Algorithm}\label{algo_global_packing}
\KwInput{$K$ a convex set, $\sigma>0$, absolute constants $c^{\ast},L>0$, and $C=8+8/c^{\ast}$.}
Initialize $\varepsilon \gets 2 \underline\varepsilon^{\ast}$ where $\underline\varepsilon^{\ast}$ is defined in \eqref{underline:varepsilon:def}\;
$\delta\gets\varepsilon^3/(4c^{\ast}\sup_{\eta \in K} w_\eta(\varepsilon/c^{\ast})\sigma) \wedge \varepsilon$\;
Form a maximal $\delta$-packing set $\{\nu_1, \nu_2,\ldots, \nu_M\}$ of $K$, where $M = \cM(\delta,K)$\;
$\Psi\gets \max_{i \in [M]} \sup_{\nu' \in B(\nu_i, 2\varepsilon-\delta)\cap K} w_{\nu'}(\varepsilon/c^{\ast}) - w_{\nu_i}(\varepsilon/c^{\ast})$\;
$T\gets  C \varepsilon^2/(2\sigma) -( L/c^{\ast})\cdot\varepsilon \sqrt{\log \cMKloc(\varepsilon)}$\;
\While{$\Psi\ge T$}{
 $\varepsilon\gets 2\varepsilon$\;
 Update $\delta$\; 
Update maximal $\delta$-packing set $\{\nu_1,\dots,\nu_M\}$ of $K$ with $M=\cM(\delta, K)$\;
 Update $\Psi$ and $T$\;
}
\Return{$\varepsilon$}  
\end{algorithm}


In computing $\Psi$ from the algorithm, we must solve the problem $\sup_{\nu' \in K\cap B(\nu_i, 2\varepsilon-\delta)} w_{\nu'}(\varepsilon/c^{\ast})$ for each $i$. This is a maximization of a concave function over a convex set and in principle is computationally tractable.

Our main result is stated in Theorem \ref{theorem:global:packing:algorithm}, which relates the output $\varepsilon$ of Algorithm \ref{algo_global_packing}  to $\epsLSE(\sigma)$. The proof uses Lemma \ref{lemma:global:algorithm} to compare $\varepsilon$ to $\epsSupwidth(\sigma)$, and that lemma in turn relies on Lemma \ref{lemma:technical:global:algorithm} to  bound $\varepsilon$ depending on whether $\Psi < T$ or $\Psi \ge T$ (the algorithm's stopping conditions).

\begin{theorem} \label{theorem:global:packing:algorithm} Suppose Algorithm \ref{algo_global_packing} does not terminate on initialization, and let $\varepsilon$ be the output of the algorithm. Then the following hold:
\renewcommand\labelenumi{(\theenumi)} 
\renewcommand{\theenumi}{\roman{enumi}}
\begin{enumerate}
    \item If $d \lesssim \sigma$, then $\epsLSE(\sigma)\asymp d \lesssim \varepsilon$. 
    \item  If $d \gtrsim \sigma$ and $\sigma\ge \varepsilon$, then $\varepsilon\asymp\epsLSE(\sigma)\asymp \sigma\wedge d$. 
    \item If $d\gtrsim \sigma$ and $\sigma\le \varepsilon$, we have $\varepsilon\asymp\epsLSE(\sigma)$. 
\end{enumerate}
If Algorithm \ref{algo_global_packing} does terminate on initialization, $\varepsilon \asymp \epsLSE(\sigma)$. 
\end{theorem}

\section{Proofs for Section \ref{section:LSE:introduction}}

\begin{lemma}[Generic Lower Bound on the Minimax Rate]\label{generic:bound:minimax:rate}
    For a closed convex body $K$ with diameter $d$, we have that $\varepsilon^{*2} \gtrsim \min(w(K)^2/n, \sigma^2\cdot w(K)^2/d^2)$. This bound is sharp for an Euclidean ball.
\end{lemma}

\begin{proof}[Proof of Lemma \ref{generic:bound:minimax:rate}]

Define the spherical width $w^s(K)$ of $K$ as $\EE_{\xi\sim \cN(\vec{0},\II_n)}\sup_{t\in K} \langle t, \xi/\|\xi\|\rangle$. 
The Dvoretzky-Milman Theorem states that given a convex body $K$, for $m\lesssim \eta^2 w(K)^2/d^2$, a random projection $P$ onto an $m$-dimensional subspace satisfies $(1-\eta)B \subset P K \subseteq (1 + \eta)B$, where $B$ is a ball of radius $w^s(K) \asymp w(K)/\sqrt{n}$ \citep[Lemma 7.5.6, Lemma 7.6.2, Exercise 11.3.9]{vershynin2018high}.  It follows that the minimax rate is always at least $\min(w(K)^2/n, w(K)^2/d^2 \sigma^2)$. This is because one can select a point $(Pv, P^{\perp}v)$ for $v\in K$ such that $Pv$ is the center of a $(1-\eta) B$ (so that when we draw a ball of radius $\varepsilon < w^s(K)$ its projection will be completely in the ball $(1-\eta)B$). Then if $\varepsilon \lesssim \min(w(K)/\sqrt{n}, w(K)/d \sigma)$, we can construct a packing set by just taking points in the sphere $(Pv^i, P^{\perp}v^i)$ so that $\|v^i - v^j\| \geq \|Pv^i - Pv^j\| \geq \varepsilon/c$. These points will be at least exponential in the dimension which is $\exp(c w(K)^2/d^2)$, and hence since $\varepsilon/\sigma \leq \sqrt{w(K)^2/d^2}$ we have that $\varepsilon^{*2} \gtrsim \min(w(K)^2/n, w(K)^2/d^2 \sigma^2)$. 
\end{proof}    

    \begin{proof}[Proof of Lemma \ref{minimax:bound:versus:sigma:and:d}]
   First suppose $d \ge \sigma$. Take $\kappa > \max(2, \sqrt{1/\log 2}).$ Then taking $\varepsilon = \sigma/\kappa \le d/\kappa \le d/2$, we can place a ball of radius $\varepsilon$ centered at a point in $K$ where there exists a diameter of that ball of length $2\varepsilon<d$ contained inside $K$. By picking equispaced points along the diameter of this ball, we conclude that $\log \cMloc(\varepsilon) >\log 2$ (provided $c^{\ast}$ is sufficiently large). So we have $\varepsilon^2/\sigma^2 = 1/\kappa^2 < \log 2 < \log \cMloc(\varepsilon)$. Thus, $\varepsilon^{\ast}\ge \sigma/\kappa$ by its definition as a supremum, so that $\varepsilon^{\ast} \gtrsim \sigma \ge \sigma\wedge d$. Suppose on the other hand that $d<\sigma$. Then fitting a diameter (of some ball) of length $\varepsilon= d/3$ inside $K$, we can ensure  $\log \cMloc(\varepsilon) >\log 2$ while also having \[\frac{\varepsilon^2}{\sigma^2}= \frac{d^2}{9\sigma^2} < \frac{\sigma^2}{9\sigma^2} = \frac{1}{9}  < \log 2.\] This proves $\varepsilon^{\ast}\ge d/3,$ so that $\varepsilon^{\ast}\gtrsim d \ge \sigma\wedge d$. In either case, $\varepsilon^{\ast}\gtrsim \sigma\wedge d$.
    \end{proof}

\begin{proof}[\hypertarget{proof:lemma:equivalent:information:lower:bound}{Proof of Lemma \ref{lemma:equivalent:information:lower:bound}}] 
Define $S(C_1, C_2)= \{\varepsilon>0:\varepsilon^{2}/\sigma^2 \leq C_1\log \cMloc(C_2\varepsilon)\}$. Our goal is to show that $\sup S(C_1,C_2) \asymp \sup S(1,1)$ for all $C_1, C_2>0$. 

First, we prove that $\sup S(1, C_2)\asymp\sup S(1,1)$ for all $C_2>0$. Suppose $C_2<1$. Pick $\varepsilon\in S(1, C_2)$, so that $\varepsilon^2/\sigma^2 \le \log \cMloc(C_2\varepsilon).$ Then multiplying by $C_2^2$ and using $C_2<1$, \[\tfrac{(C_2\varepsilon)^2}{\sigma^2} \le C_2^2  \log \cMloc(C_2\varepsilon) < \log \cMloc(C_2\varepsilon).\] Thus $C_2\varepsilon\in S(1,1)$, so $\sup S(1,1) \ge C_2\varepsilon$. Since this holds for any $\varepsilon\in S(1, C_2)$, we have $\sup S(1,1)\gtrsim \sup S(1,C_2)$. On the other hand, pick $\varepsilon\in S(1,1)$. Since $\cMloc(\varepsilon)$ is non-increasing in $\varepsilon$, we have \[\tfrac{\varepsilon^2}{\sigma^2} \le  \log \cMloc(\varepsilon) \le \log \cMloc(C_2\varepsilon),\] proving that $\varepsilon\in S(1, C_2)$. As this shows $S(1,1)\subseteq S(1, C_2)$, we have $\sup S(1, C_2)\ge \sup S(1,1)$. This proves that  $\sup S(1, C_2)\asymp\sup S(1,1)$ when $C_2<1$.

Suppose instead that $C_2\ge 1$. Then picking any $\varepsilon\in S(1, C_2)$, we have by the non-increasing property of $\cMloc(\varepsilon)$ that \[\tfrac{\varepsilon^2}{\sigma^2} \le \log \cMloc(C_2\varepsilon) \le \log \cMloc(\varepsilon).\] Then $\varepsilon\in S(1,1)$, showing that $S(1, C_2)\subseteq S(1,1)$ and thus $\sup S(1,1)\ge \sup S(1,C_2)$. On the other hand, pick $\varepsilon\in S(1,1)$ so that $\varepsilon^2/\sigma^2 \le \log \cMloc(\varepsilon)$.  Then \[ \tfrac{(\varepsilon/C_2)^2}{\sigma^2} \le \tfrac{\varepsilon^2}{\sigma^2} \le\log \cMloc(\varepsilon)=\log \cMloc(C_2\cdot \varepsilon/C_2).\] Hence $\varepsilon/C_2\in S(1, C_2)$, so that $\sup S(1, C_2)\ge \varepsilon/C_2$.  This holds for all $\varepsilon\in S(1,1)$, so $\sup S(1, C_2)\gtrsim \sup S(1,1)$, proving that $\sup S(1, C_2)\asymp\sup S(1,1)$ when $C_2\ge 1$.

We have thus proved that $\sup S(1,C_2)\asymp \sup S(1,1)$ for all $C_2>0$. It suffices to now prove that $\sup S(C_1, C_2)\asymp \sup S(1, C_2)$ for all $C_1, C_2>0$.

As before, first suppose $C_1<1$. That $S(C_1, C_2)\subseteq S(1, C_2)$ is clear since $C_1 \log \cMloc(C_2\varepsilon) \le \log \cMloc(C_2\varepsilon),$ so we have $\sup  S(1, C_2)\ge \sup S(C_1, C_2)$. On the other hand, if $\varepsilon\in S(1, C_2)$, then by the non-increasing property of $\cMloc$,  \begin{align*}
    \tfrac{(\sqrt{C_1}\varepsilon)^2}{\sigma^2} &= \tfrac{C_1\varepsilon^2}{\sigma^2}\le C_1 \cdot\log \cMloc(C_2\varepsilon) \\ &\le C_1 \log \cMloc(C_2(\sqrt{C_1}\varepsilon)). 
\end{align*} Thus $\sqrt{C_1}\varepsilon \in S(C_1, C_2)$, so $\sup  S(C_1, C_2)\ge \sqrt{C_1}\varepsilon$ for all $\varepsilon\in S(1, C_2)$. This proves $\sup  S(C_1, C_2)\gtrsim \sup S(1, C_2)$, so that  $\sup  S(1, C_2)\asymp \sup S(C_1, C_2)$ when $C_1< 1$.

Finally, if $C_1\ge 1$, we have $S(1, C_2)\subseteq S(C_1, C_2)$ since $\log  \cMloc(C_2\varepsilon) \le C_1 \log  \cMloc(C_2\varepsilon)$ so that $\sup S(C_1, C_2)\ge \sup S(1, C_2)$. On the other hand, for any $\varepsilon\in S(C_1, C_2)$,  we have \[ \tfrac{(\varepsilon/\sqrt{C_1})^2}{\sigma^2} = \tfrac{1}{C_1}\cdot \tfrac{\varepsilon^2}{\sigma^2} \le \log \cMloc(C_2\varepsilon) .\] Hence $\varepsilon/\sqrt{C_1}\in S(1,C_2)$, proving that $\sup S(1, C_2)\ge \varepsilon/\sqrt{C_1}$ for all $\varepsilon\in S(C_1, C_2)$. So $\sup S(1, C_2)\gtrsim \sup S(C_1, C_2)$ and hence  $\sup S(1, C_2)\asymp \sup S(C_1, C_2)$ when $C_1\ge 1$. 

This proves $\sup S(C_1, C_2)\asymp \sup S(1, C_2)$ for all $C_1, C_2>0$ and we already showed $\sup S(1, C_2)\asymp S(1,1)$ for all $C_2>0$. This proves that $\varepsilon^{\dag}= \sup S(C_1, C_2)\asymp \sup S(1,1) = \varepsilon^{\ast}$ for all $C_1, C_2>0$.
\end{proof}

\section{Proofs for Section \ref{main:results:sec}}

\begin{proof}[\hypertarget{proof:lemma:chatterjee:analogue}{Proof of Lemma \ref{lemma:chatterjee:analogue}}] 
    We will need the following restatement of \citet[Theorem 1.1]{chatterjee2014new} (to include an arbitrary $\sigma$). 
    \begin{theorem}[Theorem 1.1 of \citet{chatterjee2014new}]
        For any $x\ge 0$,
        \begin{align}
            \MoveEqLeft \PP(|\|\hat \mu - \mu\| - \epswidth| \geq x \sqrt{\epswidth}) \notag \\ &\leq 3 \exp\left(-\frac{x^4}{32\sigma^2(1 + x/\sqrt{\epswidth})^2}\right).  \label{eq:chatterjee:theorem}
        \end{align}
    \end{theorem}
    We now prove the lemma. Let $C = \int_0^{\infty} 6 x \exp(-x^4/(32 (1 + x)^2)) \mathrm{d}x > 1$.
    Suppose now $\epswidth \geq 4 C\sigma > \sigma$. Then by \eqref{eq:chatterjee:theorem} we have
    \begin{align*}
         \MoveEqLeft \PP(|\|\hat \mu - \mu\| - \epswidth|\geq x \sqrt{\epswidth}) \\ &\leq 3 \exp\left(-\frac{x^4}{32\sigma^2(1 + x/\sqrt{\sigma})^2}\right),
    \end{align*}
    and hence
    \begin{align*}
        \MoveEqLeft\EE (\|\hat \mu - \mu\| - \epswidth)^2/\epswidth \\ &\leq \int_{0}^{\infty} 6 x \exp\left(-\frac{x^4}{32\sigma^2(1 + x/\sqrt{\sigma})^2}\right) \mathrm{d}x \\ &= C \sigma.
    \end{align*} Rearranging and using the fact that $\EE \|\hat\mu - \mu\|\le \sqrt{\EE \|\hat\mu - \mu\|^2}$ by Jensen's inequality, we get
    \begin{align*}
        \EE \|\hat \mu - \mu\|^2 - 2 \sqrt{\EE \|\hat \mu - \mu\|^2} \epswidth + \epswidth^2 - C\sigma\epswidth \leq 0.
    \end{align*}
    Thus $\epswidth - \sqrt{C\sigma\epswidth} \leq \sqrt{\EE \|\hat \mu - \mu\|^2}  \leq \epswidth + \sqrt{C\sigma\epswidth}$.
    Since by assumption $\sqrt{C\sigma\epswidth} \leq \epswidth/2$ we conclude that $ \varepsilon^2_\mu/4 \leq \EE \|\hat \mu - \mu\|^2\leq 9 \varepsilon^2_\mu / 4$.

    On the other hand, suppose $\epswidth < 4 C \sigma$. Then setting $z = x\sqrt{\epswidth}$ we have
    \begin{align*}
         \PP(|\|\hat \mu - \mu\| - \epswidth| \geq z) &\leq 3 \exp\left(-\frac{z^4}{32\sigma^2(\epswidth + z)^2}\right)\\ &\leq 3 \exp\left(-\frac{z^4}{32\sigma^2(4C \sigma + z)^2}\right).
    \end{align*}
    Thus
    \begin{align*}
       \MoveEqLeft \EE (\|\hat \mu - \mu\| - \epswidth)^2 \\ &\leq \int_{0}^{\infty} 6 z \exp\left(-\frac{z^4}{32\sigma^2(4C \sigma + z)^2}\right) \mathrm{d}z \\ &= \sigma^2 \int_0^\infty 6 t \exp\left(-\frac{t^4}{32(4C  + t)^2}\right) \mathrm{d}t \\ &= C' \sigma^2.
    \end{align*}
    Thus we conclude that 
        \begin{align*}
        \EE \|\hat \mu - \mu\|^2 - 2 \sqrt{\EE \|\hat \mu - \mu\|^2} \epswidth + \epswidth^2 - C'\sigma^2 \leq 0.
    \end{align*}
    which implies that $ \sqrt{\EE \|\hat \mu - \mu\|^2} \leq \epswidth + \sqrt{C'}\sigma \leq C'' \sigma$, which completes the proof.
\end{proof}    

    \begin{proof}[\hypertarget{proof:lemma:epsilon_mu:nondecreasing}{Proof of Lemma \ref{lemma:epsilon_mu:nondecreasing}}]
        To prove our first claim, it suffices to show for any $\mu\in K$ that $\sigma\mapsto \epswidth(\sigma)$ is non-decreasing on $[0,\infty)$. We first consider the case where $K$ is not a singleton set and restrict to $\sigma>0$. We must first verify some technicalities about the subgradient. Fix $\mu \in K$ and pick any $\sigma>0$. Let $\partial w_{\mu}(\varepsilon)$ be the subgradient  of the map $\varepsilon \mapsto w_{\mu}(\varepsilon).$ This mapping is proper (i.e., finite everywhere) and concave, so the subgradient will exist in the interior of the domain of $w_{\nu}(\varepsilon)$ which is $(0,\infty)$. 
        
        We now verify that $\varepsilon=\epswidth(\sigma)$ is in the interior of this domain, i.e.,  $\epswidth(\sigma)>0$. To see this, observe that $w_{\mu}(\varepsilon) \gtrsim \varepsilon\wedge d>0$ by \citet[Proposition 7.5.2(vi)]{vershynin2018high}. Then for sufficiently small $\varepsilon>0$, we have \begin{align*}
            \sigma\cdot w_{\mu}(\varepsilon) -\varepsilon^2/2 &\gtrsim \sigma \cdot \varepsilon -\varepsilon^2/2 >0 \\ &= \sigma\cdot w_{\mu}(0) -0^2/2. 
        \end{align*}In other words, $0\not\in \argmax_{\varepsilon}[\sigma\cdot w_{\mu}(\varepsilon) -\varepsilon^2/2]$ so $\epswidth(\sigma)>0$. This ensures $\epswidth(\sigma)$ is in the interior of the domain of $w_\nu(\varepsilon)$, i.e., $(0,\infty)$, and the subgradient exists at this point. Moreover, since $\varepsilon \mapsto w_{\mu}(\varepsilon)$ is concave, the subgradient is indeed monotone non-increasing in $\varepsilon$.

        We now prove $\sigma\mapsto \epswidth(\sigma)$ is non-decreasing on $(0,\infty)$. Observe that $\argmax_{\varepsilon} [\sigma\cdot w_{\mu}(\varepsilon) - \varepsilon^2/2]$ is achieved at $\varepsilon=\epswidth(\sigma)$ which satisfies $\sigma\cdot \partial w_{\mu}(\varepsilon) =\varepsilon$. 
 Suppose $0<\sigma <\sigma'$.  Then we have \begin{align*}
     \epswidth(\sigma') &= \sigma'\cdot\partial w_{\mu}(\epswidth(\sigma')) \ge \sigma'\cdot\partial w_{\mu}(\epswidth(\sigma)) \\ &\ge \sigma\cdot\partial w_{\mu}(\epswidth(\sigma)) = \epswidth(\sigma).
 \end{align*} Thus, the map $\sigma\mapsto \epswidth(\sigma)$ is non-decreasing on $(0,\infty)$ so long as $K$ is not a singleton set.

        Let us now handle the case where $K$ is a singleton set and show the mapping is non-decreasing on $(0,\infty)$. Well, we must have $w_{\mu}(\varepsilon)=0$ for any $\varepsilon\ge 0$. Hence \begin{equation}\label{eq:lemma:nonincreasing}
            \epswidth(\sigma)=\argmax_{\varepsilon}[\sigma w_{\mu}(\varepsilon)-\varepsilon^2/2]=\argmax_{\varepsilon}[-\varepsilon^2/2] =0.
        \end{equation} This holds for any $\sigma$, hence $\epswidth(\sigma)=0\le 0= \epswidth(\sigma')$ for any $0<\sigma<\sigma'$. Our map is thus non-decreasing on $(0,\infty)$ without the non-singleton restriction.
        
        Let us now show the non-decreasing property extends to $[0,\infty)$. We again start with the non-singleton case for $K$. By similar logic to \eqref{eq:lemma:nonincreasing}, we have $\epswidth(0)=0$. Next, for $\sigma'>0$, recall we showed previously that $\epswidth(\sigma')>0$, i.e., lies in the interior of the domain of $w_{\mu}(\varepsilon)$. Hence $\epswidth(\sigma')>\epswidth(0)$. This proves $\sigma\mapsto \epswidth(\sigma)$ is non-decreasing on $[0,\infty)$ provided $K$ is not a singleton.

        Finally, we suppose $K$ is a singleton set and show the mapping is non-decreasing on $[0,\infty)$. We still have $\epswidth(0)=0$, and by our previous argument in \eqref{eq:lemma:nonincreasing}, $\epswidth(\sigma')=0$ for any $\sigma'>0$. Thus, trivially, $\epswidth(0)\le \epswidth(\sigma')$ for any $\sigma'>0$, verifying  $\sigma\mapsto \epswidth(\sigma)$ is non-decreasing on $[0,\infty)$ regardless of the singleton status of $K$. This completes our proof that the map $\sigma\mapsto \epsSupwidth(\sigma)$ is non-decreasing on $[0,\infty)$.


        For our next claim, let $c>1$. If either $\sigma=0$ or $K$ is a singleton set, the desired inequality is trivial. Otherwise, suppose $\sigma>0$ and $K$ is not a singleton set. We have established that for any $\mu$, the subgradient is monotone non-increasing in $\varepsilon$. Using our first order subgradient condition $\sigma\cdot \partial w_{\mu}(\varepsilon) =\varepsilon$ and the fact that $\epswidth(c\sigma) \ge \epswidth(\sigma)$, it follows that 
        \begin{align*}
            \epswidth(c\sigma) &= c \sigma \cdot \partial w_\mu(\epswidth(c\sigma)) \leq c \sigma \cdot \partial w_\mu(\epswidth(\sigma))  \\ &= c \cdot\epswidth(\sigma).
        \end{align*}
        Taking the $\sup$ over $\mu \in K$, we have $\epsSupwidth(c\sigma)\le c\epsSupwidth(\sigma)$, and since we already showed $\sigma\mapsto \epsSupwidth(\sigma)$ is non-decreasing, we also have $\epsSupwidth(\sigma) \le \epsSupwidth(c\sigma)$.

        Suppose instead $0\le c<1$. The $c=0$ case is trivial, so assume $c\in(0,1)$. Using our previous result applied to $c^{-1}>1$, \[\epsSupwidth(\sigma) = \epsSupwidth(c^{-1}\cdot c\sigma)\le c^{-1}\epsSupwidth( c\sigma),\] or equivalently $c\epsSupwidth(\sigma) \le \epsSupwidth(c\sigma)$. We also have $\epsSupwidth(c\sigma)\le \epsSupwidth(\sigma)$ again since $\sigma\mapsto \epsSupwidth(\sigma)$ is non-decreasing. This proves our final claim.
    \end{proof}

\subsection{Proofs for Section \ref{subsection:sufficient:conditions:worst:case}}

\begin{proof}[\hypertarget{proof:proposition:first:sufficient:condition:generalized}{Proof of Proposition \ref{proposition:first:sufficient:condition:generalized}}]
Let $\epsSupwidth[K';K] = \sup_{\mu \in K'} \epswidth$, where $\epswidth = \argmax_{\varepsilon} \sigma w_{K,\mu}(\varepsilon)-\varepsilon^2/2$. For brevity, we set $\overline{\varepsilon} = \overline{\varepsilon_{K';K}}$ . Observe now that for every $\mu \in K'$, we have $\sigma \cdot w_{K,\mu}(2 \overline \varepsilon) - 2 \overline{\varepsilon}^2 < 0$, and thus by  \citet[Proposition 1.3]{chatterjee2014new} we have $2\overline\varepsilon > \epsSupwidth[K';K]$.

We will now relate $\epsSupwidth[K';K]$ to ${\epsLSE[K';K]}$. 

\textsc{Case 1:} Suppose $\epsSupwidth[K';K] \gtrsim \sigma$. Pick $\tilde\mu\in K'$ that maximizes $\EE\|\hat\mu-\tilde\mu\|^2$.

\textsc{Case 1(a):} Suppose $\epswidth[\tilde\mu] \gtrsim\sigma$. Then by Lemma \ref{lemma:chatterjee:analogue} and definition of $\epsSupwidth[K';K]$, we have \begin{align*}
    \epsLSE[K';K]^2 = \EE\|\hat\mu-\tilde\mu\|^2 \asymp \epswidth[\tilde\mu]^2 \lesssim \epsSupwidth[K';K]^2 < 4\overline{\varepsilon}^2
\end{align*} as required.

\textsc{Case 1(b):} Suppose $\epswidth[\tilde\mu] \lesssim\sigma$. Then by Lemma \ref{lemma:chatterjee:analogue} and our assumption for Case 1, \begin{align*}
    \epsLSE[K';K]^2 = \EE\|\hat\mu-\tilde\mu\|^2 \lesssim \sigma^2 \lesssim\epsSupwidth[K';K]^2 < 4\overline{\varepsilon}^2.
\end{align*}

\textsc{Case 2}: Suppose $\epsSupwidth[K';K] \lesssim \sigma$. This means for any $\mu\in K'$, $\epswidth \lesssim\sigma$. Hence $\EE\|\hat\mu-\mu\|^2\lesssim \sigma^2$ by Lemma \ref{lemma:chatterjee:analogue} for all $\mu\in K'$, so that $\epsLSE[K';K]\lesssim\sigma$. We now consider two subcases.

\textsc{Case 2(a):} Suppose $d \ge \frac{2\sigma(n+1)}{n\sqrt{2\pi}} \gtrsim \sigma$, where $d = \diam(K')$. We show this implies $\overline\varepsilon \gtrsim\sigma$ which proves $\epsLSE[K';K] \lesssim \overline\varepsilon$.

Note that $K'$ contains a diameter of length $2\sigma/\kappa$ (for a sufficiently large absolute constant $\kappa$). Thus, by \citet[Proposition 7.5.2(vi)]{vershynin2018high}, \[\sup_{\mu \in K'} w_{K,\mu}( \sigma/\kappa)  \geq \sup_{\mu \in K'} w_{K',\mu}(\sigma/\kappa)  \ge \frac{2\sigma}{\kappa\sqrt{2\pi}} .\] Since for sufficiently large $\kappa$ we have $\sigma/(2\kappa^2) \leq 2\sigma/(\kappa\sqrt{2\pi})$, it follows that $\sup_{\mu\in K'} w_{K,\mu}(\varepsilon)\ge \varepsilon^2/2\sigma$ for $\varepsilon = \sigma/\kappa$. By definition of $\overline{\varepsilon}$ as a supremum of such $\varepsilon$, we have $\overline{\varepsilon} \geq \sigma/\kappa$. This completes the case when $d\gtrsim\sigma$.

\textsc{Case 2(b):} Suppose $d \le \frac{2\sigma(n+1)}{n\sqrt{2\pi}}$. Then by the fact that $K'\subseteq K$ followed by Jung's Theorem and \citet[Proposition 6.5.2(vi)]{vershynin2018high}, for some $\mu\in K'$ we have \begin{align*}
    w_{K,\mu}\left( \sqrt{\tfrac{n}{2(n+1)}} d\right) &\ge w_{K',\mu}\left( \sqrt{\tfrac{n}{2(n+1)}} d\right) = w(K') \\ &\ge \frac{d}{\sqrt{2\pi}}\ge \frac{d}{2\sqrt{2\pi}}.
\end{align*} Notice that $\frac{d}{2\sqrt{2\pi}} \le \frac{n d^2}{2(n+1)}\cdot \frac{1}{2\sigma} $ so we have $w_{K,\mu}(\varepsilon)\ge \varepsilon^2/(2\sigma)$ for $\varepsilon = \sqrt{\frac{nd^2}{2(n+1)}} $. Therefore $\overline\varepsilon \gtrsim d$, and as $\epsLSE[K';K]\lesssim d,$ we have $\epsLSE[K';K]\lesssim \overline\varepsilon$.
\end{proof}

    \begin{proof}[\hypertarget{proof:important:thm}{Proof of Theorem \ref{important:thm}}]
To argue this, we will use the well known technique for the reverse Sudakov minoration estimate \citep[see the proof of Theorem 8.1.13 of][e.g.]{vershynin2018high}. We will spell out the details due to our use of a different dyadic scale (multiplying by $c^{\ast}$ rather than $2$) and use of local packings instead of global ones. 

Fix $\mu\in K$. We start at level $\varepsilon$ and build an $\varepsilon/c^{\ast}$ maximal packing set $\cP_1$ of $B(\mu, \varepsilon) \cap K$. Next for any point $\mu_i \in \cP_1$, we build an $\varepsilon/{c^{\ast}}^2$ packing set of $B(\mu_i, \varepsilon/c^{\ast}) \cap K$. The totality of these packing set points at the second level comprise $\cP_2$, and $|\cP_2|\le \cMloc(\varepsilon) \cMloc(\varepsilon/c^{\ast})$. We then continue this process. Note that on the $k$th level, we have obtained an $\varepsilon/{c^{\ast}}^k$ covering set $\cP_k$ of $B(\mu, \varepsilon) \cap K$. This can be seen by induction. For the first level, this follows since $\cP_1$ is a maximal $\varepsilon/c^{\ast}$ packing of the set. Suppose the covering claim is true for level $k-1$. Since $\cP_{k-1}$ forms a covering, any point $\nu \in B(\mu, \varepsilon) \cap K$ will fall in an $\varepsilon/{c^{\ast}}^{k-1}$ ball centered at some point in $\cP_{k-1}$. To construct $\cP_k$, we must at perform a maximal $\varepsilon/{c^{\ast}}^k$-packing (hence covering) of this ball intersected with $K$, among others. Hence  $\nu$ will be within distance $\varepsilon/{c^{\ast}}^k$ to a point in $\cP_k$. Another inductive argument proves $|\cP_k|\le \prod_{i = 0}^{k-1} \cMloc(\varepsilon/{c^{\ast}}^i)$.

We now perform the chaining argument. Let $k$ be defined as the maximum integer so that $\varepsilon {c^{\ast}}^{-k} \geq \frac{w_{\mu}(\varepsilon)}{2 c^{\ast} \sqrt{n}}$. By \citet[Proposition 7.5.2(vi)]{vershynin2018high} we have \[w_\mu (\varepsilon) \geq \frac{\diam(B(\mu, \varepsilon) \cap K)}{\sqrt{2\pi}} \geq \frac{2\varepsilon \wedge d}{\sqrt{2\pi}},\] where  $d = \operatorname{diam}(K)$. It follows that $
    k \leq \log_{c^{\ast}} \frac{\varepsilon 2 c^{\ast} \sqrt{2 \pi n}}{2\varepsilon \wedge d}$.
We now write
\begin{align*}
   w_\mu (\varepsilon) = \EE \sup_{\nu \in B(\mu, \varepsilon) \cap K}\langle \xi, \nu \rangle = \EE \sup_{\nu \in B(\mu, \varepsilon) \cap K}\langle \xi, \nu - \mu \rangle ,
\end{align*}
and recognize $\langle \xi, \nu - \mu \rangle = \sum_{ j= 1}^k \langle \xi, \nu_j - \nu_{j - 1} \rangle + \langle \xi, \nu - \nu_k\rangle$, where $\nu_0 = \mu$, and $\nu_j$ is the closest point to $\nu$ from the $j$th packing set which we have constructed. Thus it suffices to bound
\begin{align}
    \sum_{ j= 1}^k \EE \sup_{\nu \in B(\mu, \varepsilon) \cap K}\langle \xi, \nu_j - \nu_{j - 1} \rangle + \EE \sup_{\nu \in B(\mu, \varepsilon) \cap K}\langle \xi, \nu - \nu_k\rangle. \label{eq:chaining:decomposition}
\end{align} Consider each term in the summation. The covering property implies \begin{align*}
    \|\nu_j - \nu_{j-1}\| &\leq \|\nu_j - \nu \| + \|\nu - \nu_{j-1}\| \\ &\leq \varepsilon/{c^{\ast}}^{j} + \varepsilon/{c^{\ast}}^{j- 1} \\ &\leq (1+c^{\ast} )\varepsilon/{c^{\ast}}^{j}.
\end{align*} Thus, using a Gaussian maximal inequality,
\begin{align*}
    \EE \sup_{\nu \in B(\mu, \varepsilon) \cap K}\langle \xi, \nu_j - \nu_{j - 1} \rangle & \leq \frac{C \varepsilon}{{c^{\ast}}^{j}} \sqrt{ \log \prod_{l = 0}^{j-1} \cMloc(\varepsilon/{c^{\ast}}^l)}\\
    & \leq \frac{C \varepsilon}{{c^{\ast}}^{j}} \sqrt{j \log \cMloc(\varepsilon/{c^{\ast}}^{j-1})}\\
    & \leq \frac{C \varepsilon}{{c^{\ast}}^{j}} \sqrt{k \log \cMloc(\varepsilon/{c^{\ast}}^{j-1})},
\end{align*}
where we used the fact that $\varepsilon\mapsto \log \cMloc(\varepsilon)$ is non-increasing \citep[Lemma II.8]{neykov2022minimax} and $C>1$ is an absolute constant (depending on $c^{\ast}$, which is itself a sufficiently large fixed constant). 

For the second term in \eqref{eq:chaining:decomposition}, Cauchy-Schwarz and the definition of $k$ imply
\begin{align*}
    \EE \sup_{\nu \in B(\mu, \varepsilon) \cap K}\langle \xi, \nu - \nu_k\rangle \leq \frac{\varepsilon \cdot \EE \|\xi\|}{{c^{\ast}}^k}  \leq \frac{\varepsilon\sqrt{n}}{{c^{\ast}}^k} \leq \frac{w_{\mu}(\varepsilon)}{2}.
\end{align*} Rearranging,
\begin{align*}
    w_\mu (\varepsilon) & \leq 2C \sum_{j = 1}^k \frac{\varepsilon}{c^{\ast}\cdot {c^{\ast}}^{j-1}} \sqrt{k\log \cMloc(\varepsilon/{c^{\ast}}^{j-1})}\\
    &\le 2C k^{3/2} \sup_{\delta \leq \varepsilon}  \frac{\delta}{c^{\ast}} \sqrt{\log \cMloc(\delta)} \\
    & \leq 2 C\bigg({\log_{c^{\ast}} \frac{\varepsilon 2 c^{\ast} \sqrt{2 \pi n}}{2\varepsilon \wedge d}}\bigg)^{3/2} \sup_{\delta \leq \varepsilon}  \frac{\delta}{c^{\ast}} \sqrt{\log \cMloc(\delta)}.
\end{align*}

Note this holds for all $\mu\in K$.
Observe that if $\varepsilon$ is such that the right-hand side above is smaller than $\varepsilon^2/(2\sigma)$  we will have an upper bound on $\epsLSE$ (up to constant factors) by Proposition \ref{proposition:first:sufficient:condition}. For brevity, set $C_n = 4 C\bigg({\log_{c^{\ast}} c^{\ast} \sqrt{2 \pi n}}\bigg)^{3/2}$. Set $\varepsilon = 2C_n\underline{\varepsilon}$, and assume at first that $4C_n\underline{\varepsilon} \leq d$. This implies $2 C\bigg({\log_{c^{\ast}} \frac{\varepsilon 2 c^{\ast} \sqrt{2 \pi n}}{2\varepsilon \wedge d}}\bigg)^{3/2}  \le C_n/2$. To control $w_{\mu}(\varepsilon)$, it thus suffices to show

\begin{align*}
      \frac{C_n}{2}\sup_{\delta \leq \varepsilon} \frac{\delta}{c^{\ast}}\sqrt{\log \cMloc(\delta)} \leq \frac{\varepsilon^2}{2\sigma}.
\end{align*}  
By definition of $\underline\varepsilon$, we know $(2 \underline{\varepsilon})^2/(2\sigma) \geq \frac{1}{2}\sup_{\delta \leq 2\underline{\varepsilon}} \delta/c^{\ast} \sqrt{\log \cMloc(\delta)}$, so  
\begin{align*}
    \frac{\varepsilon^2}{2\sigma}=\frac{(2 C_n \underline{\varepsilon})^2}{2\sigma} \geq  \frac{C_n^2}{2}\sup_{\delta \leq 2\underline{\varepsilon}} \frac{\delta}{c^{\ast}} \sqrt{\log \cMloc(\delta)}.
\end{align*} 
Now
\begin{align*}
    \MoveEqLeft \sup_{\delta \leq 2 C_n \underline{\varepsilon}} \frac{\delta}{c^{\ast}} \sqrt{\log \cMloc(\delta)} \\ & \leq \sup_{\delta \leq 2 \underline{\varepsilon}} \frac{\delta}{c^{\ast}} \sqrt{\log \cMloc(\delta)}  \vee \sup_{2 \underline{\varepsilon} < \delta \leq 2 C_n \underline{\varepsilon}} \frac{\delta}{c^{\ast}} \sqrt{\log \cMloc(\delta)}  \\
    & \leq \sup_{\delta \leq 2 \underline{\varepsilon}} \frac{\delta}{c^{\ast}} \sqrt{\log \cMloc(\delta)} \vee \frac{2C_n \underline{\varepsilon}}{c^{\ast}} \sqrt{\log \cMloc(2 \underline{\varepsilon})} \\
    &\leq C_n \sup_{\delta \leq 2 \underline{\varepsilon}} \frac{\delta}{c^{\ast}} \sqrt{\log \cMloc(\delta)},
\end{align*}
where we are using $C_n \geq 1$ and that for convex sets the map $\varepsilon \mapsto \sqrt{\log \cMloc(\varepsilon)}$ is non-increasing \citep[Lemma II.8]{neykov2022minimax}. Hence we showed 
\begin{align*}
   \frac{\varepsilon^2}{2\sigma} \geq  \frac{C_n}{2}\sup_{\delta \leq 2C_n\underline{\varepsilon}} \delta/c^{\ast} \sqrt{\log \cMloc(\delta)}.
\end{align*} 
On the other hand, if $4C_n\underline{\varepsilon} > d$, we already know that $d$ is an upper bound on the rate.
\end{proof}

 \subsection{Proofs for Section \ref{subsection:characterizations:conditions:worst:case}}

    \begin{proof}[\hypertarget{proof:big:width:minus:small:width:thm}{Proof of Theorem \ref{big:width:minus:small:width:thm}}] We first prove that $\overline{\varepsilon}(\sigma)/c \leq \epsSupwidth(\sigma) \leq \overline{\varepsilon}(4c\sigma/(c-1))$, where recall $\epsSupwidth(\sigma)=\sup_{\mu\in K}\epswidth$.
    Pick $\mu\in K$ and $\nu^{\ast}\in B(\mu,\overline{\varepsilon}(\sigma))\cap K$ achieving the supremum and infimum, respectively, in the definition of $\overline{\varepsilon}(\sigma)$. Then
    \begin{align*} \MoveEqLeft w_{\nu^{\ast}}((2+1/c)\bar\varepsilon(\sigma)) - \frac{(2+1/c)^2\overline{\varepsilon}(\sigma)^2}{2\sigma} \\ & \geq w_{\mu}(\overline{\varepsilon}(\sigma))- \frac{(2+1/c)^2\overline{\varepsilon}(\sigma)^2}{2\sigma} \\
        & \geq w_{\nu^{\ast}}(\bar\varepsilon(\sigma)/c) - \frac{(\overline{\varepsilon}(\sigma)/c)^2}{2 \sigma}.
    \end{align*} The first inequality follows since $B(\mu, \bar\varepsilon(\sigma)) \subset B(\nu^{\ast},(2+1/c)\bar\varepsilon(\sigma))$ and the second from the definition of $\bar\varepsilon(\sigma)$.
    Hence by concavity of $\varepsilon\mapsto w_{\nu^{\ast}}(\varepsilon)-\varepsilon^2/(2\sigma)$, we have $\epsSupwidth(\sigma) \geq \epswidth[\nu^{ast}](\sigma)\ge \overline{\varepsilon}(\sigma)/c$.  

Now define $\sigma' := 4c\sigma/(c - 1) > \sigma$. Then for all $\mu \in K$ and all $\kappa > 0$,
    \begin{align*}
        \MoveEqLeft w_{\mu}( (1 + \kappa)\overline \varepsilon(\sigma')) - w_{\mu}( (1 + \kappa)\overline \varepsilon(\sigma')/c)  \\ & \leq \frac{(4 + 4/c) (1 + \kappa)^2 \overline{\varepsilon}(\sigma')^2}{2\sigma'} \\
        & = \frac{(1 - 1/c^2)(1 + \kappa)^2\overline{\varepsilon}(\sigma')^2}{2\sigma}.
    \end{align*}
    Thus 
        \begin{align*}
       \MoveEqLeft w_{\mu}( (1 + \kappa)\overline \varepsilon(\sigma')) - \frac{(1 + \kappa)^2\overline{\varepsilon}(\sigma')^2}{2\sigma}
        \\ & \leq w_{\mu}( (1 + \kappa)\overline \varepsilon(\sigma')/c)-\frac{(1/c^2)(1 + \kappa)^2\overline{\varepsilon}(\sigma')^2}{2\sigma}.
    \end{align*} Again by concavity of $\varepsilon\mapsto w_{\mu}(\varepsilon)-\varepsilon^2/(2\sigma)$, we have $\epswidth(\sigma) \le (1+\kappa)\overline{\varepsilon}(\sigma')$. Since this holds for all $\mu\in K$, 
we have $\epsSupwidth(\sigma) \le (1 + \kappa)\overline \varepsilon(\sigma')$. Taking $\kappa \rightarrow 0$ implies $\epsSupwidth(\sigma) \leq \overline \varepsilon(\sigma')$. Thus, \begin{equation}
    \overline{\varepsilon}(\sigma)/c \leq \epsSupwidth(\sigma) \leq \overline{\varepsilon}(\sigma'). \label{eq:big:width:minus:small:width:thm:intermediate}
\end{equation} Now we proceed to bounding $\epsLSE$.

\textsc{Case 1:} Suppose $\overline{\varepsilon}(\sigma)\gtrsim\sigma$. Then $\epsSupwidth(\sigma)\gtrsim\sigma$ by \eqref{eq:big:width:minus:small:width:thm:intermediate}. Then picking $\mu\in K$ that maximizes $\epswidth(\sigma)$, we have $\epswidth(\sigma)\gtrsim \sigma$, so that by Lemma \ref{lemma:chatterjee:analogue}, $\EE\|\hat\mu-\mu\|^2 \asymp \epswidth(\sigma)^2 = \epsSupwidth(\sigma)^2$. Since $\epsLSE^2 \ge \EE\|\hat\mu-\mu\|^2$, we use \eqref{eq:big:width:minus:small:width:thm:intermediate} to conclude $\epsLSE \gtrsim \epsSupwidth(\sigma)\ge \overline{\varepsilon}(\sigma)/c$. It remains to show for Case 1 that $\epsLSE\lesssim \overline{\varepsilon}(\sigma')$. Well, let $\tilde\mu$ maximize $\EE\|\hat\mu-\tilde\mu\|^2$.

\textsc{Case 1(a):} If $\epswidth[\tilde\mu](\sigma)\gtrsim\sigma$, then by Lemma \ref{lemma:chatterjee:analogue} and \eqref{eq:big:width:minus:small:width:thm:intermediate} , \[\epsLSE^2 = \EE\|\hat\mu-\tilde\mu\|^2\asymp \epswidth[\tilde\mu](\sigma)^2\le \sup_{\mu}\epswidth(\sigma)^2=\epsSupwidth(\sigma)^2 \lesssim \overline{\varepsilon}(\sigma')^2.\] 

\textsc{Case 1(b):} If $\epswidth[\tilde\mu](\sigma)\lesssim\sigma$, then by Lemma \ref{lemma:chatterjee:analogue} and our earlier fact that $\epsSupwidth(\sigma)\gtrsim\sigma$, \[\epsLSE^2 =  \EE\|\hat\mu-\tilde\mu\|^2 \lesssim \sigma^2 \lesssim \epsSupwidth(\sigma)^2 \lesssim \overline{\varepsilon}(\sigma')^2.\] This proves in Case 1 that $\overline{\varepsilon}(\sigma)/c\lesssim\epsLSE\lesssim \overline{\varepsilon}(\sigma').$

\textsc{Case 2:} Now suppose $\overline{\varepsilon}(\sigma')\lesssim \sigma$. Then $\epsSupwidth(\sigma) \lesssim \sigma$ by \eqref{eq:big:width:minus:small:width:thm:intermediate}. Then for all $\mu\in K$, we have $\epswidth(\sigma)\lesssim\sigma$ which in turn implies for all $\mu\in K$ that $\EE\|\hat\mu-\mu\|^2\lesssim \sigma^2$ by Lemma \ref{lemma:chatterjee:analogue}. Thus $\epsLSE\lesssim \sigma$, and that $\epsLSE\lesssim d$ is clear. Hence $\epsLSE\lesssim \sigma\wedge d$. Then by Lemma \ref{minimax:bound:versus:sigma:and:d} and the minimax rate being a lower bound on the worst case LSE rate, we have $\epsLSE\gtrsim \varepsilon^{\ast}\gtrsim \sigma\wedge d$. Thus, $\epsLSE \asymp \sigma\wedge d$ when $\overline{\varepsilon}(\sigma')\lesssim \sigma$.
\end{proof}

    \begin{proof}[\hypertarget{proof:difference:of:local:widths:compared:to:eps:squared:thm}{Proof of Theorem \ref{difference:of:local:widths:compared:to:eps:squared:thm}}]

We abbreviate $\overline{\varepsilon} = \overline{\varepsilon}(\sigma)$ but if we mean $\overline{\varepsilon}(\sigma')$ we always explicitly write out the $\sigma'$.

    \textsc{Case 1:} Suppose $\overline\varepsilon(\sigma)\gtrsim\sigma$. We first upper bound $\epsLSE(\sigma)$. Take $\varepsilon' = (1 + \kappa)\overline{\varepsilon}(\sigma')$ where $\kappa>0$. For any $\mu \in K$, by definition of $\overline{\varepsilon}(\sigma')$,
\begin{align*}
    \MoveEqLeft \sup_{\nu_1 \in B(\mu,\varepsilon') \cap K}w_{\nu_1}(\varepsilon'/c^{\ast}) - \inf_{\nu_2 \in  B(\mu,\varepsilon') \cap K} w_{\nu_2}(\varepsilon'/c^{\ast}) \\ & \leq \frac{C \varepsilon'^2}{2\sigma'} - \frac{L\varepsilon'}{c^{\ast}}  \sqrt{\log \cMloc(\varepsilon')}.
\end{align*}
 By \citet[Exercise 2.4.11]{talagrand2014upper}, \[w_{\mu}(\varepsilon')\le  \sup_{\nu_1 \in B(\mu,\varepsilon') \cap K}  w_{\nu_1}(\varepsilon'/c^{\ast}) + (\tfrac{L}{c^{\ast}} )\varepsilon' \sqrt{\log \cMloc(\varepsilon')}.\] Moreover, clearly \[\inf_{\nu_2 \in  B(\mu,\varepsilon') \cap K}  w_{\nu_2}(\varepsilon'/c^{\ast})\leq w_{\mu}(\varepsilon'/c^{\ast}).\] Combining these three bounds yields
\begin{align*}
   w_{\mu}(\varepsilon') - w_{\mu}(\varepsilon'/c^{\ast})  \leq \frac{C\varepsilon'^2}{2\sigma'}.
\end{align*}
Using the definition of $\sigma'$, this implies
\begin{align*}
    w_{\mu}(\varepsilon') - \frac{\varepsilon'^2}{2\sigma} \leq w_{\mu}(\varepsilon'/c^{\ast}) - \frac{(\varepsilon'/c^{\ast})^2}{2\sigma},
\end{align*}
so that $\varepsilon'\ge \epswidth(\sigma)$ by our concavity argument again. Taking $\kappa\rightarrow 0$ shows that $\overline{\varepsilon}(\sigma')\ge \epswidth(\sigma)$. Since this holds for all $\mu\in K$, we have $\overline\varepsilon(\sigma') \geq \epsSupwidth(\sigma)$ (this fact will also hold for Case 2 by this argument). Note that $\sigma'>\sigma$, and so if $\varepsilon$ satisfies the condition in the definition of $\overline{\varepsilon}(\sigma)$, then it satisfies the condition in the analogous definition of $\overline{\varepsilon}(\sigma')$. Thus, $\overline{\varepsilon}(\sigma')\ge \overline{\varepsilon}(\sigma)$. Since this case assumed $\overline{\varepsilon}(\sigma)\gtrsim\sigma$, we have $\overline{\varepsilon}(\sigma')\gtrsim \sigma$.

Now, taking $\tilde\mu$ to maximize $\EE\|\hat\mu-\tilde\mu\|$, if $\epswidth[\tilde\mu](\sigma) \gtrsim \sigma$, then by Lemma \ref{lemma:chatterjee:analogue} and our upper bound for $\epsSupwidth(\sigma)$,  \begin{align*}
    \epsLSE^2 &=\EE\|\hat\mu-\tilde\mu\|^2\asymp \epswidth[\tilde\mu](\sigma)^2\le \sup_{\mu}\epswidth(\sigma)^2\\ &= \epsSupwidth(\sigma)^2\lesssim\overline{\varepsilon}(\sigma')^2. 
\end{align*}On the other hand, if   $\epswidth[\tilde\mu](\sigma) \lesssim \sigma$, then by Lemma \ref{lemma:chatterjee:analogue} we have \[\epsLSE^2 = \EE\|\hat\mu-\tilde\mu\|^2\le \sigma^2\lesssim \overline{\varepsilon}(\sigma')^2.
\] This completes the proof that if $\overline{\varepsilon}(\sigma)\gtrsim \sigma$, then $\epsLSE \lesssim \overline{\varepsilon}(\sigma')$.
    
    Let us now prove our lower bound $\epsLSE(\sigma) \gtrsim \overline{\varepsilon}(\sigma)$, which we accomplish by comparing $\overline{\varepsilon}(\sigma)$ to $\underline\varepsilon^{\ast}$, defined in \eqref{underline:varepsilon:def}. Recall from Remark \ref{remark:difference:of:local:widths}  that $\overline{\varepsilon}(\sigma) \geq \underline\varepsilon^{\ast}$.
    
    \textsc{Case 1(a):} Suppose $\overline{\varepsilon}(\sigma) \geq 2\underline\varepsilon^{\ast}$. Then using the definition of $\underline\varepsilon^{\ast}$ along with the non-increasing property of $\eta\mapsto\log \cMloc(\eta)$ established in \citet[Lemma II.8]{neykov2022minimax}, we have \begin{align*}
    \frac{C^2 (\overline\varepsilon/2)^2}{4\sigma^2} \ge \left(\tfrac{L}{c^{\ast}}\right)^2 \log \cMloc(\overline\varepsilon/2) \ge  (\tfrac{L}{c^{\ast}})^2 \log \cMloc(\overline\varepsilon). 
\end{align*} Taking square roots and multiplying by $\overline{\varepsilon}$, we obtain \begin{align*}
    \frac{C\overline{\varepsilon}^2}{2\sigma} \ge \frac{2L\overline{\varepsilon}}{c^{\ast}} \sqrt{\log \cMloc(\overline\varepsilon) },
\end{align*} so that \begin{align}\label{difference:of:local:widths:compared:to:eps:squared:th:eq1}
   \frac{C \overline \varepsilon^2}{2\sigma}- \frac{L\overline{\varepsilon}}{c^{\ast}}  \sqrt{\log \cMloc(\overline\varepsilon)} \ge \frac{1}{2}\cdot \frac{C\overline \varepsilon^2}{2\sigma}.
\end{align}

Now pick $\nu_1,\nu_2$ that achieves the supremum in the definition of $\overline\varepsilon$, using the fact that $B(\nu_1, \overline\varepsilon/c^{\ast}) \subseteq B(\nu_2, 2\overline\varepsilon + \overline\varepsilon/c^{\ast})$ to conclude $w_{\nu_2}(2 \overline\varepsilon + \overline\varepsilon/c^{\ast}) \ge w_{\nu_1}(\overline\varepsilon/c^{\ast})$. Using this fact along with our definition of $\overline{\varepsilon}$, \eqref{difference:of:local:widths:compared:to:eps:squared:th:eq1}, and $C$, we obtain
\begin{align*}
    \MoveEqLeft w_{\nu_2}(2 \overline\varepsilon + \overline\varepsilon/c^{\ast}) - \frac{(2 + 1/c^{\ast})^2 \overline\varepsilon^2}{2\sigma} \\ & \geq w_{\nu_1}(\overline\varepsilon / c^{\ast}) - \frac{(2 + 1/c^{\ast})^2 \overline\varepsilon^2}{2\sigma} \\
    & \geq w_{\nu_2}(\overline\varepsilon / c^{\ast}) - \frac{(\overline\varepsilon/c^{\ast})^2}{2\sigma}.
\end{align*}
This implies by concavity of $\varepsilon\mapsto w_{\nu_2}(\varepsilon)- \frac{\varepsilon^2}{2\sigma}$ that $ \epswidth[\nu_2](\sigma)\ge \overline{\varepsilon}(\sigma)/c^{\ast}$. Thus, since $\overline\varepsilon(\sigma) \gtrsim \sigma$, we have $\epswidth[\nu_2](\sigma)\gtrsim\sigma$. By Lemma \ref{lemma:chatterjee:analogue}, $\EE\|\hat\mu - \nu_2\|^2\asymp \epswidth[\nu_2](\sigma)^2\ge \overline{\varepsilon}(\sigma)^2/{c^{\ast}}^2$. Taking the supremum, we have $\epsLSE(\sigma) \gtrsim \overline{\varepsilon}(\sigma)$ as claimed.

\textsc{Case 1(b):} Suppose $\overline\varepsilon(\sigma) \in (\underline\varepsilon^{\ast}, 2\underline\varepsilon^{\ast})$. Then by Lemma \ref{lemma:equivalent:information:lower:bound} applied to $\underline\varepsilon^{\ast}$, we have $\overline\varepsilon \asymp \underline\varepsilon^{\ast}\asymp \varepsilon^{\ast}$, i.e., $\overline\varepsilon$ is the minimax rate up to constants. In this case, $\epsLSE(\sigma) \gtrsim \overline\varepsilon$, since the minimax rate is in particular upper-bounded by the LSE rate $\epsLSE(\sigma)$. 

To summarize Case 1, when $\overline\varepsilon(\sigma)\gtrsim\sigma$, we have $\overline{\varepsilon}(\sigma)\lesssim \epsLSE(\sigma) \lesssim \overline{\varepsilon}(\sigma')$.

\textsc{Case 2:} Suppose $\overline\varepsilon(\sigma') \lesssim\sigma$. Then since $\overline\varepsilon(\sigma') \geq \epsSupwidth(\sigma)$ (proved in Case 1 without using the $\overline\varepsilon(\sigma') \gtrsim\sigma$ assumption) we have $\epsSupwidth(\sigma)\lesssim \sigma$. Well if for all $\mu\in K$ we have $\epswidth\lesssim\sigma$, we have by  Lemma \ref{lemma:chatterjee:analogue} that $\EE\|\hat\mu-\mu\|^2\lesssim\sigma^2$ for all $\mu\in K$. Thus, $\epsLSE(\sigma) \lesssim\sigma$. Clearly $\epsLSE(\sigma)\lesssim d$, so $\epsLSE(\sigma)\lesssim \sigma\wedge d$.

That $\epsLSE(\sigma)\gtrsim \sigma\wedge d$ again follows from Lemma \ref{minimax:bound:versus:sigma:and:d} and the minimax rate being a lower bound on the worst case LSE rate. Thus $\epsLSE(\sigma)\asymp \sigma\wedge d$ when $\overline\varepsilon(\sigma') \lesssim\sigma$.
\end{proof}

    \begin{lemma}\label{lemma:concavity:wnu:epsilon}
        For every $\varepsilon$, the map $\nu \mapsto w_\nu(\varepsilon)$ is concave over $K$.\footnotemark{}
    \end{lemma}
    \begin{proof}
        Let $\nu_1,\nu_2 \in K$ and $\nu_3 = \alpha \nu_1 + (1-\alpha)\nu_2 \in K$.  Then \[B(\nu_3,\varepsilon) \cap K \supset \alpha (B(\nu_1,\varepsilon) \cap K) + (1-\alpha)(B(\nu_2,\varepsilon) \cap K),\] so \citet[Proposition 7.5.2(iv)]{vershynin2018high} implies \begin{align*}
            \MoveEqLeft w(B(\nu_3,\varepsilon) \cap K ) \\ &\geq w(\alpha (B(\nu_1,\varepsilon) \cap K) + (1-\alpha)(B(\nu_2,\varepsilon) \cap K))\\ &= \alpha w_{\nu_1}(\varepsilon)  + (1-\alpha)w_{\nu_2}(\varepsilon).
        \end{align*}
    \end{proof} \footnotetext{What is more, one can easily see that the map $(\varepsilon, \nu) \mapsto w_\nu(\varepsilon)$ is concave. The proof is nearly identical to the proof of the original statement so we omit it. We will however use this result later on.}
    
    \begin{proof}[\hypertarget{proof:Lipschitz:map:theorem}{Proof of Theorem \ref{Lipschitz:map:theorem}}]
    The structure of our proof is as follows: In Part I, we show that $\overline\varepsilon(\sigma) \lesssim \epsLSE(\sigma)$ when $\overline{\varepsilon}(\sigma)\gtrsim \sigma$. In Part II, we show that $\overline{\varepsilon}(\sigma')\gtrsim \epsSupwidth(\sigma)$ (where $\epsSupwidth(\sigma)$ was defined following Lemma \ref{lemma:chatterjee:analogue}). In Part III, we use the result of Part II to show that $\epsLSE(\sigma)\lesssim \overline\varepsilon(\sigma')$ and also that if $\overline\varepsilon(\sigma')\lesssim \sigma$, then ${\epsLSE}(\sigma) \asymp \sigma \wedge d$.

    \textsc{Part I}: We prove that $\overline\varepsilon(\sigma) \lesssim \epsLSE(\sigma)$ when $\overline{\varepsilon}(\sigma)\gtrsim \sigma$. For this direction, we abbreviate $\overline\varepsilon:= \overline\varepsilon(\sigma)$. Define  \[\underline\varepsilon^{\ast}:= \underline\varepsilon^{\ast}(\sigma):=\sup\{\varepsilon>0: \tfrac{C^2 \varepsilon^{2}}{\sigma^2} \leq (\tfrac{L}{c^{\ast}})^2 \log \cMloc(\varepsilon)\}.\] By Lemma \ref{lemma:equivalent:information:lower:bound}, $\underline\varepsilon^{\ast}\asymp \varepsilon^{\ast}$. By near identical reasoning to Remark \ref{remark:difference:of:local:widths}, $\overline\varepsilon\ge \underline\varepsilon^{\ast}$ (argue for $\varepsilon \le \underline\varepsilon^{\ast}$ that $\tfrac{L}{c^{\ast}}\sqrt{\log \cMloc(\varepsilon)}\ge \tfrac{C\varepsilon}{\sigma}$, multiply by $\|\nu_1-\nu_2\|$, then compare $\varepsilon$ to $\overline{\varepsilon}$). Thus we consider the cases where $\overline\varepsilon\in (\underline\varepsilon^{\ast}, 2\underline\varepsilon^{\ast})$ and $\overline\varepsilon\ge 2\underline\varepsilon^{\ast}$.

If $\overline\varepsilon \in (\underline\varepsilon^{\ast}, 2\underline\varepsilon^{\ast})$, then $\overline\varepsilon \asymp \underline\varepsilon^{\ast}\asymp \varepsilon^{\ast}$, i.e., $\overline\varepsilon$ is the minimax rate up to constants. Hence $\epsLSE(\sigma) \gtrsim \overline\varepsilon(\sigma)$ as desired, since the minimax rate is a lower bound on the worst case LSE rate. 

We consider the remaining case where $\overline\varepsilon \geq 2 \underline\varepsilon^{\ast}$. Then  an argument similar to that used in Theorem \ref{difference:of:local:widths:compared:to:eps:squared:thm} to derive \eqref{difference:of:local:widths:compared:to:eps:squared:th:eq1} shows that $C \overline \varepsilon/\sigma  > 2 (L/c^{\ast}) \sqrt{\log \cMloc(\overline\varepsilon)}$, and therefore  \begin{align}  \MoveEqLeft \frac{C \overline \varepsilon\|\nu_1 - \nu_2\|}{\sigma} - \frac{L\|\nu_1 - \nu_2\|}{c^{\ast}} \sqrt{\log \cMloc(\overline\varepsilon)} \notag\\ &>  \frac{C\overline \varepsilon \|\nu_1 - \nu_2\|}{2\sigma}.\label{Lipschitz:map:theorem:equation}\end{align}
    Next suppose that we can find $\nu_1$ and $\nu_2$ achieving the supremum in the definition of $\overline\varepsilon$. We consider two subcases and in both demonstrate that $\overline\varepsilon\lesssim \epswidth[\nu_2](\sigma)$.
    
    \textsc{Case 1:} Suppose $\|\nu_1 - \nu_2\| > \overline\varepsilon$. Then we can find a point $\nu_3 = \alpha \nu_1 + (1-\alpha) \nu_2$ such that $\|\nu_3 - \nu_2\| = \alpha \|\nu_1 - \nu_2\|  = \overline\varepsilon$, i.e., by taking $\alpha = \frac{\overline\varepsilon}{\|\nu_1 - \nu_2\|}\in(0,1)$. On the other hand, by the concavity of $\nu \mapsto w_\nu(\varepsilon)$ (Lemma \ref{lemma:concavity:wnu:epsilon}),
    \begin{align}
       w_{\nu_3}(\overline\varepsilon/c^{\ast}) - w_{\nu_2}(\overline\varepsilon/c^{\ast}) & \geq \alpha(w_{\nu_1}(\overline\varepsilon/c^{\ast}) - w_{\nu_2}(\overline\varepsilon/c^{\ast})) \notag \\ &\geq \frac{C \overline\varepsilon^2}{\sigma} - \frac{L\overline\varepsilon }{c^{\ast}} \sqrt{\log \cMloc(\overline\varepsilon)} \notag \\
        &\geq \frac{C \overline\varepsilon^2}{2\sigma}, \label{lipschitz:map:theorem:equation:2}
    \end{align}
    where the second inequality used the definition of $\overline\varepsilon$ and $\alpha$ and the third used \eqref{Lipschitz:map:theorem:equation} (after multiplying by $\overline\varepsilon/\|\nu_1-\nu_2\|$). Thus,
    \begin{align*}
    \MoveEqLeft w_{\nu_2}( \|\nu_3 - \nu_2\| + \overline\varepsilon/c^{\ast}) - \frac{(\|\nu_3 - \nu_2\| + \overline\varepsilon/c^{\ast})^2}{2\sigma} \\ & \geq w_{\nu_3}( \overline\varepsilon/c^{\ast})  - \frac{(\|\nu_3 - \nu_2\| + \overline\varepsilon/c^{\ast})^2}{2\sigma} \\
    &\ge w_{\nu_2}( \overline\varepsilon/c^{\ast}) + \frac{C\overline\varepsilon^2}{2\sigma} -  \frac{(\|\nu_3 - \nu_2\| + \overline\varepsilon/c^{\ast})^2}{2\sigma} \\
    &= w_{\nu_2}( \overline\varepsilon/c^{\ast}) + \frac{C\overline\varepsilon^2}{2\sigma} - \frac{(\overline\varepsilon+ \overline\varepsilon/c^{\ast})^2}{2\sigma} \\
    & =w_{\nu_2}( \overline\varepsilon/c^{\ast}) - \frac{(\overline\varepsilon/c^{\ast})^2}{2\sigma}.
\end{align*} The first inequality used $B(\nu_3, \overline\varepsilon/c^{\ast})\subseteq B(\nu_2, \|\nu_3-\nu_2\|+\overline\varepsilon/c^{\ast})$ which allowed us to compare the Gaussian width terms. The second used \eqref{lipschitz:map:theorem:equation:2} and the third that $\overline\varepsilon = \|\nu_3-\nu_2\|$. The final line holds since $C=1+\frac{2}{c^{\ast}}$. Then by concavity of $\varepsilon\mapsto w_{\nu_2}(\varepsilon)- \varepsilon^2/(2\sigma)$, our resulting inequality shows that $\epsSupwidth(\sigma)\ge \epswidth[\nu_2](\sigma)\ge \overline\varepsilon(\sigma)/c^{\ast}$.
    
    \textsc{Case 2:} Suppose $\|\nu_1 - \nu_2\|\leq \overline\varepsilon$. Then we have
    \begin{align*}
    \MoveEqLeft w_{\nu_2}( \|\nu_1 - \nu_2\| + \tfrac{\overline\varepsilon}{c^{\ast}}) - \frac{(\|\nu_1 - \nu_2\| + \tfrac{\overline\varepsilon}{c^{\ast}})^2}{2\sigma} \\ & \geq w_{\nu_1}(\tfrac{\overline\varepsilon}{c^{\ast}})  - \frac{(\|\nu_1 - \nu_2\| + \tfrac{\overline\varepsilon}{c^{\ast}})^2}{2\sigma} \\
    &\ge w_{\nu_2}( \tfrac{\overline\varepsilon}{c^{\ast}}) +  \frac{C \overline\varepsilon \|\nu_1 - \nu_2\|}{\sigma}   - \frac{L\|\nu_1-\nu_2\|\sqrt{\log \cMloc(\varepsilon)}}{c^{\ast}}   \\ &\qquad -\frac{(\|\nu_1 - \nu_2\| + \tfrac{\overline\varepsilon}{c^{\ast}})^2}{2\sigma} \\
    &\ge  w_{\nu_2}( \tfrac{\overline\varepsilon}{c^{\ast}}) + \frac{C \overline\varepsilon \|\nu_1 - \nu_2\|}{2\sigma} - \frac{(\|\nu_1 - \nu_2\| + \tfrac{\overline\varepsilon}{c^{\ast}})^2}{2\sigma}  \\
    & \geq w_{\nu_2}( \tfrac{\overline\varepsilon}{c^{\ast}}) - \frac{(\tfrac{\overline\varepsilon}{c^{\ast}})^2}{2\sigma}.
\end{align*} 
The first inequality used $B(\nu_1, \tfrac{\overline\varepsilon}{c^{\ast}})\subseteq B(\nu_2, \|\nu_1-\nu_2\|+\tfrac{\overline\varepsilon}{c^{\ast}})$, the second used the definition of $\overline\varepsilon$, and the third our result \eqref{Lipschitz:map:theorem:equation}. For the last inequality, we note that $\|\nu_1 - \nu_2\| \leq \overline\varepsilon$ implies that 
\begin{align*}
    \MoveEqLeft \frac{C \overline\varepsilon \|\nu_1 - \nu_2\|}{2\sigma} - \frac{(\|\nu_1 - \nu_2\| + \overline\varepsilon/c^{\ast})^2}{2\sigma} \\ 
    &= \frac{\overline{\varepsilon} \|\nu_1-\nu_2\|}{\sigma}\cdot \underbrace{\left(\frac{C}{2}- \frac{1}{c^{\ast}}\right)}_{=1/2}- \frac{\|\nu_1-\nu_2\|^2}{2\sigma}  - \frac{(\overline\varepsilon/c^{\ast})^2}{2\sigma}\\
    &= \frac{\overline{\varepsilon} \|\nu_1-\nu_2\|}{2\sigma} - \frac{\|\nu_1-\nu_2\|^2}{2\sigma}  - \frac{(\overline\varepsilon/c^{\ast})^2}{2\sigma}\\
    &\ge \frac{ \overline{\varepsilon}\|\nu_1-\nu_2\|}{2\sigma} - \frac{\overline{\varepsilon}\|\nu_1-\nu_2\|}{2\sigma}  - \frac{(\overline\varepsilon/c^{\ast})^2}{2\sigma}\\
    &=-\frac{(\overline\varepsilon/c^{\ast})^2}{2\sigma}.
\end{align*}

Thus, we have $w_{\nu_2}( \|\nu_1 - \nu_2\| + \overline\varepsilon/c^{\ast}) - \frac{(\|\nu_1 - \nu_2\| + \overline\varepsilon/c^{\ast})^2}{2\sigma}\geq w_{\nu_2}( \overline\varepsilon/c^{\ast}) - \frac{(\overline\varepsilon/c^{\ast})^2}{2\sigma}$ which implies by concavity of $\varepsilon\mapsto w_{\nu_2}(\varepsilon) - \varepsilon^2/(2\sigma)$ that $ \epswidth[\nu_2](\sigma) \ge \overline\varepsilon(\sigma)/c^{\ast}$. This completes Case 2.

Thus, we have shown for both Case 1 and Case 2 of the setting where $\overline{\varepsilon}\ge 2 \underline\varepsilon^{\ast}$ that $ \epswidth[\nu_2](\sigma) \ge \overline\varepsilon/c^{\ast}$. Now if $\overline{\varepsilon}\gtrsim \sigma$ as we assumed, we have $\epswidth[\nu_2](\sigma)\gtrsim\sigma$. Hence by Lemma \ref{lemma:chatterjee:analogue} \[\EE\|\hat{\mu}-\nu_2\|^2 \asymp \epswidth[\nu_2](\sigma)^2 \gtrsim\overline{\varepsilon}(\sigma)^2.\] Taking the sup over $\nu_2$, we conclude $\epsLSE(\sigma)\gtrsim\overline{\varepsilon}(\sigma) $ as desired.

This proves that when $\overline{\varepsilon}(\sigma)\gtrsim \sigma$, then no matter if $\overline{\varepsilon}\in(\underline\varepsilon^{\ast}, 2\underline\varepsilon^{\ast})$ or $\overline{\varepsilon}\ge 2 \underline\varepsilon^{\ast}$ (for which we had Case 1 and Case 2), we have $\epsLSE(\sigma)\gtrsim\overline{\varepsilon}(\sigma)$.

\textsc{Part II.} Next, we prove that $\overline\varepsilon(\sigma')\gtrsim \epsSupwidth(\sigma)$. Take $\varepsilon' = (1 + \kappa)\overline{\varepsilon}(\sigma')$ where $\sigma' = \frac{4C}{1-4/{c^{\ast}}^2}\cdot\sigma$ and $\kappa>0$ is an arbitrary constant. For any $\mu \in K$, consider the set $B(\mu,\varepsilon') \cap K$. Pack the set $B(\mu,\varepsilon') \cap K$ maximally at distance $\varepsilon'/c^{\ast}$. Choose the point $\nu_1$ in the packing that maximizes $w_{\nu_1}(\varepsilon'/c^{\ast})$. By \citet[Exercise 2.4.11]{talagrand2014upper}, we know that \begin{equation} \label{Lipschitz:map:theorem:equation:3}
    w_{\mu}(\varepsilon') \le w_{\nu_1}( \varepsilon'/c^{\ast}) +  \frac{\tilde{L}\varepsilon' }{c^{\ast}}\sqrt{\log \cMloc(\varepsilon')}  
\end{equation} for a sufficiently large absolute constant $\tilde{L}$. 

Similarly, pick $\nu_2$ in the packing to minimize $w_{\nu_2}(\varepsilon'/c^{\ast})$. It should satisfy \begin{equation} \label{Lipschitz:map:theorem:equation:4}w_{\nu_2}(\varepsilon'/c^{\ast})\le  w_{\mu}(2\varepsilon'/c^{\ast} )\end{equation} since if we took $\nu_3$ as the point in the packing set closest to $\mu$, we have $w_{\nu_3}( \varepsilon'/c^{\ast})  \leq w_{\mu}(2\varepsilon'/c^{\ast})$ due to $B(\nu_3,\varepsilon'/c^{\ast}) \subseteq B(\mu, 2\varepsilon'/c^{\ast})$. Thus $\nu_2$ certainly satisfies \eqref{Lipschitz:map:theorem:equation:4}. Observe also that we can always assume that $\|\nu_1 - \nu_2\| \geq \varepsilon'/c^{\ast}$ since if the two coincide, we can take them to be any other two points. That is, unless the packing contains one single point which would imply that $\varepsilon' \gtrsim \operatorname{diam}(K)\gtrsim\epsLSE$, which in turn yields our claim that $\overline{\varepsilon}(\sigma')\gtrsim \epsLSE$ by taking $\kappa\to 0$. To fully resolve the remaining claims in the case when the packing is just a single point, if $\overline{\varepsilon}(\sigma')<\sigma$, this means $\varepsilon'\lesssim\sigma$ so that $\epsLSE\lesssim \sigma$. So $\epsLSE\lesssim \sigma\wedge d$. But also $\epsLSE\gtrsim\varepsilon^{\ast}\gtrsim \sigma\wedge d$ by Lemma \ref{minimax:bound:versus:sigma:and:d}, showing $\epsLSE\asymp \sigma\wedge d$. We now return to the original case where $\|\nu_1 - \nu_2\| \geq \varepsilon'/c^{\ast}$.  Further, since we are packing $B(\mu,\varepsilon')\cap K$, we have $\|\nu_1-\nu_2\|\le 2\varepsilon'$. Additionally, by definition of $\overline\varepsilon(\sigma')$ as a supremum and the fact that $\varepsilon'>\overline\varepsilon(\sigma')$,
\begin{align}\label{Lipschitz:map:theorem:equation:5}
   \MoveEqLeft w_{\nu_1}( \varepsilon'/c^{\ast}) - w_{\nu_2}( \varepsilon'/c^{\ast}) \notag\\ &\leq \frac{C \varepsilon' \|\nu_1 - \nu_2\|}{\sigma'} - \frac{L\|\nu_1 - \nu_2\| \sqrt{\log \cMloc(\varepsilon')}}{c^{\ast}} .
\end{align}

Combining \eqref{Lipschitz:map:theorem:equation:3}, \eqref{Lipschitz:map:theorem:equation:4}, and \eqref{Lipschitz:map:theorem:equation:5} along with $\varepsilon'/c^{\ast} \leq \|\nu_1 - \nu_2\|\leq 2 \varepsilon'$, we conclude that so long as $L\ge c^{\ast}\tilde L$, we have
\begin{align*}
   \MoveEqLeft w_{\mu}( \varepsilon')  - w_{\mu}( 2\varepsilon'/c^{\ast}) \\ &= ( w_{\mu}( \varepsilon') - w_{\nu_1}(\varepsilon'/c^{\ast}))+(w_{\nu_1}(\varepsilon'/c^{\ast})- w_{\nu_2}(\varepsilon'/c^{\ast})) \\ &\quad\quad + \underbrace{(w_{\nu_2}(\varepsilon'/c^{\ast}) -w_{\mu}( 2\varepsilon'/c^{\ast}))}_{\le 0}  \\
   &\le  \frac{\tilde L \varepsilon'}{c^{\ast}}  \sqrt{\log \cMloc(\varepsilon')} + \frac{C \varepsilon' \|\nu_1 - \nu_2\|}{\sigma'}  \\ &\qquad\qquad -  \frac{L \|\nu_1-\nu_2\|}{c^{\ast}}  \sqrt{\log \cMloc(\varepsilon')}  \\
   &= \frac{1}{c^{\ast}}\sqrt{\log \cMloc(\varepsilon')} \cdot\left( \tilde{L}\varepsilon'-L\|\nu_1-\nu_2\|\right)  \\ &\qquad\qquad + \underbrace{\frac{C \varepsilon' \|\nu_1 - \nu_2\|}{\sigma'}}_{\le 2C \varepsilon'^2/\sigma'} \\
   &\le \frac{1}{c^{\ast}}\sqrt{\log \cMloc(\varepsilon')} \cdot\underbrace{( \tilde{L}\varepsilon' -L\varepsilon'/c^{\ast})}_{\le 0} + \frac{2C\varepsilon'^2}{\sigma'} \\
   &\le \frac{2C \varepsilon'^2}{\sigma'} \\ &= \frac{\varepsilon'^2}{2\sigma}  -\frac{(2\varepsilon'/c^{\ast})^2}{2\sigma}.
\end{align*} This implies for our choice of $\sigma'$ that \begin{align*}
    w_{\mu}(\varepsilon') - \frac{\varepsilon'^2}{2\sigma} \le w_{\mu}( 2\varepsilon'/c^{\ast}) - \frac{(2\varepsilon'/c^{\ast})^2}{2\sigma}.  
\end{align*} Using concavity of $\varepsilon\mapsto w_{\mu}(\varepsilon)-\frac{\varepsilon^2}{2\sigma}$ and the fact that $\varepsilon'> 2\varepsilon'/c^{\ast}$ for $c^{\ast}>2$, we conclude $\varepsilon'\gtrsim \epswidth(\sigma)$. Taking $\kappa\to 0$, we have $\overline\varepsilon(\sigma')\gtrsim \epswidth(\sigma)$. Since we took an arbitrary $\mu\in K$, we have $\overline\varepsilon(\sigma')\gtrsim \epsSupwidth(\sigma)$. 

\textsc{Part III:} Finally, we show that $\epsLSE \lesssim \overline\varepsilon(\sigma')$ and if $\overline\varepsilon(\sigma')\lesssim \sigma$, then $\epsLSE\asymp\sigma\wedge d$. 

We start with the first claim. Let $\tilde\mu\in K$ maximize $\EE\|\hat\mu-\tilde\mu\|^2.$ Suppose $\epswidth[\tilde\mu](\sigma)\gtrsim \sigma$. Then using the definition of $\tilde\mu$, Lemma \ref{lemma:chatterjee:analogue}, and the result of Part II, we have \begin{align*}
    \epsLSE^2 = \EE\|\hat\mu-\tilde\mu\|^2 \asymp \epswidth[\tilde\mu](\sigma)^2  \le \epsSupwidth(\sigma)^2 \lesssim \overline{\varepsilon}(\sigma')^2.
\end{align*} On the other hand, if $\epswidth[\tilde\mu](\sigma)\lesssim \sigma,$ then by Lemma \ref{lemma:chatterjee:analogue} \begin{align*}
    \epsLSE^2 = \EE\|\hat\mu-\tilde\mu\|^2\lesssim \sigma^2.
\end{align*} But also $\epsLSE^2\lesssim d^2$ hence $\epsLSE^2\lesssim \sigma^2\wedge d^2$. Now recall from Part I that $\overline{\varepsilon}(\sigma')\ge \underline{\varepsilon}^{\ast}(\sigma')$ (we used the scaling $\sigma$ rather than $\sigma'$ in Part I). Using these facts along with  Lemma \ref{minimax:bound:versus:sigma:and:d} to  relate $\underline{\varepsilon}^*(\sigma')$ (which is the minimax rate up to constants) to $\sigma'\wedge d$, we obtain \[\epsLSE^2 \lesssim \sigma^2\wedge d^2 \asymp {\sigma'}^2\wedge d^2 \lesssim (\underline{\varepsilon}^*(\sigma'))^2 \lesssim \overline{\varepsilon}(\sigma')^2.\]

We now prove the final claim. If $\overline\varepsilon(\sigma')\lesssim \sigma$, we have from Part II that $\epsSupwidth(\sigma)\lesssim \overline\varepsilon(\sigma')\lesssim \sigma$. By Lemma \ref{lemma:chatterjee:analogue} (applied at every $\mu\in K$), this implies $\epsLSE \lesssim \sigma$ and since $\epsLSE\le d$ as previously argued, we conclude $\epsLSE \lesssim \sigma\wedge d$. Then again by Lemma \ref{minimax:bound:versus:sigma:and:d} and the minimax rate being a lower bound on the worst case LSE rate, we have $\epsLSE\gtrsim \varepsilon^{\ast}\gtrsim \sigma\wedge d$. Hence if $\overline\varepsilon(\sigma')\lesssim \sigma$, then $\epsLSE\asymp\sigma\wedge d$. This concludes Part III, which in conjunction with Part I, completes the proof.
\end{proof}

    \begin{proof}[\hypertarget{proof:corollary:Lipschitz}{Proof of Corollary \ref{corollary:Lipschitz}}]
    We first prove two intermediate claims. Let $\overline{\varepsilon}(\sigma)$ be as defined in Theorem \ref{Lipschitz:map:theorem}.
    
\textsc{Claim 1:} $\overline{\varepsilon}(\sigma) \ge \varepsilon^{\ast}$. Proof: if $\varepsilon \le \varepsilon^{\ast}$, then $\sqrt{\log \cMloc(\varepsilon)}\ge \varepsilon/\sigma$ by definition of $\varepsilon^{\ast}$. Assuming $\nu_1,\nu_2$ are chosen without loss of generality so that $w_{\nu_1}(\varepsilon/c^{\ast}) \ge w_{\nu_2}(\varepsilon/c^{\ast})$, we have  \begin{align*}
    \MoveEqLeft w_{\nu_1}(\varepsilon/c^{\ast}) - w_{\nu_2}(\tfrac{\varepsilon}{c^{\ast}}) - \frac{C \varepsilon \|\nu_1 - \nu_2\|}{\sigma} \\ &\qquad\qquad +  \frac{L\|\nu_1 - \nu_2\|}{c^{\ast}} \sqrt{\log \cMloc(\varepsilon)} \\ &\ge  w_{\nu_1}(\varepsilon/c^{\ast}) - w_{\nu_2}(\varepsilon/c^{\ast}) - \frac{C \varepsilon \|\nu_1 - \nu_2\|}{\sigma} \\ &\qquad\qquad  +  \frac{L\varepsilon \|\nu_1 - \nu_2\|}{c^{\ast}\sigma}  \\
    &= w_{\nu_1}(\varepsilon/c^{\ast}) - w_{\nu_2}(\varepsilon/c^{\ast}) + \frac{\overline{C}\varepsilon\|\nu_1-\nu_2\|}{\sigma}   \\
    &\ge 0,
\end{align*} provided $L$ is sufficiently large. This implies the condition in the definition of $\overline{\varepsilon}(\sigma)$ holds, so $\overline{\varepsilon}(\sigma) \ge \varepsilon$. This holds for all $\varepsilon \le \varepsilon^{\ast}$, so we conclude $\overline{\varepsilon}(\sigma) \ge \varepsilon^{\ast}$.

\textsc{Claim 2:} Suppose the map $\mu\mapsto w_{\mu}(\varepsilon/c^{\ast})$ has a Lipschitz constant $\overline{C}\varepsilon/\sigma$ for all $\varepsilon>\varepsilon^{\ast}$ for some constant $\overline{C}$. Then $\overline{\varepsilon}(\sigma)\lesssim \varepsilon^{\ast}$. Proof: If $\varepsilon>\varepsilon^{\ast}$, then $(\varepsilon/\sigma)\|\nu_1-\nu_2\|\ge \|\nu_1-\nu_2\|\sqrt{\log \cMloc(\varepsilon)}$ by definition of $\varepsilon^{\ast}$. Hence for sufficiently small $\overline{C}$, we have \begin{align*}
    \MoveEqLeft |w_{\nu_1}(\varepsilon/c^{\ast})- w_{\nu_2}(\varepsilon/c^{\ast})| \\ &\le \frac{\overline{C}\varepsilon}{\sigma}\|\nu_1-\nu_2\| \\ 
    &\lesssim \frac{\varepsilon}{\sigma}\|\nu_1-\nu_2\| - \frac{L\varepsilon}{c^{\ast}\sigma} \|\nu_1 - \nu_2\| \\ &\le \frac{\varepsilon}{\sigma}\|\nu_1-\nu_2\| - \frac{L}{c^{\ast}} \|\nu_1 - \nu_2\| \sqrt{\log \cMloc(\varepsilon)}. 
\end{align*} Then we have $\overline{\varepsilon}(\sigma)\le \varepsilon$ by definition of $\overline{\varepsilon}(\sigma)$ in Theorem \ref{Lipschitz:map:theorem}, but since this holds for all $\varepsilon >\varepsilon^{\ast}$, this implies $\overline{\varepsilon}(\sigma)\le \varepsilon^{\ast}$.

Using these claims, let us prove our main claim. Suppose the map has a Lipschitz constant $\varepsilon/\sigma$  up to constants for all $\varepsilon\gtrsim\varepsilon^{\ast}$. Now let $\sigma' = \sigma\cdot \tfrac{4C}{1-4/{c^{\ast}}^2}\asymp \sigma$. Let $\varepsilon^{\ast}(\sigma')$ be as defined in \eqref{varepsilon:star:def} but with $\sigma'$ replacing $\sigma$. By Lemma \ref{lemma:equivalent:information:lower:bound}, $\varepsilon^{\ast}\asymp \varepsilon^{\ast}(\sigma')$. Thus, we conclude that the map has a Lipschitz constant $\varepsilon/\sigma'$ (up to constants) for all $\varepsilon\gtrsim \varepsilon^{\ast}(\sigma')$. Then combining Claim 1 and Claim 2, we know $\overline{\varepsilon}(\sigma')\asymp  \varepsilon^{\ast}(\sigma') \asymp\varepsilon^{\ast}$. But by the theorem we have $\epsLSE(\sigma)\lesssim \overline{\varepsilon}(\sigma')\asymp\varepsilon^{\ast}$. Thus the LSE is minimax optimal. 

Conversely, suppose the LSE is minimax optimal. Suppose $\overline{\varepsilon}(\sigma)\gtrsim\sigma$. Then by  Theorem \ref{Lipschitz:map:theorem}, $\overline{\varepsilon}(\sigma)\lesssim \epsLSE(\sigma)\asymp \varepsilon^{\ast}$. Note that $\mu \mapsto w_{\mu}(\varepsilon)$ must be $\overline{\varepsilon}(\sigma) / \sigma$-Lipschitz (up to constants) by definition of $\overline{\varepsilon}$. Then since increasing the Lipschitz constant preserves the Lipschitz property, we know the map is $\varepsilon^{\ast}/\sigma$-Lipschitz. Hence, for any $\varepsilon\gtrsim \varepsilon^{\ast}$, the map is $\varepsilon/\sigma$-Lipschitz. 

Suppose instead $\overline{\varepsilon}(\sigma)\lesssim\sigma$, so that the mapping has a Lipschitz constant $\overline{\varepsilon}(\sigma)/\sigma\lesssim 1$. If we can show $\varepsilon^{\ast}\gtrsim \sigma$, so that $\overline{\varepsilon}(\sigma)/\sigma\lesssim 1 \lesssim \varepsilon^{\ast}/\sigma$, then the mapping will also be Lipschitz with constant $\varepsilon^{\ast}/\sigma$ and hence $\varepsilon/\sigma$-Lipschitz for any $\varepsilon>\varepsilon^{\ast}$. Well, by Claim 1, $\varepsilon^{\ast}\lesssim \overline{\varepsilon}(\sigma)\lesssim \sigma$. But also $\varepsilon^{\ast}\lesssim d$, so $\varepsilon^{\ast}\lesssim \sigma\wedge d$. We also know from Lemma \ref{minimax:bound:versus:sigma:and:d} that $\varepsilon^{\ast}\gtrsim \sigma \wedge d$, so in fact $\varepsilon^{\ast}\asymp \sigma\wedge d$ . Now if $\sigma\lesssim d$, then this means $\varepsilon^{\ast}\asymp \sigma$, so certainly $\varepsilon^{\ast}\gtrsim \sigma$. Suppose $\sigma \gtrsim d$, so that $\varepsilon^{\ast}\asymp d$. Well observe that if $\varepsilon\gtrsim \varepsilon^{\ast}$, i.e., $\varepsilon\gtrsim d$ (with a constant larger than $1$), then \[|w_{\nu_1}(\varepsilon)-w_{\nu_2}(\varepsilon)|= |w(K)-w(K)|=0\le \tfrac{\varepsilon}{\sigma}\|\nu_1-\nu_2\|_2.\] Here we used that $B(\nu_1,d)\cap K = B(\nu_2,d)\cap K = K$. Hence the mapping is $\varepsilon/\sigma$-Lipschitz for all $\varepsilon\gtrsim \varepsilon^{\ast}$ as required.
    \end{proof}

\section{Proofs for Section \ref{section:LSE:examples}}
\subsection{Proofs for Section \ref{subsection:example:optimal:LSE}}

    \begin{proof}[\hypertarget{proof:isotonic:lower:bound}{Proof of Lemma \ref{isotonic:lower:bound}}]
 We will take $a = 0, b = 1$ for simplicity. Suppose $\varepsilon \gtrsim \frac{1}{\sqrt{n}}$. Set $k=\varepsilon\sqrt{n}>1$ and assume for simplicity both $k$ and $n/k = \sqrt{n}/\varepsilon$ are integers. Construct the vector 
\begin{align*}
u &= (\underbrace{0, \ldots, 0}_{k},\underbrace{\tfrac{k}{n}, \ldots, \tfrac{k}{n}}_{k}, \underbrace{\tfrac{2k}{n}, \ldots, \tfrac{2k}{n}}_{k},  \ldots,\\ &\qquad\qquad \underbrace{1-\tfrac{k}{n}, \ldots,  1-\tfrac{k}{n}}_{k})\in\RR^n.
\end{align*} Now consider perturbing this vector by
\begin{align*}
v_\alpha &= (\underbrace{0, \ldots, 0}_{k},\underbrace{\tfrac{\alpha_1 k}{n} , \ldots, \tfrac{\alpha_1 k}{n}}_{k}, \underbrace{\tfrac{\alpha_2 k}{n}, \ldots,\tfrac{\alpha_2 k}{n}}_{k},\ldots,\\ &\qquad\qquad   \underbrace{\tfrac{\alpha_{n/k-2} k}{n}, \ldots, \tfrac{\alpha_{n/k-2} k}{n}}_{k},  \underbrace{0, \ldots, 0}_{k}),
\end{align*}
where the vector $\alpha = (\alpha_1,\alpha_2, \ldots, \alpha_{n/k-2}) \in \{0,1\}^{n/k-2}$, noting that $u+v_{\alpha}$ is monotonic and has components in $[0,1]$. By Varshamov-Gilbert's bound, there exists at least $\exp((n/k-2)/8)$ vectors from $\{0,1\}^{n/k-2}$ such that the Hamming distance between any of them is at least $(n/k-2)/4$. Given $\alpha,\alpha'\in \{0,1\}^{n/k-2}$ with this property, note that
\begin{align*}
\|v_{\alpha} - v_{\alpha'}\|^2 \geq k (\varepsilon/\sqrt{n})^2 (n/k - 2)/8 \gtrsim \varepsilon^2.
\end{align*} 
Therefore, there exist at least $\exp((n/k-2)/8)\asymp\exp(\sqrt{n}/\varepsilon)$ vectors $u + v_{\alpha}$ that form an $\varepsilon$-packing set of $S^{\uparrow}(a,b)$. 
\end{proof}

    \begin{proof}[\hypertarget{proof:multivariable_isotone_upper}{Proof of Lemma \ref{multivariable_isotone_upper}}]
   We present the $p> 2$ case only, with the $p=2$ case identical except we swap the choice of upper bound from \citet[Theorem 1.1]{gao_Wellner_monotone}. Let $\|\cdot\|_{L^2(\mu)}$ denote the $L^2$-norm of a function on $\RR^p$ with respect to the Lebesgue measure in $p$-dimensions. 
   
   We begin with $a=0$ and $b=1$. Let $\varepsilon>0$. Take $\delta =\frac{\varepsilon}{ 2\sqrt{n}}$. By Lemma \ref{gao_Wellner_theorem}, there exists a $\delta$-covering $\cG_p$ of $\cF_p$ with $|\cG_p| \le C\delta^{-2(p-1)}$.

   Now, for each $\mu \in Q_{0,1}\subseteq [0,1]^n$, we construct a corresponding $f^{\mu}\in \cF_p$. We must partition the domain of $f^{\mu}$, i.e., $[0,1]^p$.  Let
   \begin{align*}
       I^{-} &= \{x=(x_1,x_2,\dots,x_p):x\in[0,1]^p, \\ &\qquad \exists i\in \{1,\dots,p\} \text{ such that }x_i= 0 \}.
   \end{align*} The $n$ unique elements $l_1,\dots,l_n$ of $\LL_{p,n}$ define $n$ cubes that partition $(0,1]^p$ as follows. Given $l_j = \left(\frac{k_1^j}{n^{1/p}},\dots, \frac{k_p^j}{n^{1/p}}  \right)$ for $k_1^j,\dots,k_{p}^j\in \{1,\dots, n^{1/p}\}$, define the corresponding cube $I_j= \prod_{i=1}^p \left(\frac{k_i^j - 1}{n^{1/p}}, \frac{k_i^j}{n^{1/p}}\right]$ that has $l_j$ as its outermost corner. Then $[0,1]^p=I^-\cup \bigcup_{j=1}^n I_j$ is a disjoint partition. Thus, we define for all $x\in[0,1]^p$ the function \[f^{\mu}(x) =\begin{cases}
        \mu_j &\text{ if }x \in I_j \\
        0 &\text{ if } x\in I^-.
   \end{cases}\] Let us verify that $f^{\mu}$ is indeed an element of $\cF_p$. Clearly $f^{\mu}(x)$ lies in $[0,1]$ for all $x\in[0,1]^p$. Let $e^i\in\RR^p$ be the unit vector with $1$ in coordinate $i$ and 0 in all other coordinates. To prove monotonicity, let $x=(x_1,\dots,x_p)\in [0,1]^p$ and consider the following cases. 
   
   Suppose $x\in I^{-}$. Pick any $\alpha >0$ and coordinate $s\in \{1,\dots,p\}$ such that $x+\alpha e^s\in [0,1]^p$ (such an $\alpha$ and $s$ must exist since $x$ is not the all-ones vector). Then $x+\alpha e^s$ is in some cube $I_j$ or still in $I^-$, so $f^{\mu}$ takes value $\mu_j\ge0$ or $0$, both options are $\ge0=f^{\mu}(x)$ (preserving monotonicity).
   
    Suppose $x \in I_j$ for some $j$, i.e., $x_i \in \left(\frac{k_i^j - 1}{n^{1/p}}, \frac{k_i^j}{n^{1/p}}\right]$ for each $1\le i \le p$. Then $f^{\mu}(x)=\mu_j$. Assume $x$ is not the all-ones vector. Now again consider $x+\alpha e^s$ for some $s\in\{1,\dots,p\}$ where $\alpha>0$ and $x+\alpha e^s\in[0,1]^p$. Then $x+\alpha e^s$ belongs to some cube $I_t$ where $l_t$  has the same coordinates as $l_j$ except a possibly larger $s$th coordinate, and $f^{\mu}(x+\alpha e^s)=\mu_t$. Since $\mu\in Q_{0,1}$, for some $g\in\cF_p$, $\mu_t = g(l_t)$ and $\mu_j = g(l_j)$. But since $l_t$ and $l_j$ differ only in the $s$th coordinate and $g$ is non-decreasing in each coordinate, we conclude $\mu_j \le \mu_t$. So $f^{\mu}(x)\le f^{\mu}(x+\alpha e^s)$. If $x$ is the all-ones vector, then repeat the same logic but compare with $x-\alpha e^s\in[0,1]^p$. This verifies that $f^{\mu}\in\cF_p$, so we have a well-defined mapping from $Q_{0,1}$ to $\cF_p$.

   Recalling our $\delta$-covering $\cG_p$, for each $g\in \cG_p$, if it exists, pick a single $\mu(g)\in Q_{0,1}$ such that $\|f^{\mu(g)} - g\|_{L^2(\mu)}<\delta$. We then produce a subset $Q_{0,1}' =\{\mu(g):g\in \cG\},$ with $|Q_{0,1}'|\le |\cG_p|\le C\delta^{-2(p-1)}$. 

   Let us verify that $Q_{0,1}'$ is indeed an $\varepsilon$-covering of $Q_{0,1}$. Pick any $\mu\in Q_{0,1}$. Then $f^{\mu}$ is an element of $\cF_p$, and since $\cG_p$ is a $\delta$-covering of $\cF_p$, there exists $g\in\cG_p$ such that $\|f^{\mu} - g\|_{L^2(\mu)}<\delta$. This also implies the existence of $\mu(g)$ in $Q_{0,1}'$ such that $\|f^{\mu(g)} - g\|_{L^2(\mu)}<\delta$, possibly equal to $\mu$ itself. By the triangle inequality, $\|f^{\mu} - f^{\mu(g)}\|_{L^2(\mu)}<2\delta$. Note that $f^{\mu}$ and $f^{\mu(g)}$ are constant on each cube $I_1,\dots, I_n$ and each of these cubes has volume $1/n$. Moreover, $I^-$ has volume $0$ with respect to the Lebesgue measure. Hence \begin{align*}
       \underbrace{\|f^{\mu} - f^{\mu(g)}\|_{L^2(\mu)}^2}_{<4\delta^2} &= \int_{[0,1]^p}(f^{\mu}(x) - f^{\mu(g)}(x))^2 \mathrm{d}\mu(x) \\
       &= \sum_{j=1}^n \int_{I_j}(f^{\mu}(x) - f^{\mu(g)}(x))^2 \mathrm{d}\mu(x) \\
       &= n^{-1}\sum_{j=1}^n [\mu_j - (\mu(g))_j]^2 \\
       &= n^{-1}\|\mu - \mu(g)\|_2^2.
   \end{align*}  Thus, $\|\mu - \mu(g)\|_2 \le 2\sqrt{n}\delta=\varepsilon$, so $Q_{0,1}'$ is indeed an $\varepsilon$-covering of $Q_{0,1}$ of cardinality bounded by $C\delta^{-2(p-1)}= C\left(\frac{\varepsilon}{2\sqrt{n}}\right)^{-2(p-1)}$. In the case $p =2$, we instead have $|Q_{0,1}'|\le C\delta^{-2}(\log 1/\delta)^2= C\left(\frac{\varepsilon}{2\sqrt{n}}\right)^{-2}\left(\log \frac{2\sqrt{n}}{\varepsilon}\right)^2$. We have therefore upper-bounded the size of a minimal $\varepsilon$-covering of $Q_{0,1}$.

   Now we consider arbitrary $a<b$, and again assume $p>2$. Let $\varepsilon>0$. Consider the linear function $\phi(x)=\frac{x}{b-a}-\frac{a}{b-a}$ that satisfies $\phi(a)=0$ and $\phi(b)=1$ and has inverse $\phi^{-1}(x) =(b-a)x+a$. Let $\Phi \colon\RR^n\to\RR^n$ apply $\phi$ coordinate-wise, and let $\Phi^{-1}\colon\RR^n\to\RR^n$ apply $\phi^{-1}$ coordinate-wise instead. Observe that $\Phi(Q_{a,b})\subseteq Q_{0,1}$ since if $f\colon[0,1]^p\to[a,b]$ is non-decreasing in each variable, then $\phi\circ f\colon[0,1]^p\to[0,1]$ is still non-decreasing in each variable and thus belongs to $\cF_p$. Therefore, using our earlier result, there exists an $\frac{\varepsilon}{b-a}$-covering $Q_{a,b}'$ of $Q_{0,1}$ and therefore of $\Phi(Q_{a,b})$ with $|Q_{a,b}'|\le C\left(\frac{\varepsilon}{2\sqrt{n}(b-a)}\right)^{-2(p-1)}$. 
   
   Now observe that $\Phi^{-1}(Q_{a,b}')$ is a subset of $Q_{a,b}$, with cardinality $|\Phi^{-1}(Q_{a,b}')|\le C\left(\frac{\varepsilon}{2\sqrt{n}(b-a)}\right)^{-2(p-1)}$. Let us verify that $\Phi^{-1}(Q_{a,b}')$ is an $\varepsilon$-covering of $Q_{a,b}$. Pick any $\mu\in Q_{a,b}$. Then for some $f\colon[0,1]^p\to[a,b]$ non-decreasing in each variable, $\mu=(f(l_1),\dots,f(l_n))$. Then $\Phi(\mu)\in \Phi(Q_{a,b})$ so there is some $\eta \in Q_{a,b}'$ with $\|\Phi(\mu)-\eta\|_2 \le \frac{\varepsilon}{b-a}$. Applying the definition of $\Phi$ and rearranging shows $\|\mu-\Phi^{-1}(\eta)\|_2 \le\varepsilon$, and $\Phi^{-1}(\eta)$ belongs to the claimed covering set $ \Phi^{-1}(Q_{a,b}')$. 
   
   The $p=2$ case is similar, and we instead obtain a covering of cardinality upper-bounded by $C\left(\frac{\varepsilon}{2\sqrt{n}(b-a)}\right)^{-2}\left(\log \frac{2\sqrt{n}(b-a)}{\varepsilon}\right)^2$.
\end{proof}

    \begin{proof}[\hypertarget{proof:multivariable_isotone_lower}{Proof of Lemma \ref{multivariable_isotone_lower}}]
        We first prove the claim for $a=0$ and $b=1$. We re-use the notation from the proof of Lemma \ref{multivariable_isotone_upper}, recalling our cubes $\cI_n=\{I_1,\dots, I_n\}$ of side lengths $1/n^{1/p}$ and volume $1/n$. We will show for $\varepsilon \gtrsim 1/n^{1/p}$ that $\cM\left(\sqrt{\frac{ c_p n\varepsilon}{16}}, Q_{0,1}\right) \gtrsim 2^{c_p\varepsilon^{1-p}/2}$ where $c_p$ is some constant depending on $p$. Performing a change of variables, this will imply $\log \cM\left(\varepsilon, Q_{0,1}\right) \gtrsim \left(\tfrac{\varepsilon}{2\sqrt{n}}\right)^{-2(p-1)}$ for $\varepsilon \gtrsim  \sqrt{\frac{n}{n^{1/p}}}$ as desired.

        Assume $\varepsilon\gtrsim 1/n^{1/p}$. We assume $\varepsilon = 2^{-m}$ for some $m\in \NN$, noting that $\varepsilon <1$. For simplicity, we assume $n^{1/p}=k2^m$ for some $k\in\NN$, i.e., $\varepsilon = k/n^{1/p}$. This means the cubes in $\cI_n$ have side lengths $2^{-m}/k$. Let's instead partition $[0,1]^p$ into a coarser set of cubes by taking cubes with side length $2^{-m}$ instead of $2^{-m}/k$. To do so, define  $\varepsilon^{-p}=2^{mp}$ cubes $\cJ_m=\{J_1,\dots, J_{2^{mp}}\}$ of the form $\prod_{i=1}^p (r_i\varepsilon, (r_i+1)\varepsilon]$ where $r_i\in\{0,1,\dots,2^{m}-1\}$. Each cube in $\cJ_m$ contains exactly $k^p=n/2^{mp}\in\NN$ cubes of $\cI_n$, and the boundary of cubes in $\cJ_m$ are a subset of boundaries of cubes in $\cI_n$. Note that $[0,1]^p$ can be partitioned as $I^-\cup\bigcup_{i=1}^{2^{mp}} J_{i}$, where $I^-$ was defined in the proof of Lemma \ref{multivariable_isotone_upper}.
     
        Now, we partition the cubes in $\cJ^m$ as follows: let $\overline{\cJ}_m$ be the set of cubes with $\sum_{i=1}^p r_i= 2^m$, $\overline{\cJ}_m^{+}$ the cubes with $\sum_{i=1}^p r_i> 2^m$, and $\overline{\cJ}_m^{-}$ the cubes with $\sum_{i=1}^p r_i< 2^m$.
        
        Consider the set $\cG_p$ of functions $g$ on $[0,1]^p$ such that $g$ is either $1$ or $0$ in each cube in $\overline{\cJ}_m$ and $g$ is identically $0$ on the cubes in $\overline{\cJ}_m^{-}$ and identically 1 on the cubes in $\overline{\cJ}_m^{+}$. We set $g$ to be value $0$ on the set $I^{-}$. Note that $|\cG_p|=2^{|\overline{\cJ}_m|}$. 
        
        Now let us prove that $\cG_p\subseteq\cF_p$. Clearly $g\in\cG_p$ takes values in $[0,1]$. To show the non-decreasing property, it suffices to enumerate all possible pairs of points where $g$ takes value $0$ on one point and $1$ on the other and argue this does not violate monotonicity. First, consider $g$ taking value $0$ on a cube in $\overline{\cJ}_m^{-}$ and $1$ on a cube in $\overline{\cJ}_m^{+}$. Well if $g$ is $1$ on a cube $\prod_{i=1}^p (r_i\varepsilon, (r_i+1)\varepsilon]\in \overline{\cJ}_m^+$ with $\sum_{i=1}^p r_i > 2^m$, moving from that cube in the (positive) direction of coordinate $s$ will never lead to a cube in $\overline{\cJ}_m^-$ for any $s$ since the value of $r_i$ can not decrease. So we cannot obtain the needed comparable point to conclude monotonicity was violated. By near identical logic, $g$ taking value $0$ in $\overline{\cJ}_m^{-}$ and $1$ on a cube in $\overline{\cJ}_m$ cannot lead to a violation, and neither will taking value $0$ in $\overline{\cJ}_m$ and $1$ on $\overline{\cJ}_m^{+}$. Next, a violation within $\overline{\cJ}_m$ cannot occur. This is because if we are at a cube with $\sum_{i=1}^p r_i=2^m$ (where $g$ is either $0$ or $1$) and move  along the positive $s$ coordinate for any $s$ until a new cube is reached, the new cube will have its defining coordinates satisfying $\sum_{i=1}^p r_i'>2^m$. But $g$ will be $1$ on this new cube, meaning no violation occurred.  Lastly, $g$ taking value $0$ on $I^-$ does not cause an issue. For if $g$ takes value $1$ at a point, it does not belong to $I^-$, i.e., that point does not have a 0 coordinate. But moving in the positive direction in coordinate $s$ will preserve this fact for any $s$, so we cannot arrive at $I^-$, i.e., a violation cannot occur. Hence $\cG_p\subseteq\cF_p$.
        
        We can show $|\overline{\cJ}_m| \ge c_p \cdot \varepsilon^{1-p}$, where $c_p$ is a constant depending only on $p$. To see this, note that $|\overline{\cJ}_m|$ is the number of integer tuples $(r_1,\dots,r_p)$ with $\sum_{i=1}^p r_i=2^m=\varepsilon^{-1}$ and $0\le r_i\le 2^{m-1}$. Setting $r_i'=1+r_i$, this is equivalent to counting the number of integer tuples $(r_1',\dots,r_p')$ with $\sum_{i=1}^p r_i'=2^m+p$ and $1\le r_i'\le 2^m$. By a standard stars-and-bars combinatorial argument, $|\overline{\cJ}_m| = \binom{2^m+p-1}{p-1}=\binom{\varepsilon^{-1}+p-1}{p-1}$. Then using a well-known binomial coefficient bound, \begin{align*}
            |\overline{\cJ}_m| &\ge \frac{(\varepsilon^{-1}+p -1 )^{p-1}}{(p-1)^{p-1}}  \ge \frac{(\varepsilon^{-1})^{p-1}}{(p-1)^{p-1}}. 
        \end{align*}
        Indeed, $|\overline{\cJ}_m| \ge c_p \cdot \varepsilon^{1-p}$ with $c_p$ only depending on $p$ as claimed.
        
        Observe that there is a bijective correspondence between $\cG_p$ and $\{0,1\}^{|\overline{\cJ}_m|}$. We apply a Varshamov-Gilbert type bound to this set of binary strings, and equivalently $\cG_p$, following the formulation of \citet[Lemma 4.12]{rigollet2023highdimensional} with $\gamma = 1/4$. The lemma states that there exists $g_1,\dots,g_N\in \cG_p$ with $N\ge 2^{|\overline{\cJ}_m|/(32\log 2)} \ge 2^{c_p\varepsilon^{1-p}/(32\log 2)}$ such that for any $i\ne j$, the functions $g_i,g_j$ disagree on at least $\frac{|\overline{\cJ}_m|}{4}\ge \frac{c_p \varepsilon^{1-p}}{4}$ many cubes in $\cJ_m$. Let $\cJ_m^{i,j}$ be these cubes for each $i\ne j$, so that $|\cJ_m^{i,j}|\ge  \frac{c_p \varepsilon^{1-p}}{4}$. Since each cube of $\cJ_m$ is partitioned into cubes of $\cI_n$, we can define $\cI_n^{i,j}$ as the cubes of $\cI_n$ on which $g_i$ and $g_j$ disagree, so that the cubes of $\cJ_m^{i,j}$ and $\cI_n^{i,j}$ both partition the same region.

        Each $g_i$ defines a unique point $\mu_i = (g_i(l_1),\dots, g_i(l_n))\in Q_{0,1}\subseteq \RR^n$ where $l_1,\dots,l_n$ are the lattice points. Let us compute  $\|\mu_i-\mu_j\|_2$ for $i\ne j$ to verify we have formed a packing set of $Q_{0,1}$. Note that $g_i-g_j$ is constant within any cube in $\cI_n$, the difference being either $0$ or $1$. By counting the number of cubes such that $g_i(l)\ne g_j(l)$, we count the number of non-zero coordinates of $\mu_i-\mu_j$. Recalling that each cube in $\cI_n$ has volume $n^{-1}$ and each cube in $\cJ_m$ volume $\varepsilon^p$, we have 
        \begin{align*}
            \frac{\|\mu_i - \mu_j\|_2^2}{n} 
            &=\sum_{I \in \cI_n^{i,j}}\int_I  \underbrace{|g_i(x)-g_j(x)|^2}_{=1}  \mathrm{d} x \\ &=  \sum_{J \in \cJ_m^{i,j}}\int_J  \underbrace{|g_i(x)-g_j(x)|^2}_{=1} \mathrm{d}x \\ 
            &= |\cJ_m^{i,j}|\cdot \varepsilon^p \ge \frac{c_p \varepsilon^{1-p}}{4} \cdot  \varepsilon^p 
            = \frac{c_p\varepsilon}{4}.
            \end{align*}
            Therefore, $\|\mu_i-\mu_j\|_2 \ge \sqrt{\frac{ c_p n\varepsilon}{4}}$. This implies $\cM\left(\sqrt{\frac{ c_p n\varepsilon}{4}}, Q_{0,1}\right) \ge 2^{c_p\varepsilon^{1-p}/(32\log 2)}$ for $\varepsilon \gtrsim \frac{1}{n^{1/p}}$, so that \[\log \cM\left(\sqrt{\frac{ c_p n\varepsilon}{4}}, Q_{0,1}\right) \ge \frac{c_p\varepsilon^{1-p}}{32}.\] Perform a change of variables with $\tilde\varepsilon = \sqrt{\frac{ c_p n\varepsilon}{4}}$ which satisfies $\tilde\varepsilon \gtrsim \sqrt{n/n^{1/p}}$. Then \[\log \cM\left(\tilde\varepsilon, Q_{0,1}\right) \ge \frac{c_p}{32} \cdot \left(\frac{4\tilde\varepsilon^2}{c_p n}\right)^{1-p}\gtrsim \left(\frac{\tilde\varepsilon}{2\sqrt{n}}\right)^{-2(p-1)}.\] Since $\varepsilon\gtrsim 1/n^{1/p}$ if and only if $\tilde\varepsilon\gtrsim \sqrt{n/n^{1/p}}$, we have proven for all $\tilde\varepsilon\gtrsim \sqrt{n/n^{1/p}}$ that $\log \cM\left(\tilde\varepsilon, Q_{0,1}\right)\gtrsim \left(\frac{\tilde\varepsilon}{2\sqrt{n}}\right)^{-2(p-1)}$. 

            Now we consider arbitrary $a<b$, and again assume $p\ge 2$. Suppose $\frac{\varepsilon}{b-a} \gtrsim \sqrt{n/n^{1/p}}$.  Consider the linear function $\phi(x)=\frac{x}{b-a}-\frac{a}{b-a}$ which satisfies $\phi(a)=0$ and $\phi(b)=1$ and has inverse $\phi^{-1}(x) =(b-a)x+a$. Let $\Phi \colon\RR^n\to\RR^n$ apply $\phi$ in each coordinate, and $\Phi^{-1}\colon\RR^n\to\RR^n$ apply $\phi^{-1}$ instead. It is easy to verify that $\Phi(Q_{a,b})= Q_{0,1}$. Then by our prior result, there is some set $Q_{a,b}'\subseteq \Phi(Q_{a,b})$ with  $\log |Q_{a,b}'|\ge \left(\frac{\varepsilon}{2\sqrt{n}(b-a)}\right)^{-2(p-1)}$ such that for any $\mu,\mu'\in Q_{a,b}'$, we have $\|\mu-\mu'\|_2 \ge \frac{\varepsilon}{b-a}.$
   
   Now observe that $\Phi^{-1}(Q_{a,b}')$ is a subset of $Q_{a,b}$, with cardinality satisfying $\log |\Phi^{-1}(Q_{a,b}')|\ge \left(\frac{\varepsilon}{2\sqrt{n}(b-a)}\right)^{-2(p-1)}$. Moreover, by definition of $\Phi^{-1}$, for any $\tilde\mu,\tilde\mu'\in \Phi^{-1}(Q_{a,b}')$, we have $\|\tilde\mu-\tilde\mu'\|_2 \ge  \varepsilon.$ We have thus lower bounded the metric entropy of $Q_{a,b}$, i.e., provided $\frac{\varepsilon}{b-a} \gtrsim \sqrt{n/n^{1/p}}$, we have  \[\log \cM\left(\varepsilon, Q_{a,b}\right) \gtrsim \left(\frac{\varepsilon}{2\sqrt{n}(b-a)}\right)^{-2(p-1)}.\]
    \end{proof}

    \begin{proof}[\hypertarget{proof:lemma:multivariate:isotonic:local:metric}{Proof of Lemma \ref{lemma:multivariate:isotonic:local:metric}}]
        We prove the $a=0$ and $b=1$ case since the argument is identical otherwise. From our global entropy result in \eqref{eq:global:entropy:multivar:isotone}, there exist constants $0<c_1<c_2$ such that \begin{align*}
            c_1\cdot \left(\tfrac{\varepsilon}{2\sqrt{n}}\right)^{-2(p-1)} &\le \log \cM\left(\varepsilon, Q_{0,1}\right) \le \log \cM\left(\varepsilon/c^{\ast}, Q_{0,1}\right) \\ &\le c_2\cdot \left(\tfrac{\varepsilon}{2c^{\ast}\sqrt{n}}\right)^{-2(p-1)}. 
        \end{align*} Assume $c^{\ast} > ((1+c_1)/c_2)^{\tfrac{1}{2(p-1)}}$. Then, \begin{align*}
            \MoveEqLeft \log \cM\left(\varepsilon/c^{\ast}, Q_{0,1}\right)  - \log \cM\left(\varepsilon, Q_{0,1}\right) \\ &\ge c_2\cdot \left(\tfrac{\varepsilon}{2c^{\ast}\sqrt{n}}\right)^{-2(p-1)} -  c_1\cdot \left(\tfrac{\varepsilon}{2\sqrt{n}}\right)^{-2(p-1)} \\
            &= \left(\tfrac{\varepsilon}{2\sqrt{n}}\right)^{-2(p-1)} \left[ c_2 \cdot (c^{\ast})^{2(p-1)} - c_1\right] \\
            &\gtrsim \left(\tfrac{\varepsilon}{2\sqrt{n}}\right)^{-2(p-1)}.
        \end{align*} Hence using \eqref{yb:local:entropy:bound}, 
        \begin{align*}
            \left(\tfrac{\varepsilon}{2\sqrt{n}}\right)^{-2(p-1)} &\lesssim \log \cM_{Q_{0,1}}^{\loc}\left(\varepsilon\right) \le \log \cM\left(\varepsilon/c^{\ast}, Q_{0,1}\right) \\ &\lesssim \left(\tfrac{\varepsilon}{2\sqrt{n}}\right)^{-2(p-1)}.
        \end{align*}
    \end{proof}

    \begin{proof}[\hypertarget{proof:lemma:multivariate:isotone:LSE:optimal}{Proof of Lemma \ref{lemma:multivariate:isotone:LSE:optimal}}]
        Take $\varepsilon\asymp \sigma^{1/p}$, noting that this satisfies $\frac{1}{n^{1/2p}} \lesssim \varepsilon \lesssim 1$  by our assumption on $\sigma$. Then we can verify that $\varepsilon^2/\sigma^2 \asymp \sigma^{\tfrac{-2(p-1)}{p}}\asymp\log \cM_{Q_{-1/\sqrt{n},1/\sqrt{n}}}^{\loc}(\varepsilon)$ using \eqref{eq:multivariate_isotonic_metric_entropy}. Note that the minimum with $1$ is needed as if $\sigma^{1/p}$ is larger than the diameter of $K$ then the logarithm of local metric entropy becomes 0. Thus when $\varepsilon\asymp \sigma^{1/p}$, there exists absolute constants $C_1, C_2>0$ such that $\varepsilon^2/\sigma^2\ge C_1 \log \cM_{Q_{-1/\sqrt{n},1/\sqrt{n}}}^{\loc}(\varepsilon)$ and $\varepsilon^2/\sigma^2\le C_2 \log \cM_{Q_{-1/\sqrt{n},1/\sqrt{n}}}^{\loc}(\varepsilon)$. Let us now show that this implies $\varepsilon^{\ast}\asymp \sigma^{1/p}$.
    
    First, take $C>\max(1,\sqrt{C_2})$. Then by the non-increasing property of the metric entropy \citep[Lemma II.8]{neykov2022minimax}, we have \begin{align*}
        \varepsilon^2/\sigma^2 &\le C_2 \log \cM_{Q_{-1/\sqrt{n},1/\sqrt{n}}}^{\loc}(\varepsilon) \\ &\le C^2\log \cM_{Q_{-1/\sqrt{n},1/\sqrt{n}}}^{\loc}(\varepsilon/C),
    \end{align*} which rearranges to $(\varepsilon/C)^2/\sigma^2 \le \log \cM_{Q_{-1/\sqrt{n},1/\sqrt{n}}}^{\loc}(\varepsilon/C)$. By definition as a supremum, we have $\varepsilon^{\ast}\ge \varepsilon/C$. 
    
    Then take $0<C'<\min(1,\sqrt{C_1})$. Then by the non-increasing property of the metric entropy, we have \begin{align*}
        \varepsilon^2/\sigma^2 &\ge C_1 \log \cM_{Q_{-1/\sqrt{n},1/\sqrt{n}}}^{\loc}(\varepsilon) \\ &\ge (C')^2\log \cM_{Q_{-1/\sqrt{n},1/\sqrt{n}}}^{\loc}(\varepsilon/C'),
    \end{align*} which rearranges to $(\varepsilon/C')^2/\sigma^2 >\log \cM_{Q_{-1/\sqrt{n},1/\sqrt{n}}}^{\loc}(\varepsilon/C')$. By definition, we have $\varepsilon^{\ast}\le \varepsilon/C'$. Thus, $\varepsilon^{\ast}\asymp \varepsilon\asymp \sigma^{1/p}$, and the minimax rate is given by ${\varepsilon^{\ast}}^2\asymp \sigma^{2/p}$.
    
    We now upper bound $\epsLSE$. Observe  that $Q_{-1/\sqrt{n},1/\sqrt{n}} \subseteq B_2(2)$. Set $K= Q_{-1/\sqrt{n},1/\sqrt{n}} = Q_{-1/\sqrt{n},1/\sqrt{n}}\cap B_2(2)$, and then using Remark \ref{remark:upper:bound:wK} and \citet[Proposition 5]{isotonic_general_dimensions} (using $B_2(2)$ in place of $B_2(1)$ simply scales the result), for $\sigma = \frac{1}{\sqrt{n}}$ we have \begin{align*} \epsLSE &\lesssim \sqrt{\sigma  w(K)} = \sqrt{\tfrac{1}{\sqrt{n}}\cdot w(K)} \\ &\lesssim \sqrt{\tfrac{1}{\sqrt{n}}\cdot n^{1/2 - 1/p}\cdot\log^4(n)} \\ &=\sigma^{1/p}\log^2n.
    \end{align*} \end{proof}

\begin{lemma}\label{LSE:is:minimax:in:R1} Let $a>0$ and consider $K = [-a,a]$. Then the estimator defined in \citet{neykov2022minimax} is the projection $\Pi_K Y$, which is the constrained least squares estimator.
\end{lemma}

\begin{proof}

First, observe that when $Y\in K$ and $K$ is bounded, the estimator $\nu^{\ast}$ from \citet{neykov2022minimax} will output $Y$. This can be seen from the proof of Lemma V.2 of \citet{neykov2022minimax}, where it is shown $\nu^{\ast}$ will be within distance $6(C+1)\|Y-\Pi_K Y\|_2$ from the projection $\Pi_K Y$. But in this case, $\Pi_K Y=Y$, so $\nu^{\ast}=\Pi_K Y=Y$.

Suppose $Y \not \in [-a,a]$, and suppose without loss of generality that $Y > a$.  Let $\nu_{0}$ be an arbitrary point in $K$ that starts the algorithm of \citet{neykov2022minimax} and let $c$ be the constant appearing in the local metric entropy. Let $\nu_k$ denote the $k$th update. We claim for each $k$, $|\nu_k-a|\le d/2^k$ where $d=2a$ is the diameter of $K$ (ignoring some absolute constants). This will prove that $\nu_k\to a$. We proceed by induction. For $k=1$, observe that we form a $d/c$-packing of $B(\nu_0,d)\cap K=K$, and the first update $\nu_1$ is the point closest to $Y$, and therefore to $a$. Since we form a $(d/c)$-maximal packing of $K$, this also implies a $(d/c)$-covering of $K$ by the packing set points, so we must have $|\nu_1-a|\le d/c\le d/2^1$, verifying the base case. Now suppose $|\nu_k-a|\le d/2^k$ holds for some $k\in\NN$. To obtain $\nu_{k+1}$, we form a $d/(2^k c)$-packing of $B(\nu_k, d/2^k)\cap K$. Then $\nu_{k+1}$ will be the point closest to $Y$, and therefore $a$. Because $|\nu_k-a|\le d/2^k$, we know $a\in B(\nu_k, d/2^k)\cap K$. Thus, $a$ is within distance $d/(2^k c) \le d/2^{k+1}$ of one of the points in the maximal packing set of $B(\nu_k, d/2^k)\cap K$, in particular $\nu_{k+1}$. This completes the induction.
\end{proof}


    \begin{proof}[\hypertarget{proof:hyperrectangle:optimal:lemma}{Proof of Lemma \ref{hyperrectangle:optimal:lemma}}]
        First, we show that $\sum_{i=1}^n a_i^2 \wedge \sigma^2 \lesssim (k+1)\sigma^2 \wedge \sum_{i=1}^n a_i^2$. It is trivial that $\sum_{i=1}^n a_i^2\wedge \sigma^2 \le \sum_{i=1}^n a_i^2$, and using the definition of $k$, \begin{align*}
            \sum_{i = 1}^n a_i^2 \wedge \sigma^2 &= \sum_{i = 1}^{n-k-1} a_i^2 \wedge \sigma^2 + \sum_{i = n-k}^n a_i^2 \wedge \sigma^2  \\ &\le   \sum_{i = 1}^{n-k-1} a_i^2 +\sum_{i=n-k}^n \sigma^2 
 \\ &=\sum_{i = 1}^{n-k-1} a_i^2 + (k+1)\sigma^2 \\ &\le 2(k+2)\sigma^2.
        \end{align*} Thus $\sum_{i=1}^n a_i^2 \wedge \sigma^2 \lesssim (k+1)\sigma^2 \wedge \sum_{i=1}^n a_i^2$.
        
        Now we show $\sum_{i=1}^n a_i^2 \wedge \sigma^2 \ge (k+1)\sigma^2 \wedge \sum_{i=1}^n a_i^2$. We consider two cases. First, suppose $a_{n-k}^2 < \sigma^2$. Then by ordering of the $a_i$ and definition of $k$, we have \begin{align*}
            \sum_{i = 1}^n a_i^2 \wedge \sigma^2 &\ge \sum_{i = 1}^{n-k} a_i^2 \wedge \sigma^2  = \sum_{i = 1}^{n-k} a_i^2 \geq (k + 1)\sigma^2\\ &\ge (k+1)\sigma^2 \wedge \sum_{i=1}^n a_i^2.
        \end{align*}
 
  On the other hand, if $a_{n-k} \geq \sigma^2$ then \begin{align*}
      \sum_{i=1}^n a_i^2\wedge \sigma^2 &\ge \sum_{i = n-k}^n a_i^2 \wedge \sigma^2   = \sum_{i = 1}^{n-k} \sigma^2 =(k + 1)\sigma^2 \\ &\ge (k+1)\sigma^2 \wedge \sum_{i=1}^n a_i^2.
  \end{align*} Thus, we have the desired lower bound on $\sum_{i=1}^n a_i^2\wedge \sigma^2$ in both cases.
    \end{proof}

    \begin{proof}[\hypertarget{proof:lemma:local:metric:entropy:ell_1}{Proof of Lemma \ref{lemma:local:metric:entropy:ell_1}}]
        It is known \citep[Section III.D][e.g.,]{neykov2022minimax} that the global entropy satisfies
\begin{align*}
   \log \cM(\varepsilon, K) \asymp \begin{cases}
    \frac{\log (\varepsilon^2 n)}{\varepsilon^2} & \varepsilon \gtrsim 1/\sqrt{n} \\
    n & \varepsilon \asymp 1/\sqrt{n} \\
    n \log \frac{1}{n\varepsilon^2} & \varepsilon \lesssim 1/\sqrt{n}.
\end{cases} 
\end{align*} Recall also the earlier result \eqref{yb:local:entropy:bound} which is due to  \citet{yang1999information}. In the  $\varepsilon\gtrsim 1/\sqrt{n}$ case, we assume $\varepsilon \ge {c^{\ast}}^2/\sqrt{n}$, so that  $\log(\varepsilon^2 n)/2\log({c^{\ast}}^2) \ge 1$.
From the global entropy result stated above, since both $\varepsilon$ and $\varepsilon/c^{\ast}$ are $\gtrsim 1/\sqrt{n}$, there exist absolute constants $c_1,c_2>0$ such that \begin{align*} 
    \log \cM(\varepsilon, K) &\le \log \cM(\varepsilon/c^{\ast},K) \le c_1 \cdot \frac{\log (\varepsilon^2 n)}{\varepsilon^2} , \\  \log \cM(\varepsilon/c^{\ast}, K) &\ge  c_2 \cdot\frac{\log (\varepsilon^2 n/{c^{\ast}}^2)}{\varepsilon^2/{c^{\ast}}^2}.
\end{align*} Suppose $c^{\ast}$ is taken sufficiently large such that $c_2{c^{\ast}}^2/2> c_1$. Hence \begin{align*}
    \MoveEqLeft \log \cM(\varepsilon/c^{\ast}, K) - \log \cM(\varepsilon, K) \\ &\ge c_2\cdot \frac{\log(\varepsilon^2 n/{c^{\ast}}^2)}{\varepsilon^2/{c^{\ast}}^2} - c_1\cdot \frac{\log(\varepsilon^2n)}{\varepsilon^2} \\ &= (c_2{c^{\ast}}^2-c_1)\cdot\frac{\log (\varepsilon^2 n)}{\varepsilon^2}- c_2{c^{\ast}}^2 \cdot\frac{\log({c^{\ast}}^2)}{\varepsilon^2} \\
    &\ge (c_2{c^{\ast}}^2-c_1)\cdot\frac{\log (\varepsilon^2 n)}{\varepsilon^2}- \frac{c_2{c^{\ast}}^2\log({c^{\ast}}^2)}{\varepsilon^2} \cdot \frac{\log(\varepsilon^2 n)}{2\log({c^{\ast}}^2)} \\ &= (c_2{c^{\ast}}^2/2-c_1)\cdot \frac{\log (\varepsilon^2 n)}{\varepsilon^2}.
\end{align*}  

Thus, using \eqref{yb:local:entropy:bound}, \begin{align*}
    (c_2{c^{\ast}}^2/2-c_1)\cdot \frac{\log (\varepsilon^2 n)}{\varepsilon^2} &\le \log \cMloc(\varepsilon) \\ &\le  \log \cM(\varepsilon/c^{\ast},K) \\ &\lesssim \frac{\log (\varepsilon^2 n)}{\varepsilon^2}. 
\end{align*}Thus, $\log \cMloc(\varepsilon)\asymp \frac{\log (\varepsilon^2 n)}{\varepsilon^2}$ when $\varepsilon\gtrsim 1/\sqrt{n}$ for a sufficiently large constant and with $c^{\ast}$ chosen appropriately large.

Next, consider the case where $\varepsilon \le {c^{\ast}}^2/\sqrt{n}$, i.e., $\varepsilon\asymp 1/\sqrt{n}$ or $\varepsilon\lesssim 1/\sqrt{n}$. case. Since $\eta\mapsto \log \cMloc(\eta)$ is non-increasing, we know \begin{align*}
    n &\gtrsim\log \cMloc(\varepsilon) \ge \log \cMloc({c^{\ast}}^2/\sqrt{n}) \\ &\asymp  \frac{\log (({c^{\ast}}^2/\sqrt{n})^2\cdot n)}{({c^{\ast}}^2/\sqrt{n})^2}\asymp n.
\end{align*} The first inequality was proven in Case 2 of the proof of Corollary \ref{corollary:geometric_average:minimax}.
    \end{proof}

\subsection{Proofs for Section \ref{subsection:example:suboptimal:LSE}}

    \begin{proof}[\hypertarget{proof:lemma:pyramid:suboptimal}{Proof of Lemma \ref{lemma:pyramid:suboptimal}}]
        Consider the balls $B(0,\|v\|_2/2)$ and $B(v,\|v\|_2/2)$. We have $w(B(0,\|v\|_2/2)\cap P)\ge w(K)$ because $K\subseteq B(0,\|v\|_2/2)\cap P$. To see this, note clearly $K\subseteq P$ and for any $k\in K$, using symmetry we have \[\|k\|_2 =\tfrac{1}{2}\|k-(-k)\|_2\le \tfrac{1}{2}\diam(K)\le \tfrac{1}{2}\|v\|_2.\] Next, $B(v,\|v\|_2/2)\cap P$ is contained in the set $Q=\cup_{\beta \in [0,1]}[\beta v+(1-\beta) (v/2+K/2)]$. To see this, suppose $x\in B(v, \|v\|_2/2)\cap P$. Then $x = \alpha v + (1-\alpha)k$ for some $k\in K$ and $\alpha\in[0,1]$ and must satisfy $\|x-v\|_2\le \|v\|_2/2$. This implies \[0 \le (1-\alpha)\|v-k\|_2=\|x-v\|_2  \le \tfrac{1}{2}\|v\|_2.\] Since $\|v-k\|_2 = \sqrt{\|v\|_2^2+\|k\|_2^2}$ by  the Pythagorean theorem and orthogonality of $v$ to $K$, we can show $1-\alpha \le 1/2$, i.e., $\alpha\ge 1/2$, since \[(1-\alpha)^2 \le \frac{\|v\|_2^2}{4\|v-k\|^2}= \frac{1}{4}\cdot\frac{\|v\|_2^2}{\|v\|_2^2 + \|k\|_2^2}\le \frac{1}{4}.\]  Now take $\beta = 2(\alpha - 1/2)\in[0,1]$, and write \[x = \left(\tfrac{1}{2} +\tfrac{\beta}{2}\right)v + \left(\tfrac{1}{2} -\tfrac{\beta}{2}\right)k = \beta v + (1-\beta)\left(\tfrac{v}{2}+\tfrac{k}{2}\right)\in Q.\] Therefore $w(Q) \ge w(B(v,\|v\|_2/2)\cap P)$.

For any $a\in\RR$, define $a_{+}=\max(0,a).$ Then we compute  \begin{align*}
    w(Q) &= \EE \sup_{q\in Q} \langle \xi, q\rangle \\ &= \EE \sup_{\beta\in[0,1], k\in K}\langle \xi, \beta v+(1-\beta)(v/2+k/2)\rangle \\ &= \EE[ \sup_{\beta \in [0,1]}\beta \langle \xi, v/2\rangle + \langle \xi,v/2\rangle+ (1-\beta) \sup_{k \in K} \langle \xi, k/2 \rangle]\\ &= \EE[ \sup_{\beta \in [0,1]}\beta \langle \xi, v/2\rangle + (1-\beta) \sup_{k \in K} \langle \xi, k/2 \rangle] \\ &= \EE[\sup_{k \in K} \langle \xi, k/2 \rangle] + \EE[\sup_{\beta\in[0,1]} \beta\langle \xi, v/2\rangle - \beta \sup_{k \in K} \langle \xi, k/2 \rangle]  \\ &=  \tfrac{1}{2}w(K) +\EE \left[\langle \xi, v/2\rangle - \sup_{k \in K}\langle \xi, k/2 \rangle\right]_+ \\ &\leq \tfrac{1}{2}w(K) + \EE (\langle \xi, v/2\rangle)_+ \\
    &= \tfrac{1}{2}w(K) + \sqrt{2/\pi}\cdot \|v\|_2/4,
\end{align*} where in the last few lines, we evaluated the supremum at $\beta\in\{0,1\}$ and used that $\sup_{k \in K}\langle \xi, k/2 \rangle \geq 0$ since $0 \in K$. The final line used properties of the rectified Gaussian distribution, proved in, e.g., \citet[Appendix A]{beauchamp_recitifed_Gaussian}. Recall we assume $\|v\|_2^2 \lesssim w(K)$. We thus have \begin{align*}
    w_{P, v}(\|v\|_2/2) &= w(B(v,\|v\|_2/2)\cap P) \le w(Q) \\ &\lesssim \tfrac{1}{2}w(K) + \|v\|_2 \\
    &\lesssim \tfrac{1}{2}w(K) +\sqrt{w(K)}.
\end{align*} 

Hence \begin{align*}
    w_{P, 0}(\|v\|_2/2) -w_{P, v}(\|v\|_2/2) &\ge w(K)  -w_{P,v}(\|v\|_2/2) \\ &\gtrsim \frac{1}{2}w(K)-\sqrt{w(K)}.
\end{align*} So if $w(K)$ is sufficiently big, then $w_{P, 0}(\|v\|_2/2) -w_{P, v}(\|v\|_2/2)\gtrsim  w(K)$.

Hence we have for constants $\kappa,\tilde C$ that \begin{align*}
    w_{P,0}(\|v\|_2/2) - w_{P,v}(\|v\|_2/2) \ge \kappa w(K) \geq \tilde C \|v\|_2^2.
\end{align*}  Then recalling the definition of $\overline{\varepsilon}(\sigma)$ in Theorem \ref{difference:of:local:widths:compared:to:eps:squared:thm} (take $\varepsilon = c^{\ast}\|v\|_2/2$, $\nu_1=v$, $\nu_2=0$), we obtain $\overline{\varepsilon} \gtrsim \|v\|_2$. This implies for sufficiently large $c$ (so that $\overline{\varepsilon}\gtrsim 1=\sigma$) that  $\epsLSE[P] \gtrsim \overline{\varepsilon}\gtrsim \|v\|_2$.
    \end{proof}

\begin{lemma} \label{lemma:multivariate:isotonic:suboptimal:lower:bound} Set $K= Q_{-1/\sqrt{n},1/\sqrt{n}}\subset\RR^n$ where $n^{1/p}\in\NN$ for simplicity. Then for $0<t<1$, $w_{K,0}(t) \gtrsim t n^{1/2-1/p}$.
\end{lemma}
    \begin{proof}[\hypertarget{proof:lemma:multivariate:isotonic:suboptimal:lower:bound}{Proof of Lemma \ref{lemma:multivariate:isotonic:suboptimal:lower:bound}}] The proof will be very similar to that of \citet[Proposition 5]{isotonic_general_dimensions} with some rescaling and the use of $Q_{a,b}$ instead of their more general cone. By definition $w_{K,0}(t) =\EE_{\xi\sim \cN(0,\II_n)} \sup_{q\in Q_{a,b}\cap B(0,t)}\langle \xi, q\rangle$ where $a =-b= -\frac{1}{\sqrt{n}}$. As in \citet{isotonic_general_dimensions}, define $W= \{l\in \LL_{p,n}: \sum_{j=1}^p [l]_j =1\}$, $W^{+}= \{l\in \LL_{p,n}: \sum_{j=1}^p [l]_j >1\}$, and $W^-= \{l\in \LL_{p,n}: \sum_{j=1}^p [l]_j  <1\}$, where $[l]_j$ is the $j$th component of $l\in\LL_{p,n}\subset\RR^p$. For any draw of $\xi\sim \cN(0,\II_n)$, where coordinate $i$ corresponds to $l_i\in\LL_{p,n}$,  define $\Theta(\xi)\in\RR^n$ as \[[\Theta(\xi)]_i= \begin{cases}
            1/\sqrt{n} &\text{if } l_i \in W^+ \\
             \mathrm{sgn}(\xi_i)/\sqrt{n}&\text{if } l_i\in W \\
            - 1/\sqrt{n} &\text{if } l_i \in W^-.
        \end{cases}\] Observe that $\|\Theta(\xi)\|_2 =1$, and moreover $\Theta(\xi)\in Q_{a,b}$. Thus, since $0<t<1$, we have $t\Theta(\xi)\in Q_{a,b}\cap B(0,t)$. Then \begin{align*}
            \MoveEqLeft \EE_{\xi\sim \cN(0,\II_n)} \sup_{q\in Q_{a,b}\cap B(0,t)}\langle \xi, q\rangle \\ &\ge  \EE_{\xi\sim \cN(0,\II_n)} \langle \xi,t\Theta(\xi)\rangle \\
            &= \frac{t}{\sqrt{n}}\cdot\EE_{\xi\sim \cN(0,\II_n)} \left[ \sum_{l_i\in W^+} \xi_i -\sum_{l_i\in W^-} \xi_i +\sum_{l_i\in W} |\xi_i|\right] \\
            &= \frac{t\sqrt{2}}{\sqrt{\pi n}} |W|.
        \end{align*} As stated in \citet{isotonic_general_dimensions}, $|W|\ge \left(\frac{n^{1/p}-1}{p-1}\right)^{p-1}$, so we obtain $w_{K,0}(t) \gtrsim \frac{tn^{1/2-1/p}}{(p-1)^{p-1}}.$
    \end{proof}

\begin{lemma} \label{lemma:multivariate:isotonic:suboptimal:upper:bound} Set $K= Q_{-1/\sqrt{n},1/\sqrt{n}}\subset\RR^n$. Then for $0<c,t<1$, we have $w_{K,0}(ct) \lesssim c t n^{1/2 - 1/p} \log^4(n) $. 
\end{lemma}
    \begin{proof}[\hypertarget{proof:lemma:multivariate:isotonic:suboptimal:upper:bound}{Proof of Lemma \ref{lemma:multivariate:isotonic:suboptimal:upper:bound}}]
    In Proposition 5 of \citet{isotonic_general_dimensions}, the authors consider a cone which they denote $\cM(\LL_{p,n})$---to avoid clashing notations, we write it as $\cC(\LL_{p,n})$. Observe that \[\cC(\LL_{p,n}) \cap B(0,ct) = ct \cdot \cC(\LL_{p,n})\cap B(0,1).\] Thus, $w(\cC(\LL_{p,n})\cap B(0,ct))= ct\cdot w(\cC(\LL_{p,n})\cap B(0,1))$. By \citet[Proposition 5]{isotonic_general_dimensions}, we have \[w(\cC(\LL_{p,n}))\cap B(0,ct)) \lesssim c t n^{1/2-1/p}\log^4 n.\]
    But $K\cap B(0,ct)$ is a subset of $\cC(\LL_{p,n})\cap B(0,ct),$ so \begin{align*}
            w_{K,0}(ct)  \lesssim c t n^{1/2-1/p}\log^4 n.
        \end{align*} 
    \end{proof}

    \begin{proof}[\hypertarget{proof:lemma:solid:revolution}{Proof of Lemma \ref{lemma:solid:revolution}}]
        We define 4 sets and relate their Gaussian widths. Let $T_1 =B(0,b/4)\cap K$ and $T_2 = \bigcup_{x\in[0,b/4]}(\{x\}\times B_x)$. Let $B^{n}(s,r)$ mean the $n$-dimensional ball centered at $s$ and with radius $r$. Then let $T_3 = B^{n}(b/2 e_1, f(b/2))\cap K$ and $T_4 = B^{n}(b/2e_1, b/4)\cap K$ where $e_1=(1,0,\dots,0)\in\RR^n$.

First, observe that $T_1\subseteq T_2$. To see this, note that if $u\in T_1$, then $u\in K$ so that $u\in \{x\}\times B_{x}$ for some $x\in[0,b]$. This means the first coordinate of $u$ is $x$. Hence $x\le \|u\|_2 \le b/4$ using that $u\in B(0,b/4)$, so indeed $u\in T_2$. Hence $w(T_1)\le w(T_2)$.

Now let us compute $w(T_2)$. As $T_2$ is convex, for a fixed Gaussian vector $g$, the solution to $\arg\sup_{x\in T_2}\langle x,g\rangle$ will lie on an extreme point of $T_2$, i.e., $x = (c\cdot b/4)\cdot e_1  + f(c\cdot b/4)\cdot  u$ for some $c\in[0,1]$ and unit vector $u\in e_1^{\perp}$. Thus, letting $y_+:=\max(0,y)$, we have \begin{align*}
    w(T_2) &= \EE_g \sup_{x\in T_2}\langle x,g\rangle \\ &=\EE_g \sup_{\substack{c\in[0,1] \\ u\in e_1^{\perp}, \|u\|_2=1}}\langle (c\cdot b/4)\cdot e_1  + f(c\cdot b/4)\cdot u,g\rangle  \\
    &\le \EE_g \sup_{u\in e_1^{\perp}, \|u\|_2=1} [(b/4)\langle e_1, g\rangle_{+}+ f(b/4)\langle u,g\rangle_+ ]\\
    &\le  (b/4)\cdot \EE_g \langle e_1, g\rangle_{+}+ f(b/4)\cdot  \EE_g \sup_{u\in e_1^{\perp}, \|u\|_2=1} \langle u,g\rangle_+   \\
    &\le \frac{b}{4\sqrt{2 \pi}} + f(b/4)\sqrt{n-1} \\
    &\le  f(b/2)\sqrt{n-1}.
\end{align*} 
In the fourth line, we note that $\sup_{u\in e_1^{\perp}, \|u\|_2=1}\langle u, g\rangle_{+}$ is achieved with $u=(0,g_{-1}/\|g_{-1}\|)$, where $g_{-1}$ is the random vector in $\RR^{n-1}$ obtained from removing the first coordinate of $g$. Hence $\EE_g\sup_{u\in e_1^{\perp}, \|u\|_2=1}\langle u, g\rangle_{+} = \EE_g\|g_{-1}\|_2\le \sqrt{n-1}$. In the same line, we also used that $\EE_g\langle e_1, g\rangle_{+} = \EE\max(0,g_1)=\tfrac{1}{\sqrt{2\pi}}$ \citep[see][Appendix A]{beauchamp_recitifed_Gaussian}. The final inequality was from our assumption \eqref{eq:solid:revolution:secant} about the secant line. 

Proceeding to $T_3$, note this set has width at least $\sqrt{n-1}\cdot f(b/2)$ up to constants by \citet[Example 7.5.7]{vershynin2018high} since it contains a $(n-1)$-dimensional ball of radius $f(b/2)$. Note that $T_3\subset T_4$ since $b/4>f(b/2)$. Putting our results together, we obtain \begin{align*}
    w_{K,0}(b/4) &= w(T_1)\le w(T_2) \\ &\le \frac{b}{4\sqrt{2 \pi}} + f(b/4)\sqrt{n-1} \\ &\le  f(b/2)\sqrt{n-1}\\ &\lesssim w(T_3) \le w(T_4) = w_{K,b/2\cdot e_1}(b/4).
\end{align*} This means \begin{align*}
    \MoveEqLeft w_{K,b/2\cdot e_1}(b/4) - w_{K,0}(b/4) \\ &\ge  f(b/2)\sqrt{n-1} - \left(\frac{b}{4\sqrt{2 \pi}} + f(b/4)\sqrt{n-1} \right).
\end{align*}

For sufficiently large $n$ and any constant $\overline{C}>0$, we have \[f(b/2) \cdot\sqrt{n-1}- \frac{b}{4\sqrt{2 \pi}} - f(b/4)\cdot\sqrt{n-1} \ge \overline{C}b^2,\] e.g., take $n \ge 1+ \left(\frac{\overline{C}b^2+b/(4\sqrt{2\pi})}{f(b/2)-f(b/4)}\right)^2$. This implies $w_{K,b/2\cdot e_1}(b/4)-w_{K,0}(b/4) \ge\overline{C}b^2$. Take $\sigma=1$. Now consider the inequality  defining $\overline{\varepsilon}(\sigma)$ in Theorem \ref{Lipschitz:map:theorem}. Taking $\nu_1 = (b/2)\cdot e_1$, $\nu_2=0$, $\varepsilon=b c^{\ast}/4$,  and choosing $\overline{C}$ sufficiently large, we can show that $\varepsilon$ satisfies the condition in the definition of $\overline{\varepsilon}(\sigma)$, hence $\overline{\varepsilon}(\sigma)\ge \varepsilon\asymp b$. Noting that $\overline{\varepsilon}(\sigma)\gtrsim b\gtrsim \sigma$, we conclude $\epsLSE \gtrsim \overline{\varepsilon}(\sigma) \gtrsim b$ by the theorem.
    \end{proof}

    \begin{proof}[\hypertarget{proof:lemma:ellipse:prelude}{Proof of Lemma \ref{lemma:ellipse:prelude}}]
    We first show that if $\sigma\gtrsim d^2/w(K)$, then for all $y\in K$, we have $w_y(\epswidth[y])\gtrsim w(K)$. By Jung's theorem, setting $\kappa_n = \sqrt{\frac{n}{2(n+1)}}\asymp 1$, for some $\mu\in K$ we have $K\subset B(\mu, d\kappa_n)$. Using our Lipschitz result\footnote{This can be rearranged to the form $w_{\nu}(\varepsilon)\ge \frac{w_{\mu}(\varepsilon)}{1+\|\mu-\nu\|/\varepsilon}$.} in \eqref{remark:Lipschitz:result} and that $\|y-\mu\|\le d\kappa_n$, we have \begin{align*}
        w_{y}(d\kappa_n) &\ge \frac{w_{\mu}(d\kappa_n)}{1+\tfrac{\|y-\mu\|}{d\kappa_n}} \ge  \frac{w_{\mu}(d\kappa_n)}{2} = \frac{w(K)}{2}.
    \end{align*} Then by definition of $\epswidth[y]$, \begin{align*}
        w_y(\epswidth[y])-\frac{\epswidth[y]^2}{2\sigma} &\ge w_{y}(d\kappa_n) - \frac{(d\kappa_n)^2}{2\sigma} \\ &\ge \frac{w(K)}{2} - \frac{(d\kappa_n)^2}{2\sigma}.
    \end{align*} Hence, for $\sigma\gtrsim d^2/w(K)$ we have $w_y(\epswidth[y])\gtrsim w(K)$. 

 Next, recall we chose $x\in\operatorname{bd}K$ that minimizes $\|\nabla G(x)\|_2$ and set $x^{\ast} =  \frac{\nabla G(x)}{\|\nabla G(x)\|_2}$. We now upper bound $w_x(\epswidth[x])$, i.e., \begin{align*}
     \EE\sup_{y\in K\cap B(x,\epswidth[x])}\langle \xi, y\rangle &= \EE\sup_{y\in K\cap B(x,\epswidth[x])}\langle \xi, y-x\rangle \\ &= \EE\sup_{z\in K'} \langle \xi, z\rangle 
 \end{align*} where $K'=\{z: z+x\in K, \|z\|\le \epswidth[x]\}$. Now write $\xi=\sum_i \alpha_i \mu_i$ and $z=\sum_i \alpha_i'\mu_i$ where $\mu_i$ are the ortho-normalized eigenvectors of $\tilde{M}/2$. Using \eqref{strong:convexity:separation}, observe that that $K'$ is a subset of $K''=\{z:\|z\|\le \epswidth[x], -\langle x^{\ast},z\rangle\ge z^T\tilde{M}z /2\}$. Thus, indexing $K''$ by $\alpha'$ in the supremum and using the Cauchy-Schwarz inequality, we have \begin{align*}
    w_x(\epswidth[x]) &=  \EE\sup_{z\in K'}\langle \xi, z\rangle \le \EE\sup_{z\in K''} \langle \xi,z\rangle = \EE_{\xi} \sup_{\alpha'} \sum_i \alpha_i \alpha_i' \\ &\leq \EE_{\xi}\sup_{\alpha'} \sqrt{\sum \alpha_i'^2 \lambda_i} \sqrt{\sum \alpha_i^2/\lambda_i} \\ &\leq \sup_{\alpha'} \sqrt{\sum \alpha_i'^2 \lambda_i} \sqrt{\EE \sum \alpha_i^2/\lambda_i} \\
    &= \sup_{\alpha'} \sqrt{\sum \alpha_i'^2 \lambda_i} \sqrt{ \sum 1/\lambda_i}.
\end{align*} The final two lines used Jensen's inequality and then that $\EE_{\xi}\alpha_i^2=1$ as the $\alpha_i$ are standard Gaussian random variables.

On the other hand $-\langle x^*, z\rangle \geq z\T \tilde M z /2$ is equivalent to $\sum \alpha_i'^2 \lambda_i \leq -\sum \alpha_i' \langle \mu_i, x^*\rangle$, while $\|z\|\leq \epswidth[x]$ is equivalent to $\sum \alpha_i'^2 \leq \epswidth[x]^2$. Thus by another application of Cauchy-Schwarz we get
\begin{align*}
    w_x(\epswidth[x]) &\leq  \sup_{\alpha'} \sqrt{-\sum \alpha_i' \langle \mu_i, x^*\rangle} \sqrt{ \sum 1/\lambda_i}\\
    & \leq \sqrt{\sqrt{\sum \alpha_i'^2} \sqrt{\sum \langle \mu_i, x^*\rangle^2}} \sqrt{ \sum 1/\lambda_i}\\
    & \leq \sqrt{\epswidth[x]}\sqrt{ \sum 1/\lambda_i}.
\end{align*} 

This shows $w(K)\lesssim w_x(\epswidth[x]) \le\sqrt{\epswidth[x]}\sqrt{ \sum 1/\lambda_i}$ provided $\sigma \gtrsim d^2/(w(K))$, and therefore that $\epswidth[x]\gtrsim w^2(K)/\sum (1/\lambda_i)$. So taking the supremum over $x\in K$, we have $\epsSupwidth(\sigma) \gtrsim w^2(K)/\sum (1/\lambda_i)$ for $\sigma \gtrsim d^2/(w(K))$.
    \end{proof}

    \begin{proof}[\hypertarget{proof:ellipsoid:lemma}{Proof of Lemma \ref{ellipsoid:lemma}}]

   In Remark \ref{remark:Lipschitz}, we found for each $\mu,\nu\in K$ that $|w_{\mu}(\varepsilon)-w_{\nu}(\varepsilon)| \le \|\mu-\nu\|\cdot \frac{w_{\mu}(\varepsilon)\wedge w_{\nu}(\varepsilon)}{\varepsilon}$. Specializing to $\mu =0$, we obtain $w_\nu(\varepsilon) \geq w_0(\varepsilon)/(\|\nu\|/\varepsilon + 1)$. Now pick $\nu = (0, \ldots, 1/d_{n-k}, 0,\ldots,0)$, where all but coordinate $n-k$ is zero. Set $\delta = 1/d_{n-k}$ so that $\|\nu\|/\delta=1$, and set $\varepsilon=\epswidth[\nu]$ as defined in \eqref{equation:varepsilon:mu}. Then we have $w_\nu(\delta) \geq w_0(\delta)/2.$
Hence by taking $ \sigma \gtrsim \frac{\delta^2}{w_0(\delta)}$ we obtain 
\begin{align*}
    w_{\nu}(\epswidth[\nu]) - \frac{\epswidth[\nu]^2}{2\sigma} &\ge w_\nu(\delta) - \frac{\delta^2}{2\sigma} \geq  \frac{w_0(\delta)}{2} -  \frac{\delta^2}{2\sigma} \gtrsim w_0(\delta).
\end{align*} The first inequality used the definition of $\epswidth[\nu]$. For such $\sigma$, we have \begin{align} \label{eq:ellipse:wnu:bigger:w0}
    w_\nu(\epswidth[\nu]) \gtrsim w_0(\delta).
\end{align}
    
    Now we will attempt to control $w_\nu(\epswidth[\nu])$ from above following a similar logic to before. We need to calculate $\EE \sup_{y\in K'} \langle \xi, y \rangle = \EE \sup_{y\in K'} \langle \xi, y - \nu\rangle$ for $K'=\{y\in K:\|\nu - y\| \leq \epswidth[\nu]\}$. Note that $\nu$ is a boundary point of $K$ since $\|D\nu\|_2^2 = 1$.  Set $x^* = \frac{\nabla G(\nu)}{\|\nabla G(\nu)\|}$ with $G(\nu) = \|D \nu\|_2^2$, $M=2D^2$, $\tilde M = M/\|\nabla G(\nu)\|$. Also, $\nabla G(\nu) = 2 D^2 \nu = (0,0,\ldots, 2 d_{n-k},0,\ldots, 0)$, $x^* = (0,0,\ldots,1, 0, 0, \ldots, 0)$, and $\tilde M/2$ is the diagonal matrix with element $d_i^2/(2 d_{n-k})$ in entry $(i,i)$. Using \eqref{taylor:expansion:convex:G}, $K'$ is a subset of \[B(\nu,\epswidth[\nu]) \cap \Big\{y\in K:\langle x^*, \nu - y \rangle \geq \frac{(\nu-y)\T \tilde M (\nu-y)}{2} \Big\}.\] 

Let $y-\nu = z$. We need to optimize
$\langle \xi, z \rangle$ given that $\|z\| \leq \epswidth[\nu]$, $-\langle x^*, z\rangle \geq z\T \tilde M z /2$. We have by the Cauchy-Schwarz inequality that
\begin{align*}
    \EE \sup_z \langle \xi, z \rangle &=\EE \sup_z \left[ \sum_{i=n-k}^n \xi_i z_i+\sum_{i = 1}^{n-k-1} \xi_i z_i \right] \\ &\leq  \epswidth[\nu]\sqrt{k + 1} + \EE \sqrt{\sum_{i = 1}^{n-k-1} \!\frac{\xi_i^2}{\lambda_i}}\sqrt{\sum_{i = 1}^{n-k-1}\!\!\! z_i^2 \lambda_i},
\end{align*}
where the $\lambda_i = d_i^2/(2 d_{n-k})$ are the eigenvalues of $\tilde M/2$. Bringing the expectation inside the square root by Jensen's inequality and noting that $\EE \xi_i^2 = 1$ we get
\begin{align*}
    \EE \sup_z \langle \xi, z \rangle \leq \epswidth[\nu] \sqrt{k + 1} + \sqrt{\sum_{i = 1}^{n-k-1}\!\frac{ 2 d_{n-k}}{d_i^2}}\sqrt{\sum_{i = 1}^{n-k-1} \!\! z_i^2 \lambda_i}.
\end{align*}
Observe that \begin{align*}
    \sum_{i = 1}^{n-k-1} z_i^2 \lambda_i &\leq \sum_{i = 1}^{n} z_i^2 \lambda_i =z^T(\tilde M/2)z \\ &\leq -2 \langle x^*, z \rangle \leq 2\epswidth[\nu]. 
\end{align*} Recalling \eqref{eq:ellipse:wnu:bigger:w0} (which requires $\sigma\gtrsim \delta^2/w_0(\delta)$) and combining the previous two lines,
\begin{align}
    w_{0}(\delta)&\lesssim w_{\nu}(\epswidth[\nu])\le \EE \sup_z \langle \xi, z \rangle \notag \\ &\leq \epswidth[\nu] \sqrt{k + 1}  + \sqrt{\sum_{i = 1}^{n-k-1}\frac{ 2 d_{n-k}}{d_i^2}} \cdot\sqrt{2\epswidth[\nu]}.  \label{ellipse:gaussian_width_upper_bound:0}
\end{align}
Now we note several facts. The set $B(0,\delta)\cap K$ is all $x$ such that $\|D x\|_2 \leq 1$, and $\|x\| \leq \delta$. Let $D'$ be the diagonal matrix with $(i,i)$ entry $d_i'$, where  where $d_i' = d_i$ for $i \leq n-k-1$ and $d_i' = 1/\delta$ for $i \ge n-k$. If $\|D'x\|\leq 1$  then $\|D x\|_2 \leq 1$.  Moreover, $\|D'x\|\leq 1$ implies $\|x\|\le \delta$.  Thus using \citet[Exercise 5.9]{wainwright2019high} to compute the Gaussian width, we obtain \begin{align}\label{w0:bigger}
    w_0(\delta)  &= w(B(0,\delta)\cap K) \ge w(\{x:\|D'x\|\le 1\}) \notag \\ &\gtrsim \sqrt{\sum_{i = 1}^n \frac{1}{d_i'^2}}  = \sqrt{\sum_{i=1}^{n-k-1}\frac{1}{d_i^2} +\sum_{i=n-k}^n \delta^2  }\notag \\ &\gtrsim \delta\sqrt{k + 1} + \sqrt{\sum_{i = 1}^{n-k-1} \frac{1}{d_i^2}},
\end{align} where in the last step we used that $\sqrt{a + b} \geq (\sqrt{a} + \sqrt{b})/\sqrt{2}$ for $a,b\ge 0$.


Combining this result with our upper bound on $w_0(\delta)$ from \eqref{ellipse:gaussian_width_upper_bound:0}, we conclude   
\begin{align*}
    \MoveEqLeft \epswidth[\nu] \sqrt{k + 1}  + \sqrt{\sum_{i = 1}^{n-k-1}\frac{ 4\epswidth[\nu] d_{n-k}}{d_i^2}}  \\ &\gtrsim  \delta\sqrt{k + 1} + \sqrt{\sum_{i = 1}^{n-k-1}\frac{1}{d_i^2}}.
\end{align*}
From the above we have that either $ \epswidth[\nu]\sqrt{k + 1} \gtrsim \delta\sqrt{k + 1}$ or $\sqrt{\sum_{i = 1}^{n-k-1}\frac{ 4\epswidth[\nu] d_{n-k}}{d_i^2}} \gtrsim \sqrt{\sum_{i = 1}^{n-k-1} \frac{1}{d_i^2}}$; in the first case, we conclude that $\epswidth[\nu] \gtrsim \delta$; in the second case we obtain $\epswidth[\nu] \gtrsim 1/d_{n-k} = \delta$. Hence $\epsSupwidth(\sigma)\ge \epswidth[\nu](\sigma)\gtrsim \delta_{n-k}$ for $\sigma\gtrsim \delta_{n-k}^2/w_0(\delta_{n-k})$, as claimed.

Next, suppose the LSE is minimal optimal for $K$ for all $\sigma$. Take $\sigma\asymp \delta^2/w_0(\delta)$ using the same $\delta=\delta_{n-k}$. First, suppose $\delta \gtrsim \sigma$ for a sufficiently large constant, which means $\epswidth[\nu]\gtrsim \delta\gtrsim\sigma$ using our previous result, where $\nu$ is as chosen above. By Lemma \ref{lemma:chatterjee:analogue}, we have $\EE\|\hat\mu-\nu\|_2^2\gtrsim \epswidth[\nu]^2$.

Using \eqref{varepsilon:star:def}, observe that the minimax rate $\varepsilon^{\ast}$ for this value of $\sigma$  satisfies
\begin{align*}
 \varepsilon^{*2}\delta^{-4} w^2_0(\delta)\asymp \varepsilon^{*2}/\sigma^2  \lesssim \log \cMloc(\varepsilon^*).
\end{align*} 
Furthermore, if the LSE is minimax optimal, we have ${\varepsilon^*}^2\gtrsim \EE\|\hat\mu-\nu\|_2^2 \gtrsim \epswidth[\nu]^2 \gtrsim \delta ^2$. Hence that $\log \cMloc(\varepsilon^*)\le \log \cMloc(c\delta)$ for a sufficiently small absolute constant $c$ since $\varepsilon\mapsto \log \cMloc(\varepsilon)$ is non-increasing. And also $\varepsilon^{*2}\delta^{-4} w^2_0(\delta) \gtrsim \delta^{-2}w^2_0(\delta)$. Therefore, we obtain
\begin{align*}
    \delta^{-1}w_0(\delta) \lesssim \sqrt{\log \cMloc(c\delta)}.
\end{align*}

We now assume $\delta \lesssim \sigma$, i.e., $\delta\lesssim \delta^2/w_0(\delta)$ for a sufficiently large absolute constant. Rearranging, $w_0(\delta)/\delta \lesssim 1 \leq \sqrt{\log \cMloc(c\delta)}$ since $\delta \leq d$ for some $c < 1$. 
\end{proof}

    \begin{proof}[\hypertarget{proof:lemma:strongly:convex:body}{Proof of Lemma \ref{lemma:strongly:convex:body}}]
        Note that if $\sigma \gtrsim d$, then by Lemma \ref{minimax:bound:versus:sigma:and:d} the minimax rate satisfies $\varepsilon^{\ast}\gtrsim d$ which implies the LSE satisfies $\epsLSE\gtrsim d$. Now consider the case where $\sigma \lesssim d$.

        We fix $c > 4$, $\kappa\in (1/c,1-2/c)$, and pick any $\varepsilon$ such that $d\lesssim\varepsilon < \tfrac{d}{2}\wedge \tfrac{d}{2(\kappa+1/c)}$. Take a point $\mu \in K$ which is a midpoint of a diameter. Pick $\nu', \nu'' \in K$ such that $\|\mu - \nu'\| = \|\mu - \nu''\| = \kappa \varepsilon < d/2$ and $\mu = \frac{\nu'+\nu''}{2}$. Note that $\|\nu'-\nu''\|=2\kappa\varepsilon$. Define for some fixed $\xi \sim \cN(0, \II_n)$ the points \begin{align*}
            u' &= \argmax_{x \in B(\nu', \varepsilon/c) \cap K} \langle \xi, x \rangle, \quad u'' = \argmax_{x \in B(\nu'', \varepsilon/c) \cap K} \langle \xi, x \rangle
        \end{align*} in $K$. Apply strong convexity to $u',u''$ to conclude $B((u' + u'')/2, k \|u' - u''\|_2^2/4) \subset K$. Then we repeatedly use $u'\in B(\nu',\varepsilon/c)$ and $u''\in B(\nu'',\varepsilon/c)$ along with the triangle inequality to obtain
\begin{align}
    2  \kappa\varepsilon + 2\varepsilon/c &\geq \overbrace{\|\nu' - \nu''\|}^{= 2\kappa\varepsilon} + \overbrace{\|u' - \nu'\|}^{\le \varepsilon/c} + \overbrace{\|u'' - \nu''\|}^{\le \varepsilon/c}  \label{eq:strong:convexity:upper}\\
    &\geq \|\nu' - \nu'' + u' - \nu' + \nu''-u''\| \notag \\
    & = \|u' - u''\| \notag \\
    & \geq \|\nu' - \nu''\| - \|\nu' -u'\| - \|u''-\nu'' \| \notag  \\ &\ge 2 \varepsilon \kappa - 2\varepsilon/c. \label{eq:strong:convexity:lower}
\end{align}
Consider the point $u = (u' + u'')/2 + (\xi/\|\xi\|)\cdot k \|u'-u''\|^2_2/4$. Then \begin{align*}
    \|u- \mu\| &\leq \tfrac{1}{2}\|u' - \nu'\| + \tfrac{1}{2} \|u'' - \nu''\| +\tfrac{1}{4} k \|u'-u''\|^2_2 \\ &\leq \tfrac{1}{2}\cdot \tfrac{\varepsilon}{c} + \tfrac{1}{2}\cdot \tfrac{\varepsilon}{c} +k(\varepsilon \kappa + \tfrac{\varepsilon}{c})^2 \\&= \tfrac{\varepsilon}{c}  + k(\varepsilon \kappa + \tfrac{\varepsilon}{c})^2 \\ &\leq \tfrac{\varepsilon}{c}  + 2d^{-1}(\varepsilon \kappa + \tfrac{\varepsilon}{c})^2 \\
    &= \tfrac{\varepsilon}{c}+ (\varepsilon\kappa+\tfrac{\varepsilon}{c})\cdot \tfrac{2\varepsilon(\kappa+1/c)}{d} \\
    &< \tfrac{\varepsilon}{c}+ \varepsilon\kappa+\tfrac{\varepsilon}{c} \\
    &= \varepsilon(\kappa+\tfrac{2}{c})<1.
\end{align*} The first line used the triangle inequality and that $\mu = \tfrac{\nu'+\nu''}{2}$. The second line used $u'\in B(\nu',\varepsilon/c)$ and $u''\in B(\nu'',\varepsilon/c)$ along with our result from \eqref{eq:strong:convexity:upper}. The fourth line used $k<2d^{-1}$ as we established in the main text. The final two lines used $\varepsilon < \frac{d}{2(\kappa+1/c)}$ followed by $\kappa + 2/c < 1$. Hence $u\in B(\mu,\varepsilon)\cap K$. 

Then using that $u\in B(\mu,\varepsilon)\cap K$ along with \eqref{eq:strong:convexity:lower}, we have
\begin{align*} 
    \sup_{x \in B(\mu, \varepsilon)\cap K} \langle \xi,x \rangle &\geq \langle \xi, u \rangle \\ 
    &= \langle\xi, \tfrac{u' + u''}{2}\rangle + \tfrac{1}{4} k\|\xi\| \|u'-u''\|^2_2 \\ 
    &\geq \tfrac{1}{2}\langle\xi, u'\rangle + \tfrac{1}{2}\langle\xi, u''\rangle+ k \|\xi\| (\varepsilon \kappa - \varepsilon/c)^2.
\end{align*}
Taking expectation with respect to $\xi$ we have
\begin{align*}
    w_{\mu}(\varepsilon) &\geq \tfrac{1}{2} \cdot w_{\nu'}(\varepsilon/c) + \tfrac{1}{2} \cdot w_{\nu''}(\varepsilon/c) + \EE \|\xi\| k (\varepsilon \kappa - \varepsilon/c)^2 \\
    &\gtrsim \inf_{x\in B(\mu,\varepsilon)\cap K}w_{x}(\varepsilon/c) + \sqrt{n} k \varepsilon^2(\kappa - 1/c)^2.
\end{align*} The first line again used  $u'\in B(\nu',\varepsilon/c)$ and $u''\in B(\nu'',\varepsilon/c)$, and the second that $\nu',\nu''\in B(\mu,\kappa\varepsilon)\subset B(\mu,\varepsilon)$ along with   $\EE \|\xi \| \gtrsim \sqrt{n}$ (see \citet[Section 3.1]{chandrasekaran2012convex}).

Now assume $\sigma\asymp (k\sqrt{n})^{-1}$  which implies $\sigma\gtrsim d/\sqrt{n}$ since $k<4d^{-1}$. This implies \[w_{\mu}(\varepsilon)\gtrsim  \inf_{x\in B(\mu,\varepsilon)\cap K}w_{x}(\varepsilon/c)   +\overline{C}\varepsilon^2/2\sigma\] for some constant $\overline{C}>0$. Then by definition of $\overline{\varepsilon}(\sigma)$ in Theorem \ref{big:width:minus:small:width:thm} we conclude $\overline{\varepsilon}(\sigma) \gtrsim \varepsilon \asymp d$. As we have $d\gtrsim\sigma$, we have $\overline{\varepsilon}(\sigma)  \gtrsim \sigma$ which by Theorem \ref{big:width:minus:small:width:thm} implies $\epsLSE(\sigma)\gtrsim \overline{\varepsilon}(\sigma)   \gtrsim d$. Then note that $\epsLSE(\sigma)\le d$ always holds to deduce our claim.
    \end{proof}

\section{Proofs for Appendix \ref{section:algorithm_proofs}}

\begin{lemma} \label{lemma:local:packing:algorithm} Let $\epsSupwidth(\sigma)=\sup_{\mu\in K}\epswidth(\sigma)$ where $\epswidth$ is as defined in \eqref{equation:varepsilon:mu}. Set $c'=\frac{(2c^{\ast}-4)(4c^{\ast}-1/c^{\ast})}{(c^{\ast}-4)c^{\ast}}$ for some $c^{\ast}>4$. \renewcommand\labelenumi{(\theenumi)}
\renewcommand{\theenumi}{\roman{enumi}}
\begin{enumerate}
    \item If Algorithm \ref{algo_local_packing} terminates after $k$ iterations of the \textbf{while} loop, it returns $\varepsilon = \frac{d}{2^{k-1}c^{\ast}}$ and satisfies $\epsSupwidth(\sigma)\ge \varepsilon$.
    \item If Algorithm \ref{algo_local_packing} does not terminate within $k$ iterations, then we have $\epsSupwidth(\sigma) \leq c'\cdot \frac{d}{2^k}$.
\end{enumerate}  
\end{lemma}   
    \begin{proof}[\hypertarget{proof:lemma:local:packing:algorithm}{Proof of Lemma \ref{lemma:local:packing:algorithm}}]
        Suppose that Algorithm \ref{algo_local_packing} has terminated on the $k$th level of the infinite tree, i.e., has returned $d/(2^{k-1}c^{\ast})$. Then for some $\mu,\nu\in K$ such that $\|\mu-\nu\| \leq d/2^{k-1}$, we have
\begin{align*}
   w_\mu\left(\frac{d}{2^{k-1}}\right)   - w_\nu\left(\frac{d}{2^{k-1}c^{\ast}}\right) \geq \frac{C (\tfrac{d}{2^{k-1}})^2}{2\sigma}.
\end{align*} Since $B(\mu, d/2^{k-1}) \subseteq B(\nu, d/2^{k-2})$, \begin{align*}
    w_\nu\left( \frac{d}{2^{k-2}}\right)  &\geq w_\mu\left(\frac{d}{2^{k-1}}\right) \\ &\geq w_\nu\left(\frac{d}{2^{k-1}c^{\ast}}\right)  +  \frac{C (\tfrac{d}{2^{k-1}})^2}{2\sigma}.
\end{align*} Then since $C = 4 - \frac{1}{{c^{\ast}}^2}$, rearranging yields \[w_\nu\left( \frac{d}{2^{k-2}}\right)  - \frac{(\tfrac{d}{2^{k-2}})^2}{2\sigma} \ge w_\nu\left(\frac{d}{2^{k-1}c^{\ast}}\right)  - \frac{(\tfrac{d}{2^{k-1}c^{\ast}})^2}{2\sigma}.\] By our usual concavity argument with $\varepsilon\mapsto w_{\nu}(\varepsilon)-\varepsilon^2/2\sigma$ and that $d/2^{k-2}\ge d/(2^{k-1}c^{\ast})$, we have $\epsSupwidth(\sigma) \geq \epswidth[\nu](\sigma) \ge d/(c^{\ast}2^{k-1})$. This proves (i).

We will now prove (ii). Suppose the algorithm has run for at least $k+1$ steps. This means for any level $j \leq k $ of the tree we haven't found any points that satisfy the given inequality. Take any $\mu \in K$. Recall that at the $(k-1)$th step of the algorithm, we formed $d/(c^{\ast}2^{k-2})$-packing sets of $B(\cdot, d/2^{k-2})\cap K$, whose elements populate the $k$th level of the tree. Note that at any level $j$ of the tree, the union of the balls of radius $d/2^{j-1}$ cover $K$. To prove this, repeat the induction argument in the proof of Theorem \ref{important:thm}. Hence $\mu$ belongs to one of the $B(\cdot, d/2^{k-2})\cap K$ and thus there is some  $\nu\in K$ from one of the packing sets satisfying $\|\mu-\nu\|_2\le d/(c^{\ast}2^{k-2}) \le d/2^{k-1}$. The packing process at level $k$ then forms a $d/(c^{\ast}2^{k-1})$-packing of $B(\nu,d/2^{k-1})\cap K$, so we may pick some $\nu'\in K$ such that $\|\nu' - \mu\|_2\le d/(c^{\ast}2^{k-1})$.

Then $\Phi \le T$ for this particular packing of $B(\nu,d/2^{k-1})\cap K$, i.e.,
\begin{align} \label{eq1:local:packing:algorithm}
    w_\nu\left(\frac{d}{2^{k-1}}\right)  - w_{\nu'}\left( \frac{d}{c^{\ast}2^{k-1}}\right) \leq  \frac{C(d/2^{k-1})^2}{2\sigma}.
\end{align}
But since $B(\mu,d/2^{k-1}-  d/(c^{\ast}2^{k-2}))\subseteq B(\nu, d/2^{k-1})$, we have \begin{align} \label{eq2:local:packing:algorithm}
    w_\mu\left(\frac{d}{2^{k-1}}-  \frac{d}{c^{\ast}2^{k-2}}\right) \le w_\nu\left( \frac{d}{2^{k-1}}\right).
\end{align} Next, since $B(\nu',d/(c^{\ast}2^{k-1}))\subseteq B(\mu, d/(c^{\ast}2^{k-2}))$, \begin{align}\label{eq3:local:packing:algorithm}
    w_{\nu'}\left( \frac{d}{c^{\ast}2^{k-1}}\right)  \leq w_\mu\left( \frac{d}{c^{\ast}2^{k-2}}\right).
\end{align} Combining \eqref{eq1:local:packing:algorithm}, \eqref{eq2:local:packing:algorithm}, and \eqref{eq3:local:packing:algorithm}, we have 
\begin{align*}
   w_\mu\left(\frac{d}{2^{k-1}}-  \frac{d}{c^{\ast}2^{k-2}}\right)- w_\mu\left( \frac{d}{c^{\ast}2^{k-2}}\right)  \leq  \frac{C(d/2^{k-1})^2}{2\sigma}.
\end{align*}

Now define $\sigma' = \left(\frac{c^{\ast}-4}{Cc^{\ast}}\right)\cdot \sigma = \frac{c^{\ast}-4}{4c^{\ast} - 1/c^{\ast}}\cdot \sigma$. Rearranging, we obtain that \begin{align*}
   \MoveEqLeft w_\mu\left(\frac{d}{2^{k-1}}-  \frac{d}{c^{\ast}2^{k-2}}\right) - \frac{(d/2^{k-1}-  d/(c^{\ast}2^{k-2}))^2}{2\sigma'} \\ &\leq  w_\mu\left( \frac{d}{c^{\ast}2^{k-2}}\right) - \frac{(d/(c^{\ast}2^{k-2}))^2}{2\sigma'}.
\end{align*}


Note that $d/2^{k-1}-  d/(c^{\ast}2^{k-2}) > d/(c^{\ast}2^{k-2})$ since $c^{\ast} >4$. Consequently, by our same concavity argument as before but with $\sigma'$, we conclude $ \epswidth(\sigma')\le d/2^{k-1}-  d/(c^{\ast}2^{k-2})$. Since $\mu\in K$ was chosen arbitrarily, taking the supremum we conclude \[\epsSupwidth(\sigma') \le  \frac{d}{2^{k-1}} -  \frac{d}{c^{\ast}2^{k-2}}= \frac{d(2c^{\ast}-4)}{c^{\ast}2^k}.\]

 Now, observe that $\sigma'=c\sigma$ where $c=\frac{c^{\ast}-4}{4c^{\ast} - 1/c^{\ast}}<1$. Hence by Lemma \ref{lemma:epsilon_mu:nondecreasing}, we have \begin{align*}
     \epsSupwidth(\sigma) &\le c^{-1}\epsSupwidth(c\sigma)=c^{-1}\epsSupwidth(\sigma')\\ &\le c^{-1}\cdot \left[\frac{d(2c^{\ast}-4)}{c^{\ast}2^k}\right] \\ &= \frac{(2c^{\ast}-4)(4c^{\ast}-1/c^{\ast})}{(c^{\ast}-4)c^{\ast}}\cdot \frac{d}{2^k}.
 \end{align*}
    \end{proof}

    \begin{proof}[\hypertarget{proof:theorem:local:packing:algorithm}{Proof of Theorem \ref{theorem:local:packing:algorithm}}]
        Suppose the algorithm terminates after $k$ iterations. By Lemma \ref{lemma:local:packing:algorithm}, this implies  $\epsSupwidth(\sigma)\ge d/(2^{k-1}c^{\ast})$. Consider two cases.
        
        First, suppose  $\sigma \lesssim d/(2^{k-1}c^{\ast})$. Then for some $\mu \in K$, we have $\epswidth(\sigma) \ge d/(2^{k-1}c^{\ast})\gtrsim\sigma$, so by Lemma \ref{lemma:chatterjee:analogue}, \[\EE\|\hat{\mu}-\mu\|_2^2\asymp\epswidth(\sigma)^2\gtrsim d^2/(2^{k-1}c^{\ast})^2.\] Taking the supremum, we conclude $\epsLSE^2\gtrsim d^2/(2^{k-1}c^{\ast})^2$, verifying (i) in this case. 
        
        On the other hand, suppose $\sigma\gtrsim d/(2^{k-1}c^{\ast})$. Recall that $\epsLSE(\sigma)\gtrsim \varepsilon^{\ast}(\sigma)\gtrsim\sigma\wedge d$ by Lemma \ref{minimax:bound:versus:sigma:and:d}, where $\varepsilon^{\ast}(\sigma)$ is the minimax rate. But $\sigma\wedge d \gtrsim d/(2^{k-1}c^{\ast})$. Hence $\epsLSE(\sigma)\gtrsim d/(2^{k-1}c^{\ast})$. This proves (i).
        
        We next prove (ii). Suppose the algorithm does not terminate within $k$ iterations. By (ii) of Lemma \ref{lemma:local:packing:algorithm}, we have $\epsSupwidth(\sigma)\le c'\cdot \frac{d}{2^k}$. Consider two scenarios.
        
        Consider the scenario (a) in (ii), where  $\sigma \ge c'\cdot \frac{d}{2^k}$.  Then for any $\mu\in K$, \[\epswidth(\sigma) \le c'\cdot \frac{d}{2^k} \le \sigma,\] which in turn implies by Lemma \ref{lemma:chatterjee:analogue} that for all $\mu\in K$ we have $\EE\|\hat{\mu}-\mu\|_2^2\lesssim {\sigma}^2.$ Hence $\epsLSE(\sigma)\lesssim \sigma$, and we always have $\epsLSE(\sigma)\lesssim d$. On the other hand, by Lemma \ref{minimax:bound:versus:sigma:and:d} we know $\epsLSE(\sigma)\gtrsim \varepsilon^{\ast}(\sigma)\gtrsim \sigma\wedge d$. Thus $\epsLSE(\sigma)\asymp \sigma\wedge d$, proving (a) in (ii).

        For scenario (b) in (ii), suppose  $\sigma \le c'\cdot \frac{d}{2^k}$. We claim  $\epsLSE(\sigma) \lesssim c'\cdot \frac{d}{2^k}$. Suppose not. Then for any absolute constant $C'>0$, we have $\epsLSE(\sigma)> C' \cdot c'\cdot \frac{d}{2^k}$, which implies  $\epsLSE(\sigma)\ge C'\sigma$.

        Take the $\mu$ that maximizes $\EE\|\hat\mu-\mu\|_2^2 $. We claim for this $\mu\in K$, $\epswidth(\sigma)\gtrsim \sigma$. If not, we would have $\epswidth(\sigma)\lesssim \sigma$. Then by Lemma \ref{lemma:chatterjee:analogue}, \[\epsLSE(\sigma)^2=\EE\|\hat\mu-\mu\|_2^2\lesssim \sigma^2.\] But this contradicts $\epsLSE(\sigma)\ge C'\sigma$ for a sufficiently large choice of $C'$. Hence $\epswidth(\sigma)\gtrsim \sigma$ as we claimed. Then by Lemma \ref{lemma:chatterjee:analogue} and (ii) of Lemma \ref{lemma:local:packing:algorithm}, \begin{align*}
            \epsLSE(\sigma)^2 &=\EE\|\hat\mu-\mu\|_2^2\asymp \epswidth(\sigma)^2 \le \epsSupwidth(\sigma)^2 \le \left(\frac{c'd}{2^k}\right)^2.
        \end{align*} But this contradicts our assumption from before  that $\epsLSE(\sigma)\ge C' \cdot c'\cdot \frac{d}{2^k}$ for any $C'>0$. Thus,  $\epsLSE(\sigma) \lesssim c'\cdot \frac{d}{2^k}$, proving (b) in (ii).
        \end{proof}

Proceeding to the global packing algorithm, we begin with two lemmas before proving our main result (Theorem \ref{theorem:global:packing:algorithm}).

\begin{lemma} \label{lemma:technical:global:algorithm} Let $\varepsilon$ be the output of $k$ steps of the \textbf{while} loop of Algorithm \ref{algo_global_packing}, including $k=0$ (initialization). Let $c^{\ast}>1$ be given and set $c' = \frac{({c^{\ast}}^2-1)}{{c^{\ast}}^2(10+8/c^{\ast})}$. \renewcommand\labelenumi{(\theenumi)}
\renewcommand{\theenumi}{\roman{enumi}} \begin{enumerate}
    \item  If $\varepsilon$ satisfies $\Psi\ge T$, then $\varepsilon/c^{\ast} \le \epswidth[\nu_{i^*}](\sigma) \le \epsSupwidth(\sigma)$.
    \item If $\varepsilon$ satisfies $\Psi < T$, then  $c'\cdot \epsSupwidth(\sigma) \le \varepsilon$.
\end{enumerate} 
\end{lemma}
    \begin{proof}[\hypertarget{proof:lemma:technical:global:algorithm}{Proof of Lemma \ref{lemma:technical:global:algorithm}}]
        We first prove (i). Suppose our update (or initial choice) of $\varepsilon$ satisfies $\Psi \geq T$. Following the argument leading up to \eqref{difference:of:local:widths:compared:to:eps:squared:th:eq1}, $T\geq \frac{C\varepsilon^2}{4\sigma}$ since $\varepsilon\ge 2\underline\varepsilon^{\ast}$. Given the maximal packing $\{\nu_1,\dots,\nu_M\}$ at that step, for some $i^*\in[M]$ we have
\begin{align*}
    w_{\nu_{i^*}}(\tfrac{\varepsilon}{c^{\ast}} + 2\varepsilon) -  w_{\nu_{i^*}}(\tfrac{\varepsilon}{c^{\ast}}) &\geq \sup_{\nu' \in  B(\nu_{i^*}, 2\varepsilon-\delta)\cap K} w_{\nu'}(\tfrac{\varepsilon}{c^{\ast}}) - w_{\nu_{i^*}}(\tfrac{\varepsilon}{c^{\ast}}) \\ &= \Psi \geq T =  \frac{C\varepsilon^2}{4\sigma},
\end{align*} where the first inequality follows since 
\begin{align*}
    \bigcup_{\mathclap{\nu' \in  B(\nu_{i^*}, 2\varepsilon-\delta)\cap K}} B(\nu',\tfrac{\varepsilon}{c^{\ast}}) \cap K  &\subseteq B(\nu_{i^*}, \tfrac{\varepsilon}{c^{\ast}} + 2\varepsilon-\delta) \cap K \\ &\subseteq B(\nu_{i^*},\tfrac{\varepsilon}{c^{\ast}} + 2\varepsilon) \cap K. 
\end{align*}
 Rearranging and using $C=8+8/c^{\ast}$, we obtain \[w_{\nu_{i^*}}(\varepsilon/c^{\ast} + 2\varepsilon) - \frac{(\varepsilon/c^{\ast}+2\varepsilon)^2}{2\sigma} \ge w_{\nu_{i^*}}(\varepsilon/c^{\ast} ) - \frac{(\varepsilon/c^{\ast})^2}{2\sigma}\] which by concavity of $\varepsilon\mapsto w_{\nu_{i^*}}(\varepsilon) - \varepsilon^2/2\sigma$ implies that $\epsSupwidth(\sigma)\ge \epswidth[\nu_{i^*}](\sigma) \ge \varepsilon/c^{\ast}$. This proves (i).

 Next, we prove (ii). Suppose $\Psi < T$.  We will first show that \begin{equation}\label{global:packing:algorithm:Psi:ineq} \Psi \geq \sup_{\substack{\nu_1,\nu_2\in K:\\ \|\nu_1-\nu_2\|\le 2\varepsilon}} w_{\nu_1}(\varepsilon/c^{\ast}) - w_{\nu_2}(\varepsilon/c^{\ast}) - \frac{\varepsilon^2}{\sigma}.\end{equation} Let $\nu_1^*,\nu_2^*$ attain this supremum. The $\nu_2^*$ is at most distance $\delta$ away from one of the points $\nu_i$ in our maximal $\delta$-packing set. Using \eqref{remark:Lipschitz:result} from Remark \ref{remark:Lipschitz} and the definition of $\delta$, for some $i$,
\begin{align*}
    |w_{\nu_i}(\varepsilon/c^{\ast}) - w_{\nu_2^*}(\varepsilon/c^{\ast})| &\leq \frac{\delta c^{\ast}}{\varepsilon}\cdot \sup_{\eta \in K} w_\eta(\varepsilon/c^{\ast}) \\ &\leq \frac{\varepsilon^2}{4\sigma}\leq \frac{\varepsilon^2}{2\sigma}.
\end{align*}

Next, since $\|\nu_2^* - \nu_i\| \leq \delta$ and $\|\nu_1^* - \nu_2^*\| \leq 2\varepsilon$, we have that $\|\nu_1^* - \nu_i\| \leq 2\varepsilon + \delta$. Then setting $\alpha = \frac{2\delta}{2\varepsilon+\delta}\in (0,1)$ and taking $\zeta = \alpha \nu_i +  (1-\alpha) \nu_1^*$, we have $\zeta \in B(\nu_i, 2\varepsilon-\delta)\cap B(\nu_1^*, 2\delta)\cap K$. Again by \eqref{remark:Lipschitz:result}, 
\begin{align*}
    |w_{\zeta}(\varepsilon/c^{\ast}) - w_{\nu_1^*}(\varepsilon/c^{\ast})| \leq \frac{2\delta c^{\ast}}{\varepsilon}\cdot \sup_{\eta \in K} w_\eta(\varepsilon/c^{\ast}) \leq \frac{\varepsilon^2}{2\sigma}. 
\end{align*} Both results imply $w_{\zeta}(\varepsilon/c^{\ast}) - w_{\nu_1^*}(\varepsilon/c^{\ast})\ge -\frac{\varepsilon^2}{2\sigma}$ and  $w_{\nu_2^*}(\varepsilon/c^{\ast}) - w_{\nu_i}(\varepsilon/c^{\ast})\ge -\frac{\varepsilon^2}{2\sigma}$.

By the definition of $\Psi$ as a max and supremum followed by the preceding two lower bounds  of $-\varepsilon^2/2\sigma$, we have
\begin{align*}
    \Psi &\geq w_{\zeta}(\varepsilon/c^{\ast}) - w_{\nu_i}(\varepsilon/c^{\ast}) \\ &\geq w_{\nu^*_1}(\varepsilon/c^{\ast}) - w_{\nu^*_2}(\varepsilon/c^{\ast}) - \frac{\varepsilon^2}{\sigma}.
\end{align*} Hence $\Psi$ satisfies \eqref{global:packing:algorithm:Psi:ineq} as claimed.

Next, since $\Psi <T$, \eqref{global:packing:algorithm:Psi:ineq} implies \begin{align*}
   \MoveEqLeft \sup_{\substack{\nu_1,\nu_2\in K:\\ \|\nu_1-\nu_2\|\le 2\varepsilon}} w_{\nu_1}(\varepsilon/c^{\ast}) - w_{\nu_2}(\varepsilon/c^{\ast}) - \frac{\varepsilon^2}{\sigma} \\ &< \frac{C \varepsilon^2}{2\sigma} - \frac{L\varepsilon}{c^{\ast}} \sqrt{\log \cMloc(\varepsilon)}.
\end{align*} Pick any $\mu\in K$. Then we have  \begin{align*}
   \MoveEqLeft \sup_{\nu_1\in B(\mu,\varepsilon)\cap K}w_{\nu_1}(\varepsilon/c^{\ast}) - \inf_{\nu_2\in B(\mu,\varepsilon)\cap K}w_{\nu_2}(\varepsilon/c^{\ast}) \\ &\le \sup_{\substack{\nu_1,\nu_2\in K:\\ \|\nu_1-\nu_2\|\le 2\varepsilon}} w_{\nu_1}(\varepsilon/c^{\ast}) - w_{\nu_2}(\varepsilon/c^{\ast}).
\end{align*} Combining the previous two inequalities, \begin{align*}
    \MoveEqLeft\sup_{\nu_1\in B(\mu,\varepsilon)\cap K}w_{\nu_1}(\varepsilon/c^{\ast}) - \inf_{\nu_2\in B(\mu,\varepsilon)\cap K}w_{\nu_2}(\varepsilon/c^{\ast}) \\ &< \frac{(C+2)\varepsilon^2}{2\sigma} -  \frac{L\varepsilon}{c^{\ast}} \sqrt{\log \cMloc(\varepsilon)}.
\end{align*}
By \citet[Exercise 2.4.11]{talagrand2014upper},  \[ \sup_{\nu_1 \in B(\mu,\varepsilon) \cap K}  w_{\nu_1}(\varepsilon/c^{\ast}) +  \frac{L\varepsilon}{c^{\ast}} \sqrt{\log \cMloc(\varepsilon)} \geq w_{\mu}(\varepsilon).\] On the other hand \[\inf_{\nu_2 \in  B(\mu,\varepsilon) \cap K}  w_{\nu_2}(\varepsilon/c^{\ast})\leq w_{\mu}(\varepsilon/c^{\ast}),\] and combining the previous three bounds, we have \begin{align*}
    w_{\mu}(\varepsilon)- w_{\mu}(\varepsilon/c^{\ast})<  \frac{(C+2)\varepsilon^2}{2\sigma}.
\end{align*} Taking $\sigma' = \frac{({c^{\ast}}^2-1)}{{c^{\ast}}^2(C+2)}\cdot \sigma$, we have 
\begin{align*}
    w_{\mu}(\varepsilon)- \frac{\varepsilon^2}{2\sigma'} < w_{\mu}(\varepsilon/c^{\ast}) - \frac{(\varepsilon/c^{\ast})^2}{2\sigma'}.
\end{align*} By concavity again, $\epswidth(\sigma') \le \varepsilon$, and since this holds for all $\mu\in K$, we have $\epsSupwidth(\sigma') \le\varepsilon$.

We now observe that $\sigma' = c\sigma$ where  $c=\frac{({c^{\ast}}^2-1)}{{c^{\ast}}^2(C+2)}\in(0,1)$. Therefore, using Lemma \ref{lemma:epsilon_mu:nondecreasing}, \begin{align*}
    \epsSupwidth(\sigma) &\le c^{-1}\epsSupwidth(c\sigma)=c^{-1}\epsSupwidth(\sigma') \\ &\le c^{-1}\varepsilon = \frac{{c^{\ast}}^2(10+8/c^{\ast})}{({c^{\ast}}^2-1)}\cdot \varepsilon.
\end{align*} Rearranging and recalling our definition of $c'$ for this lemma, we have $c'\cdot \epsSupwidth(\sigma)\le \varepsilon$ and this holds for any $\varepsilon$ satisfying $\Psi< T$. 
    \end{proof}

\begin{lemma} \label{lemma:global:algorithm}  Let $c^{\ast}>1$ be given and set $c' = \frac{({c^{\ast}}^2-1)}{{c^{\ast}}^2(10+8/c^{\ast})}$. Then Algorithm \ref{algo_global_packing} satisfies the following: 
\renewcommand\labelenumi{(\theenumi)}
\renewcommand{\theenumi}{\roman{enumi}}  \begin{enumerate}
    \item Algorithm \ref{algo_global_packing} will terminate in finitely many steps. 
    \item If Algorithm \ref{algo_global_packing} does not immediately terminate, its output $\varepsilon$ satisfies $c'\cdot \epsSupwidth(\sigma)\le \varepsilon\le  2c^{\ast}\cdot\epsSupwidth(\sigma).$
    \item If  Algorithm \ref{algo_global_packing} terminates on initialization, i.e., returns $\varepsilon=2\underline\varepsilon^{\ast}$, then we have $c'\cdot \epsSupwidth(\sigma)\le \varepsilon \lesssim \epsLSE(\sigma).$
\end{enumerate} 
\end{lemma}
    \begin{proof}[\hypertarget{proof:lemma:global:algorithm}{Proof of Lemma \ref{lemma:global:algorithm}}]
   If $\Psi \geq T$ at some step of the algorithm (including at initialization), then (i) of Lemma \ref{lemma:technical:global:algorithm} states that  $ \varepsilon/c^{\ast} \le \epswidth[\nu_{i^*}](\sigma) \le \epsSupwidth(\sigma)$. On the other hand, if $\Psi < T$, then (ii) of Lemma \ref{lemma:technical:global:algorithm} states that $c'\cdot\epsSupwidth(\sigma) \le\varepsilon$.

    The first claim implies the algorithm must terminate (i.e., eventually $\Psi < T$), for suppose not. Then after sufficiently many iterations (each of which doubles $\varepsilon$), we have $\varepsilon > d c^{\ast}$. This implies $\epsSupwidth(\sigma)> d$, contradicting that $\epsSupwidth(\sigma)\le d$ from Lemma \ref{lemma:overline_epsilon_K_diameter}. Hence $\Psi< T$ after finitely many steps, proving (i).

We now prove (ii). Suppose the algorithm terminates after strictly more than one step. Then the procedure doubles $\varepsilon$ while it still satisfies a condition ($\Psi\ge T$) implying $\varepsilon/c^{\ast} \le \epsSupwidth(\sigma)$, until the final update, after which $\Psi <T$ and $\varepsilon\ge c'\cdot \epsSupwidth(\sigma)$. The final $\varepsilon$ therefore satisfies \[c'\cdot \epsSupwidth(\sigma)\le \varepsilon\le  2c^{\ast}\epsSupwidth(\sigma).\]

We finally prove (iii). Suppose the algorithm terminates immediately, i.e., at the initialization step, $\Psi < T$. Then we have \[c'\cdot \epsSupwidth(\sigma) \le \varepsilon = 2\underline\varepsilon^{\ast}\asymp \varepsilon^{\ast}\lesssim\epsLSE(\sigma).\] The first inequality is just (ii) of Lemma \ref{lemma:technical:global:algorithm}, and the last two inequalities (up to constants) follow from Lemma \ref{lemma:equivalent:information:lower:bound} and the fact that the minimax rate lower bounds  the LSE rate.
\end{proof}

    \begin{proof}[\hypertarget{proof:theorem:global:packing:algorithm}{Proof of Theorem \ref{theorem:global:packing:algorithm}}]  Assume the algorithm does not immediately terminate. We first prove (i). Suppose $d\lesssim \sigma$. Then $d\gtrsim\varepsilon^{\ast}\gtrsim \sigma\wedge d \asymp d$ by Lemma \ref{minimax:bound:versus:sigma:and:d}. Thus the minimax rate satisfies $\varepsilon^{\ast}\asymp d$, which means  $\epsLSE\asymp d$ (as does any estimator). Moreover, since we initialize $\varepsilon$ to be the minimax rate up to constants and then increase it, we have $\varepsilon\gtrsim \varepsilon^{\ast}\asymp d$. This proves (i).

    Proceeding to (ii) and (iii), suppose $d\gtrsim \sigma$. Then $\varepsilon^{\ast}\gtrsim \sigma\wedge d \asymp \sigma$. But since $\varepsilon \gtrsim \varepsilon^{\ast}$ by construction, we have $\varepsilon \gtrsim \sigma$. This means $\epsSupwidth(\sigma)\gtrsim \varepsilon \gtrsim \sigma$ by (ii) of Lemma \ref{lemma:global:algorithm}. Then for some $\mu\in K$, $\epswidth(\sigma)\gtrsim \varepsilon\gtrsim\sigma$. Lemma \ref{lemma:chatterjee:analogue} implies $\EE\|\hat\mu-\mu\|_2^2 \asymp \epswidth^2 \gtrsim \varepsilon^2.$ Thus, $\epsLSE(\sigma)\gtrsim \varepsilon$.

We then consider two subcases regarding $\sigma$. First suppose  $\sigma \ge\varepsilon$, i.e., scenario (ii). Then \[c'\cdot \epsSupwidth(\sigma)\le \varepsilon\le \sigma\] by (ii) of Lemma \ref{lemma:global:algorithm}, where $c' = \frac{({c^{\ast}}^2-1)}{{c^{\ast}}^2(10+8/c^{\ast})}$.  Then for any $\mu\in K$, $c'\epswidth(\sigma) \le \varepsilon \le \sigma$, which in turn implies by Lemma \ref{lemma:chatterjee:analogue} that for all $\mu\in K$ we have $\EE\|\hat{\mu}-\mu\|_2^2\lesssim {\sigma}^2.$ Hence $\epsLSE(\sigma)\lesssim \sigma$. But since $\epsLSE(\sigma)\lesssim d$ always holds, and we know from Lemma \ref{minimax:bound:versus:sigma:and:d} that $\epsLSE(\sigma)\gtrsim \varepsilon^{\ast}\gtrsim \sigma\wedge d$, we have $\epsLSE(\sigma)\asymp \sigma\wedge d$. But we also have $\varepsilon\gtrsim\varepsilon^{\ast}\gtrsim \sigma\wedge d\asymp \epsLSE(\sigma)$, again using Lemma \ref{minimax:bound:versus:sigma:and:d} and our argument from scenario (i) that $\varepsilon\gtrsim \varepsilon^{\ast}$. Thus, $\epsLSE(\sigma)\gtrsim\varepsilon\gtrsim \epsLSE(\sigma)\asymp \sigma\wedge d$, proving (ii).

        Next, suppose $\sigma \le \varepsilon$, i.e., scenario (iii). We claim  $\epsLSE(\sigma) \lesssim \varepsilon$. Suppose to the contrary that $\epsLSE(\sigma)\ge C' \cdot \varepsilon$ for any absolute constant $C'>0$, which implies  $\epsLSE(\sigma)\ge C'\sigma$ for all such $C'$.

        Take the $\mu\in K$ that maximizes $\EE\|\hat\mu-\mu\|_2^2 $. We claim for this $\mu\in K$, $\epswidth(\sigma)\gtrsim \sigma$. Suppose to the contrary $\epswidth(\sigma)\lesssim \sigma$. Then by Lemma \ref{lemma:chatterjee:analogue}, $\epsLSE(\sigma)^2=\EE\|\hat\mu-\mu\|_2^2\lesssim \sigma^2.$ But this will contradict $\epsLSE(\sigma)\ge C'\sigma$ (so long as $C'$ is picked sufficiently large). Hence $\epswidth(\sigma)\gtrsim \sigma$ as we claimed. Then by Lemma \ref{lemma:chatterjee:analogue} and Lemma \ref{lemma:global:algorithm}, \[\epsLSE(\sigma)^2=\EE\|\hat\mu-\mu\|_2^2\asymp \epswidth(\sigma)^2 \le  \epsSupwidth(\sigma)^2 \lesssim \varepsilon^2.\] But this contradicts our assumption from before  that $\epsLSE(\sigma)\ge C' \cdot \varepsilon$ for any $C'$. Thus,  $\epsLSE(\sigma) \lesssim \varepsilon$. Then recall that for both (ii) and (iii), we showed $\varepsilon\lesssim \epsLSE(\sigma)$. Hence we have $\varepsilon\asymp \epsLSE(\sigma)$, proving (iii).

        For our final claim, we suppose Algorithm \ref{algo_global_packing} terminates upon initialization. Recall  that \begin{equation} \label{eq:global:algo:theorem:1}
            c'\cdot \epsSupwidth(\sigma)\le \varepsilon \asymp \varepsilon^{\ast}\lesssim \epsLSE(\sigma)
        \end{equation} from  (iii) of Lemma \ref{lemma:global:algorithm}. Note that $\varepsilon$ is just the minimax rate. Suppose $\epsLSE(\sigma)\gtrsim \sigma$ for a sufficiently large constant. Now take $\mu$ that maximizes $\EE\|\hat\mu-\mu\|_2^2$. Well $\EE\|\hat\mu-\mu\|_2^2\gtrsim \sigma^2$. We claim $\epswidth(\sigma)^2\asymp \EE\|\hat\mu-\mu\|_2^2$. If $\epswidth(\sigma)\gtrsim \sigma$, this immediately follows from Lemma \ref{lemma:chatterjee:analogue}. If $\epswidth(\sigma)\lesssim \sigma$, then by Lemma \ref{lemma:chatterjee:analogue}  there exists a universal constant $C'>0$ such that $\EE\|\hat\mu-\mu\|_2^2\le C'\sigma^2$. But we assumed  $\epsLSE(\sigma)\gtrsim \sigma$ for a sufficiently large constant, i.e., larger than $C'$, leading to a contradiction. So indeed $\epswidth(\sigma)\gtrsim \sigma$, and using \eqref{eq:global:algo:theorem:1},
        \begin{align*}
            \epsLSE(\sigma)^2 &= \EE\|\hat\mu-\mu\|_2^2 \asymp\epswidth(\sigma)^2 \\ &\le \epsSupwidth(\sigma)^2\lesssim \varepsilon^2 \lesssim \epsLSE(\sigma)^2. 
        \end{align*}This proves $ \varepsilon\asymp \epsLSE(\sigma)$.

        Suppose instead $\epsLSE(\sigma)\lesssim \sigma$. We always have $\epsLSE(\sigma)\lesssim  d$ and  $\epsLSE(\sigma)\gtrsim \varepsilon^{\ast}\gtrsim \sigma\wedge d$ from Lemma \ref{minimax:bound:versus:sigma:and:d}. So $\epsLSE(\sigma)\asymp \sigma\wedge d$. But we also have $\varepsilon\asymp\varepsilon^{\ast}\gtrsim \sigma\wedge d$ and $\varepsilon\asymp\varepsilon^{\ast}\lesssim \epsLSE(\sigma)\asymp \sigma\wedge d$ since the minimax rate is a lower bound on the LSE, proving that $\varepsilon\asymp \sigma\wedge d \asymp \epsLSE(\sigma)$ as claimed.
    \end{proof}

\end{document}